\documentclass[hidelinks,onefignum,onetabnum]{siamart220329}

\usepackage{lipsum}
\usepackage{amsfonts}
\usepackage{graphicx}
\usepackage{epstopdf}
\usepackage{algorithmic}
\usepackage{xcolor}
\usepackage{amsmath}
\usepackage{amssymb}
\usepackage{mathtools}
\usepackage{subcaption}
\usepackage{hyperref}
\hypersetup{colorlinks=true, linkcolor=black}
\usepackage{mathrsfs}

\usepackage{booktabs}   %
\usepackage{multirow} 

\usepackage{tikz,pgfplots}
\usetikzlibrary{fit}
\usepackage{color,graphicx,fancybox,fancyvrb}

\definecolor{darkgreen}{RGB}{0, 100, 0}
\definecolor{forestgreen}{RGB}{34, 139, 34}
\definecolor{ccolor}{RGB}{203,96,21}

\definecolor{matlab1}{RGB}{0  113.9850  188.9550}
\definecolor{matlab2}{RGB}{ 216.7500   82.8750   24.9900}
\definecolor{matlab3}{RGB}{  236.8950  176.9700   31.8750}
\definecolor{matlab4}{RGB}{  125.9700   46.9200  141.7800}

\definecolor{matlab5}{RGB}{    118.8300  171.8700   47.9400}

\newcommand{\R}{\mathbb{R}} 
 
\newcommand{\C}{\mathbb{C}} 
\newcommand{\N}{\mathbb{N}}

\newcommand{\0}{\mathbf{0}}
\newcommand{\1}{\mathbf{1}}
\newcommand{\psd}{\mathbb{S}}

\newcommand{\es}{\mathcal{E}}

\newcommand{\tlA}{\tilde{A}}
\newcommand{\tlB}{\tilde{B}}
\newcommand{\tlC}{\tilde{C}}
\newcommand{\tlD}{\tilde{D}}

\DeclarePairedDelimiter{\abs}{\lvert}{\rvert}
\DeclarePairedDelimiter{\norm}{\lVert}{\rVert}

\DeclarePairedDelimiter{\rank}{\textrm{rank}(}{)}

\DeclarePairedDelimiterX{\inp}[2]{\langle}{\rangle}{#1, #2}

\DeclareMathOperator*{\argmin}{arg\!\,min}

\usepackage{mathdots}

\definecolor{cn}{RGB}{93,147,191}           %
\definecolor{cz}{RGB}{233,  72, 73}      %
\definecolor{cp}{RGB}{113, 191, 110}    %
\definecolor{cf}{RGB}{157,0,255}    %
\newcommand{\cn}[1]{\textcolor{cn}{#1}}
\newcommand{\cz}[1]{\textcolor{cz}{#1}}

\newcommand{\cf}[1]{\textcolor{cf}{#1}}

\newtheorem{prob}{Problem}
\newtheorem{assum}{Assumption}
\newtheorem{example}{Example}[section]

\newcommand{\bq}{\mathbf{q}}

\newcommand{\ds}{d^*}

\newcommand{\bz}{\mathbf{z}} 

\newcommand{\bm}{\mathbf{m}} 
 
\newcommand{\bsig}{\boldsymbol{\sigma}} 
\newcommand{\blam}{\cz{\boldsymbol{\lambda}}} 
\newcommand{\Blam}{\boldsymbol{\Lambda}}

\newcommand{\nulls}{\textrm{null}}

\newcommand{\Aco}{\cn{A_{\text{core}}}}
\newcommand{\Bco}{\cn{B_{\text{core}}}}
\newcommand{\Cco}{\cn{C_{\text{core}}}}
\newcommand{\Dco}{\cn{D_{\text{core}}}}

\newcommand{\Ac}{\green{A_c}}
\newcommand{\Bc}{\green{B_c}}
\newcommand{\Cc}{\green{C_c}}
\newcommand{\Dc}{\green{D_c}}
\newcommand{\Kc}{\green{K}}

\newcommand{\Sco}{\cn{\Sigma_{\text{core}}}}
\newcommand{\linfo}{\mathcal{L}_{\mathrm{info}}}

\newcommand{\Acf}{\cf{A_{f}}}
\newcommand{\Bcf}{\cf{B_{f}}}
\newcommand{\Ccf}{\cf{C_{f}}}
\newcommand{\Dcf}{\cf{D_{f}}}

\newcommand{\tto}{\rightrightarrows}
\newcommand{\pcc}{p.c.c.}

\usepackage{tikz,pgfplots,datatool}
\usepackage[tikz]{bclogo}
\newcommand{\skipthis}[1]{}

\usepackage{arydshln}

\newcommand{\hdl}{\\ \hdashline}
\newcommand{\enu}[1]{ \begin{enumerate} #1 \end{enumerate} }

\newcommand{\mas}[2]{\left[\begin{array}{#1}#2\end{array}\right]}
\newcommand{\mat}[2]{\left(\begin{array}{#1}#2\end{array}\right)}

\newcommand{\green}[1]{{\color[rgb]{0,0.7,0}#1}}

\definecolor{yel}{rgb}{1,.9,.6}
\definecolor{lmag}{rgb}{0,0.82,1}

\definecolor{ggreen}{rgb}{0,0.7,0}
\definecolor{mag}{rgb}{0,0.7,.9}
\definecolor{magenta}{rgb}{1,0,1}
\definecolor{lgrey}{rgb}{0.85,0.85,0.85}
\definecolor{lgr}{rgb}{.9,0.9,0.9}
\definecolor{lgreen}{rgb}{0,0.92,0.7}
\colorlet{lred}{red!20!white}
\definecolor{lblue}{rgb}{1,0.95,0.7}

\definecolor{sblue}{rgb}{.1,.2,.5}
\colorlet{lsblue}{sblue!20!white}
\definecolor{sred}{rgb}{.85,.25,0}
\colorlet{lsred}{sred!20!white}
\definecolor{lye}{rgb}{1,0.95,0.7}

\usepgflibrary{fpu}
\usetikzlibrary{positioning,calc,arrows,shapes,shadows,matrix,backgrounds,external}
\tikzsetexternalprefix{figures/}

\tikzset{
auto,
poi/.style={
minimum size=0,
inner sep=0
},
sys/.style 2 args={
rectangle,
draw,
rounded corners,
drop shadow,
draw=black,
top color=black!20,bottom color=black!0,
minimum height=#2,
minimum width=#1,
inner sep=\dn},
syse/.style 2 args={
rectangle,
draw=none,
rounded corners,
minimum height=#2,
minimum width=#1,
inner sep=\dn},
nod/.style={
circle,
draw,
fill=white,
minimum size=5ex
},
sum/.style={circle,draw,draw=black,inner sep=0mm,minimum size=2mm,drop shadow,fill=white,
draw=black!100,top color=black!20,bottom color=black!0},
sume/.style={circle,draw=none,inner sep=0mm,minimum size=2mm,drop shadow,fill=white,
draw=black!100,top color=black!20,bottom color=black!0},
jun/.style={circle,draw,draw=black,inner sep=0mm,minimum size=0mm},
>={latex},
every path/.style={rounded corners},
lin/.style={color=black,draw,->},
cbox/.style n args={3}{
rectangle,
draw,
minimum width=#1,
minimum height=#2,
rounded corners=#3,
},
cobox/.style 2 args={
rectangle,
rounded corners,
draw=#1,
fill=#2,
inner sep=\dn},
}

\newcommand{\tio}[4]{\coordinate (#1) at ($(#2.south #3)!#4!(#2.north #3)$)}

\def\dn{1ex}
\def\dl{3*\dn}

\tikzstyle{sy0}=[sys={0*\dn}{0*\dn}]
\tikzstyle{sy1}=[sys={12*\dn}{8*\dn}]
\tikzstyle{sy2}=[sys={8*\dn}{6*\dn}]
\tikzstyle{sy3}=[sys={5*\dn}{5*\dn}]
\tikzstyle{sy0}=[sys={0*\dn}{0*\dn}]
\tikzstyle{sye1}=[syse={12*\dn}{8*\dn}]
\tikzstyle{sye2}=[syse={8*\dn}{6*\dn}]
\tikzstyle{sye3}=[syse={5*\dn}{5*\dn}]

\tikzstyle{sy4}=[sys={12*\dn}{6*\dn}]
\tikzstyle{sye4}=[syse={12*\dn}{6*\dn}]

\newcommand{\cl}{\prec}

\newcommand{\col}{\text{col}}

\renewcommand{\c}[1]{ {\cal #1} }

\newcommand{\diag}{ \operatornamewithlimits{diag} }

\renewcommand{\t}[1]{ \tilde{#1} }

\newcommand{\la}{\lambda}
\newcommand{\La}{\Lambda}

\newcommand{\te}[1]{\text{\ \ #1\ \ }}

\newcommand{\al}{\alpha}
\newcommand{\si}{\sigma}

\newcommand{\Acl}{{\cal A}}
\newcommand{\Bcl}{{\cal B}}
\newcommand{\Ccl}{{\cal C}}
\newcommand{\Dcl}{{\cal D}}

\newcommand{\hl}{\\\hline}

\newcommand{\eql}[2]{\begin{equation}\label{#1}#2\end{equation}}

\renewcommand{\r}[1]{(\ref{#1})}

\ifpdf
  \DeclareGraphicsExtensions{.eps,.pdf,.png,.jpg}
\else
  \DeclareGraphicsExtensions{.eps}
\fi

\newsiamremark{remark}{Remark}
\newsiamremark{hypothesis}{Hypothesis}
\crefname{hypothesis}{Hypothesis}{Hypotheses}
\newsiamthm{claim}{Claim}

\headers{Structure, Analysis, and Synthesis of Algorithms}{J. Miller \textit{et. al.}}

\title{
Structure, Analysis, and Synthesis of First-Order Algorithms
\thanks{Submitted to the editors DATE.
\funding{ J. Miller and C. Scherer are funded by Deutsche Forschungsgemeinschaft (DFG, German Research Foundation) under Germany's Excellence Strategy - EXC 2075 – 390740016. We acknowledge the support by the Stuttgart Center for Simulation Science (SimTech).
F. Jakob acknowledges the support of the International Max Planck Research School for Intelligent Systems (IMPRS-IS).
}}}

\author{Jared Miller, Carsten Scherer, Fabian Jakob, Andrea Iannelli
\thanks{ J. Miller and C. Scherer are with the Chair of Mathematical Systems Theory, Department of Mathematics,  University of Stuttgart, Stuttgart, Germany 
  (\email{jared.miller@imng.uni-stuttgart.de}, \email{carsten.scherer@imng.uni-stuttgart.de}), F. Jakob and A. Iannelli are with the Institute for Systems Theory and Automatic Control, University of Stuttgart, Germany
(\email{fabian.jakob@ist.uni-stuttgart.de}, \email{ andrea.iannelli@ist.uni-stuttgart.de}).}}

\usepackage{amsopn}

\ifpdf
\hypersetup{
  pdftitle={Structure, Analysis, and Synthesis of Algorithms},
  pdfauthor={J. Miller,  C. Scherer, F. Jakob, A. Iannelli}
}
\fi

\begin{document}

\makeatletter
\AtBeginDocument{
\def\refstepcounter@optarg[#1]#2{%
  \cref@old@refstepcounter{#2}%
  \cref@constructprefix{#2}{\cref@result}%
  \@ifundefined{cref@#1@alias}%
    {\def\@tempa{#1}}%
    {\def\@tempa{\csname cref@#1@alias\endcsname}}%
  \protected@edef\cref@currentlabel{%
    [\@tempa][\arabic{#2}][\cref@result]%
    \csname p@#2\endcsname\csname the#2\endcsname}}}      
\makeatother

\maketitle

\begin{abstract}
\label{sec:abstract}
First-order optimization algorithms can be interpreted through the lens of dynamical systems as the interconnection of linear systems and a set of subgradient nonlinearities. This dynamical systems formulation allows for the analysis and synthesis of optimization algorithms by solving robust control problems. In this work, we use the celebrated internal model principle in control theory to structurally factorize convergent composite optimization algorithms into suitable  internal models and core subcontrollers. As the key benefit, we reveal that this permits us to synthesize optimization algorithms even if information is transmitted over networks featuring dynamical phenomena such as time delays, channel memory, or crosstalk. Design of certified linearly-convergent algorithms  is achieved under bisection either through a nonconvex local search or by alternation of convex semidefinite programs. We demonstrate factorization of existing optimization algorithms and the automated synthesis of new optimization algorithms in the networked setting.

\end{abstract}

\maketitle

\section{Introduction}
\label{sec:introduction}

Composite optimization problems minimize the sum of a set of functions $(f_i)_{i=1}^s$, as in $\beta^* \in \text{argmin}_\beta \sum_{i=1}^s f_i(\beta)$. These composite problems arise in  applications such as regression \cite{tibshirani1996regression}, image denoising \cite{chambolle2011first}, and  matrix completion \cite{candes2010matrix}.
First-order optimization algorithm are procedures that find an optimal point $\beta^*$ of the composite optimization problem by using gradient/subdifferential evaluations of $\partial f_i$, without needing higher order information  such as Hessians $\nabla^2 f_i$ \cite{wright1999numerical}.  First-order optimization algorithms can be understood through a system-theoretic lens \cite{wang2011control} as the interconnection between memoryless nonlinearities $\partial f_i$  and a linear system  \cite{lur1944theory, meyer1965liapunov}.

The problem of optimization algorithm design can be posed as choosing a linear time-invariant  system  to connect to the subgradient nonlinearity. The performance of an optimization algorithm can be judged by criteria such as its worst-case convergence rate to the optima $\beta^*$ (e.g. exponential convergence). 

Composite optimization problems may also be solved over networks featuring dynamics such as time delays, channels with memory, or crosstalk \cite{nedic2018network,
doostmohammadian2025survey}. Such network dynamics could cause algorithms that are nominally convergent to
diverge, or to converge to a non-optimal point. The difficulty of  optimization algorithm design  is is increased in the presence of these network dynamics.

When the given composite optimization problem is to minimize the sum of quadratics, the interconnection of $\partial f_i$ and the linear system forms an overall closed-loop linear system. The optimal iterate $\beta^*$ and its subgradients can be treated as a constant disturbance disturbing the closed-loop linear system. Algorithm convergence in the quadratic setting may then be interpreted as a disturbance rejection property of linear systems. This disturbance rejection problem for linear systems has been completely solved using regulation theory  \cite{francis1976internal, francis1977linear}: linear systems that reject disturbances must be stable, and must satisfy an affine system of equations among their state space matrices (Regulator Equation).

We will show that regulation theory also gives insights into structural properties of optimization algorithms. The search over linear systems when designing algorithms can therefore be restricted to linear systems that satisfy the relation, thus reducing the degrees of freedom in the design procedure.
Using an internal model principle \cite{francis1976internal, wonham2007towards}, we can factorize the linear system part of convergent optimization algorithms over networks into  the cascade of a core subcontroller  containing the algorithm parameters, and an internal model that depends only on the network. We show that the Regulator Equation requirement condition is independent of the specific quadratic composite optimization problem. We provide a theory of Convergence and Structure for well-posed optimization algorithms to solve composite optimization problems with unique optimal solutions.

The internal model factorization allows us to synthesize optimization algorithms over networks by solving structured robust control problems. These robust control problems arise by abstracting the subgradient nonlinearity into an uncertainty, and describing the uncertainty using the frameworks of dissipation inequalities \cite{willems1972dissipativeI, willems1972dissipativeII}.
This principled search over controllers also certifies well-posedness and exponential convergence rates of the resultant algorithm.

\textbf{Literature Review. }
The application of regulation theory and the  internal model principle towards understanding optimization algorithms is a continuation of existing unifications between control theoretic tools and optimization \cite{wang2011control, choppella2021algodynamics}.

The IQC-based the analysis of optimization algorithms originated in  \cite{lessard2016analysis} to bound the asymptotic exponential convergence rate of optimization algorithms. Input-output sequences of an oracle satisfying properties such as such as (strong) monotonicity, cocoercivity, and Lipschitzness satisfy  a family of dissipation relations/Integral Quadratic Constraints (IQCs) \cite{megretski2002system}. When the oracle is the subdifferential of a convex function, the subgradient inequality induces an infinite set of dynamic O`Shea-Zames-Falb multipliers \cite{zames1968stability}, each of which expresses a dissipation relation that is valid for all compatible oracles. 
IQC theory was used to derive the Triple Momentum scheme \cite{van2017fastest}, which was proven to be optimal among all first-order and fixed-stepsize algorithms for strongly convex optimization in \cite{taylor2023optimal}.
The IQC-analysis method in \cite{lessard2016analysis, lessard2022analysis} designed optimization algorithms based on a hyperparameter search: choose algorithm parameters such that the IQC-derived programs yield a minimal convergence rate.

In the case of a single oracle (e.g. unconstrained optimization), the single-input-single-output structure of the controller allows for convex synthesis of the algorithm and the filter coefficients \cite{holicki2021algorithm, scherer2021convex, scherer2023optimization}.
Extensions of IQC analysis and design beyond static optimization algorithms include smooth games \cite{zhang2021unified}, parameter-varying optimization \cite{jakob2025linear}, switched optimization algorithms \cite{miller2025analysis}, saddle-point algorighms \cite{gramlich2022synthesis} and online optimization \cite{jakob2025online}.

The work in \cite{wu2024tannenbaum, zhang2024frequency, wu2025frequency} use frequency-domain methods to synthesize optimization algorithms in the absence of network dynamics. The Consensus and Optimality properties are interpreted as tangential  interpolation constraints of transfer function matrices \cite{ball2013interpolation}. 
Nevalinna-Pick interpolation theory \cite{pick1915beschrankungen} is then used to generate algorithms that satisfy the Consensus and Optimality constraints. This frequency-domain  framework is extended to equality-constrained algorithms in \cite{ozaslan2025automated} and distributed optimization algorithms in \cite{zhang2024frequency}.

The Performance Estimation Problem (PEP) approach is an alternative method for the analysis and synthesis of optimization algorithms \cite{drori2014performance}. The PEP approach is based on finding conditions for which a convex function interpolates a set of given iterates, function values, and subdifferentials.
Interpolation constraints can be used in asymptotic \cite{taylor2018lyapunov, van2023automated, upadhyaya2025automated} and finite-horizon \cite{taylor2017smooth} analysis of optimization algorithms. Increasing the number of points considered in the interpolation problem reduces conservatism of the developed analysis.
Optimal variable-step and fixed-step rules for unconstrained gradient descent of smooth convex functions were synthesized using these interpolation constraints in \cite{taylor2023optimal}.

\textbf{Contributions and Paper Structure. }
The contributions of this work are:
\begin{itemize}
    \item Section \ref{sec:convergence}: A two-part condition for convergence of optimization algorithms (Robust Stability, Regulator Equation). 
    \item Section \ref{sec:structure}: A structural state-space factorization of all convergent well-posed first-order optimization algorithms based on the internal model. 
    \item Section \ref{sec:synthesis}: An analysis and  synthesis framework for certifying and generating optimization algorithms over networks based on linear matrix inequalities. 
    \item Section \ref{sec:examples}: Demonstration of algorithm synthesis in the dynamical   setting.
\end{itemize}

Table \ref{tab:main_results} lists the main theoretical contributions of the paper.

\vspace{-3mm}
\begin{table}[h]
    \centering
    \caption{Main results of the paper}
    \vspace{-2mm}
    \begin{tabular}{ll}
        \toprule
         Lemma \ref{lem:well_posed} &  Sufficient condition  for well-posedness of  algorithms \\
         Theorem \ref{thm:convergence} & Two-part condition for convergence of  algorithms \\
         Theorem \ref{thm:main_structure} & Internal-Model-Principle-based factorization of algorithms \\
         Theorem \ref{thm:analysis} & LMI-based convergence rate analysis of well-posed algorithms \\
         Prop. \ref{prop:converge_full} & LMI-based synthesis of certified  convergent algorithms\\
         \bottomrule
    \end{tabular}
    \label{tab:main_results}
\end{table}

Section \ref{sec:preliminaries} reviews preliminaries of notation, linear systems, convex analysis, and optimization algorithms. Sections \ref{sec:convergence}, \ref{sec:structure}, \ref{sec:synthesis}, and \ref{sec:examples} give the main results as described in the above table. Section \ref{sec:conclusion} concludes the paper.

\section{Preliminaries}
\label{sec:preliminaries}

We first review preliminaries of notation, linear systems, and optimization algorithms.

\subsection{Notation and Linear Algebra}

The set of natural numbers including zero is $\N$. The set of integers between $a$ and $b$ inclusive is $\{a, \ldots, b\}.$
The extended real line is $\bar{\R} = \R \cup \{-\infty, \infty\}$. The $n$-dimensional real Euclidean space  is $\R^n$. 
Its positive orthant is $\R^n_{>}:=\{(x_1,\ldots,x_n)\in\R^n\mid x_i>0\ \forall i\in\{1,\ldots,n\}\}$.

The set of $n$-dimensional symmetric matrices is $\psd^n$. An empty matrix will be denoted as $[\cdot]$, and an empty set as $\varnothing$.
If $M$ is square, the
symmetric part of $M \in \R^{n \times n}$ is $\text{Sym}(M) = \frac{1}{2} (M + M^\top)$.
The spectral radius of a matrix $M$ is $\rho(M)$.
A square matrix $M$ is Schur if $\rho(M)<1$.
If $M$ is symmetric, its  minimal and maximal eigenvalues are
$\lambda_{\min}(M)$ and\ $\lambda_{\max}(M)$, respectively.

The all zeros matrix is $0_{m \times n}$, the all ones vector is $\1_s$, and the identity matrix is $I_s$. 
The diagonal matrix defined by a vector $v$ is $\mathrm{diag}(v)$.
The block diagonal concatenation of a set of matrices is  $\text{blkdiag}(\cdot, \ldots, \cdot)$.

\subsection{Linear Systems}\label{Ssys}
A discrete time signal $x: \N \rightarrow \R^n$ is a time-indexed sequence $(x_k)_{k \in \N}.$
The unit shift operator $\bz$ acts on a sequence $x$ as $\bz x  = \bz (x_{k})_{k \in \N} = (x_{k+1})_{k \in \N}$ and $(\bz x)_{0} = 0$.
A linear system $G$ is a  mapping from an input signal $u: \N \rightarrow \R^{n_u}$ and an initial state $x_0 \in \R^n$ to an output signal $y: \N \rightarrow \R^{n_y}$ described with matrices $(\Acl, \Bcl, \Ccl, \Dcl)$ in the state-space as
\begin{align}
\label{eq:linear_system}
   & & G:  \mat{c}{x_{k+1} \hl y_k} =
    \mat{c|c}{\Acl & \Bcl \hl \Ccl &\Dcl} \mat{c}{x_k \hl u_k}, \; \text{ abbreviated as } \; y & = \mas{c|c}{        \Acl & \Bcl \hl \Ccl &\Dcl
     } u.
\end{align}
{The  transfer matrix of the system $G$ is 
\mbox{$T(\bz) = \Ccl (\bz I - \Acl)^{-1} \Bcl + \Dcl$} for $\bz\in\C$. If a given rational function $T(\bz)$
is represented as $T(\bz) = \Ccl' (sI - \Acl')^{-1} \Bcl' + \Dcl'$, the tuple $(\Acl', \Bcl', \Ccl', \Dcl')$
is a realization of $T(\bz)$. Realizations are generally nonunique.}

The cascade of  linear systems $G_1$ and $G_2$  cascade is $G_2 G_1$, and their block diagonal is $\text{blkdiag}(G_1,  G_2)$. 
The well-posed feedback interconnection between systems $G_1: (x_0^1, (w, u)) \rightarrow (z, y)$ and $G_2: (x_0^2, (y, a)) \rightarrow (u, b)$ along their shared signals $(u, y)$  is the \textit{star product} $G_1 \star G_2: ((x_0^1, x_0^2), (w, a)) \rightarrow (z, b)$. 
Well-posedness ensures that the shared signals $(u, y)$ are unique given $(x^1, x^2, w, a)$, and this is {true iff} $I - \Dcl_2 \Dcl_1$ is invertible \cite{zhou1998essentials}.

{
The representation $(\Acl, \Bcl, \Ccl, \Dcl)$ of $G$ is stable if $\Acl$ is Schur.
This holds iff all trajectories of \eqref{eq:linear_system} for $u=0$ and any $x_0\in\R^n$
satisfy $\lim_{k \rightarrow \infty} x_k = 0$.
We will refer to controllability (stabilizability) and observability (detectability) of the realization \eqref{eq:linear_system} if $(\Acl,\Bcl)$ is controllable (stabilizable) or $(\Acl,\Ccl)$ is observable (detectable) according to the standard definition \cite{zhou1998essentials}.
A {realization} $(\Acl, \Bcl, \Ccl, \Dcl)$ is minimal
{if it is} controllable and observable.

{
Let us finally recall that a controlled interconnection is described by
\begin{align}
\label{GP}
    P:& & \qquad  \mat{c}{
    x_{k+1}  \hl z_k \\ y_k
    } & = \mat{c|cc}{ A & B_1 & B_2 \hl
    C_1 & D_1 & D_{12} \\
    C_2 & D_{21} & D_2 }\mat{c}{
        x_k \hl w_k \\ u_k
    }, \\
        \Kc: & & \qquad  \mat{c}{
        \xi_{k+1}  \hl u_k} & =\mat{c|c}{\Ac & \Bc \hl
    \Cc & \Dc}\mat{c}{
        \xi_k \hl y_k
    }.
\end{align}
Here $P$ is the plant and $\Kc$ is a controller. They share the control input $u$ and the measurement output $y$.
This special star product $P \star K$ is well-posed if $E:=I - D_2 \Dc$ is invertible. Well-posedness permits elimination of the
common signals $(u, y),$ to arrive at the description
\vspace{-2mm}
\begin{equation}\arraycolsep.5ex
\mat{c}{x_{k+1}\\\xi_{k+1}\hl z_k}=
\mat{c|c}{\Acl&\Bcl\hl\Ccl&\Dcl}
\mat{c}{x_k\\\xi_k\hl w_k}, \label{cll}
\end{equation}
in which the above calligraphic matrices are defined according to 
\begin{equation}\label{eq:LFT}\arraycolsep.5ex
    \mat{c|c}{\Acl&\Bcl\hl\Ccl&\Dcl} \! := \! \mat{cc|cc}{
A  & 0 &   B_1  \\
0  & 0 &   0    \hl
C_1& 0 &  D_1   }
+
\mat{ccc}{
 0 & B_2    \\
 I & 0      \hl
 0 & D_{12} }
 \!
 \mat{ccc}{
\Ac   +\Bc E^{-1}D_2\Cc   &\Bc    E^{-1} \\
(I +\Dc E^{-1}D_2)\Cc     &\Dc    E^{-1} }
\!
\mat{cc|cc}{
0  & I &   0       \\
C_2& 0 &   D_{21}  }.
\end{equation}
We refer to the interconnection in \eqref{eq:LFT} as $P\star\Kc$, with the tacit understanding that it is well-posed and
described  with the matrices in \r{cll}. 
Moreover, $\Kc$ internally stabilizes $P$ if $P\star \Kc$ is well-posed and if $\Acl$ is Schur stable.
There exists a controller $\Kc$ which internally stabilizes $P$ iff $(A,B_2)$ is stabilizable and $(A,C_2)$ is detectable.
}

\subsection{Function Class}\label{Sconana}

A function that is proper, convex, and closed will be referred to as
{being} \pcc \
For $m \geq 0$ and $L > m$, we consider the class $\mathcal{S}_{m, L}$,
consisting of all $f: \R^c \rightarrow \bar{\R}$ {such that} $f - m \mathbf{q}$ and $L\bq -f$ {is} \pcc, where $\bq := \tfrac{1}{2} \| \cdot \|^2$.
If $L<\infty$, the class $\mathcal{S}_{m,L}$ contains all $m$-convex and $L$-smooth functions.
For $L=\infty$, this is} an extension of existing definitions in \cite{scherer2023robustozf, rotaru2024exact} to the case where $f$ is extended real-valued, and to  \cite{freeman2018noncausal, gramlich2022synthesis} when $f$ is nonsmooth.

The indicator function {of} a set $\mathcal{Z} \in \R^c$ {is defined as $I_{\mathcal{Z}}(x) = 0$ if $x \in \mathcal{Z}$ and $I_{\mathcal{Z}}(x) = \infty$ if $x \not\in \mathcal{Z}$}.
As an example, i
If $\mathcal{Z}\neq\varnothing$ is closed and convex, $I_{\mathcal{Z}} \in \mathcal{S}_{0, \infty}.$

\textbf{Subdifferentials.}
{The subdifferential of a \pcc\ function $f$ 
is defined as
\begin{equation}
    \partial f(x) = \{g \in \R^c \mid f(x') \geq f(x) + g^\top (x' - x)\ \forall x'\in\R^c\}. \label{eq:subdifferential}
\end{equation}

Given a function $f \in \c{S}_{m, L}$, the \pcc function $f - m \bq$ has a subdiffential $\partial (f - m \mathbf{q})$.

The subdifferential is extended to a function $f\in\c{S}_{m,L}$ with the possibility of $m<0$ by  
$\partial f := \partial (f - m \mathbf{q}) + m I$ (Frech\'et  subdifferential\cite[Definition 6.2, (iii)]{upadhyaya2025autolyap} using  \cite[Proposition 1.107 (i)]{mordukhovich2006variational}).
$\mathcal{S}_{m, L}^0$ is the subset of functions $f \in \mathcal{S}_{m, L}$ which are zero-centered as $0=f(0)$ and $0 \in \partial f(0)$.
As examples, $\partial I_{\mathcal{Z}}$ is the normal cone operator {for} $\mathcal{Z}$,
and $\partial f(x)=\{\nabla f(x)\}$ holds for all $x\in\R^c$ if $f\in\c{S}_{m,L}$ and $L<\infty$.
The identity function
$I:\R^c \rightarrow \R^c$, $I(x)=x$, satisfies $I= \partial \bq$.}

\textbf{Operators. }
{
An operator is a set-valued map $F: \R^c \rightrightarrows \R^c$ \cite{rockafellar2009variational, bauschke2020hilbert}. Its inverse is defined as $F^{-1}(x):=\{y\in\R^c\mid x\in F(y)\}$ for $x\in\R^c$.
A matrix $D\in\R^{c\times c}$ is identified with the linear map $x\mapsto Dx$ for any $x \in \R^c$.
In this paper, $(I + D F)^{-1}$ is referred to as a generalized resolvent, and $(F^{-1}-D)^{-1}$ is referred to as the generalized Yosida operator.
$F$ is said to {be} globally single-valued if, for every $x\in\R^c$, $F(x)$ consists of exactly one element. $F$ is said to be globally continuous if it is globally single-valued
and if the map $F:\R^c\to\R^c$, $x\mapsto F(x)$ is continuous.

If  $f\in\c{S}_{m,L}$ then $\partial f:\R^c\rightrightarrows\R^c$ is a set-valued map.
If $m>0$, then $f$ is strongly convex, which implies that $(\partial f)^{-1}$ is globally continuous.
The set of subdifferentials to $\mathcal{S}_{m, L}$ is $\partial \mathcal{S}_{m, L} = \{\partial f \mid {f \in \mathcal{S}_{m, L}}\}.$ We refer to Appendix \ref{app:extended_subdiff} for further details.}

\subsection{Composite Optimization Problems}

{Given $s\in\N$ functions $f_i\in \c{S}_{m_i, L_i}$, $i\in\{1,\ldots,s\}$,}
a composite optimization problem {amounts to finding} 
\vspace{-2mm}
\begin{align}
    \beta^* \in \argmin_{{\beta\in\R^c}} \sum_{i=1}^s f_i(\beta).  \label{prob:composite0}
\end{align}

The following zero-inclusion problem is a candidate necessary optimality condition.
\begin{prob}[Composite Optimization Principle]
\label{prob:composite}
    {Given $s\in\N$ functions $f_i\in \c{S}_{m_i, L_i}$ for $i\in\{1,\ldots,s\}$,} find a point {$\beta^*\in\R^c, w^* \in \R^{sc}$} which satisfies
    \vspace{-2mm}
\begin{align}
\textstyle \sum_{i=1}^s w^*_i &= 0, & \forall i \in \{1, \ldots, s\}: & \ w^*_i \in \partial f_i(\beta^*).
\label{eq:composite}
\end{align}
\end{prob}

In this work, we consider composite optimization problems that obey  the following restrictions:
\begin{enumerate}
    \item The overall problem is strongly convex ($\sum_{i=1}^s m_i > 0$).
    \item There is a unique pair $(\beta^*, w^*)$ such that \eqref{eq:composite} is satisfied.    
\end{enumerate}

Satisfaction of the principle in \eqref{eq:composite} is necessary for a point $\beta^*$ to be optimal for the optimization problem in \eqref{prob:composite0}. This necessity is also sufficient due to the imposition of strong convexity. 

An example of composite optimization subject to these restrictions arises {from} constrained minimization
 of a function $f \in \mathcal{S}_{m_f, L_f}$ with $0 < m_f < L_f<\infty$ on a closed convex set {$\mathcal{Z}\subset\R^c$}.
 {If we choose} $f_1=f$ and $f_2=I_\mathcal{Z}\in\c{S}_{0,\infty}$, {we infer}  that $f_1(x)+f_2(x)=f(x)<\infty$ holds iff $x\in\mathcal{Z}$, and hence \r{prob:composite0} is equivalent to $\beta^*\in\argmin_{x\in\c{Z}} f(x)$. The set $\partial f_1(\beta)$  is single-valued for all $\beta \in \R^c$ because $L_f < \infty$, and therefore the pair $(\beta^*, w^*)$ is unique with $w^* = (\partial f(\beta^*), -\partial f(\beta^*)). $.

{
{

With the function $f(z):=f_1(z^1)+\cdots +f_s(z^s)$ for $z=(z^1,\ldots,z^s)\in\R^{sc}$, the optimality condition
\eqref{eq:composite} reads as $0\in(\1_s\otimes I_c)^\top F(\1_s\otimes\beta^*)$ involving the operator $F:=\partial f$.

In this paper, we work with the following classes of operators.
\begin{definition}
    Given vectors  $m=(m_1,\ldots,m_s)$ and $L=(L_1,\ldots,L_s)$ with $\sum_{i=1}^s m_i > 0$, {let $\c{O}_{m,L}$ be} the class of all operators $F$   
    for which \eqref{eq:composite} has a unique solution $(\beta^*, w^*)$.
    
Moreover, $\mathcal{O}^0_{m, L}$ is the subset of all operators $F \in \mathcal{O}_{m, L}$ satisfying
     $f_i(0)=0$ for $i \in \{1, \ldots, s\}$ and $0\in F(0)$.
    
\end{definition}
}

\subsection{Optimization Algorithms}

Time-independent iterative schemes  to solve Problem \ref{prob:composite} can be expressed as a linear time-invariant system {$G$ represented by $(\Acl, \Bcl, \Ccl, \Dcl)$} interconnected with the operator
{$F\in\mathcal{O}_{m, L}$ %
as}
\begin{equation}
\label{eq:algorithm}
   \mat{c}{
        x_{k+1}  \hl z_k
    } = \mat{c|c}{ \mathcal{A} & \mathcal{B} \hl
    \mathcal{C} &  \mathcal{D}}\mat{c}{
        x_k \hl w_k
    },  \qquad   w_k \in F(z_k).
\end{equation}
{This feedback interconnection is depicted as a block-diagram shown in Figure \ref{fig:algorithm}.}

\begin{figure}[h]
    \centering
    \begin{tikzpicture}[xscale=1,yscale=1,baseline=(ko1)]
\def\dl{2*\dn}
\def\ds{3*\dn}
\node[sy3] (g) at (0,0)  {
$
F
$};
\tio{i2}{g}{east}{1/2};
\tio{o2}{g}{west}{1/2};

\node[sy3,below=\dl of g] (k) {
$
\mas{c|c}{\Acl & \Bcl \hl \Ccl & \Dcl } 
$
};

\tio{ki1}{k}{east}{1/2};
\tio{ko1}{k}{west}{1/2};
\draw[<-] (ko1)--   ($(ko1) + (-2*\ds, 0)$)  |- node[pos=.25]{$w$} (o2) ;
\draw[->] (ki1)--  ($(ki1) + (2*\ds, 0)$) |- node[pos=.25,swap]{$z$} (i2);
\end{tikzpicture}
    \caption{Standard algorithmic {inter}connection in \eqref{eq:algorithm}}
    \label{fig:algorithm}
\end{figure}
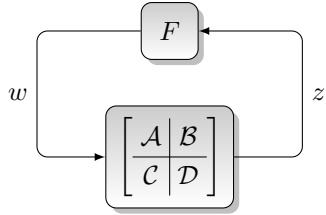
\vspace{-2mm}

As an example, the  Forward-Backward algorithm $x_{k+1} = (I + \gamma \partial f_2)^{-1}(x_k - \gamma \nabla f_1(x_k))$ may be {represented} as
\begin{align}
\label{eq:algorithm_gd}
   \mat{c}{
        x_{k+1}  \hl z_k^1 \\ z_k^2
    } = \left[ \mat{c|cc}{1 & -\gamma & -\gamma \hl 1 & 0 & 0 \\ 1 & -\gamma & -\gamma} \otimes I_c \right] \mat{c}{
        x_k \hl w_k^1 \\ w_k^2
    }, \qquad  \mat{c}{w_k^1 \\ w_k^2} = \mat{c}{\nabla f_1 (z_k^1) \\ \partial f_2(z_k^2)}.
\end{align}

{The Forward-Backward  algorithm exhibits a Kronecker Structure \cite{michalowsky2021robust}}:
\begin{definition}
    A system in \eqref{eq:algorithm} has \textbf{Kronecker Structure} if there exists matrices $(\Acl^a, \Bcl^a, \Ccl^a, \Dcl^a)$ with
    \label{def:kronecker}
    \begin{align}
         \mat{c|c}{ \mathcal{A} & \mathcal{B} \hl
    \mathcal{C} &  \mathcal{D}} =  \mat{c|c}{ \mathcal{A}^a & \mathcal{B}^a \hl
    \mathcal{C}^a &  \mathcal{D}^a} \otimes I_c. \label{eq:kronecker}
    \end{align}
\end{definition}
{We will use the superscript $a$ to denote matrices that {have a Kronecker structure
(e.g. $\Acl^a$ with  $\Acl^a \otimes I_c = \Acl$ in \eqref{eq:kronecker})}. Furthermore, we will use the Kronecker product to denote Kronecker-structured linear systems, using the notation
\begin{align}
\label{eq:linear_system_kron}
   & & G = G^a \otimes I:  \mat{c}{x_{k+1} \hl y_k} & =
    \left(\mat{c|c}{\Acl^a & \Bcl^a \hl \Ccl^a &\Dcl^a} \otimes I_c \right) \mat{c}{x_k \hl u_k}, \\
    &\text{abbreviated as} &  y & = \left(\mas{c|c}{        \Acl & \Bcl \hl \Ccl &\Dcl
     } \otimes I_c \right) u.
     \end{align}

{The following definition of a fixed point of an algorithm is standard.}

\begin{definition}
A \textbf{fixed point} of \eqref{eq:algorithm} %
is a  tuple $(x^*, z^*, w^*)$ satisfying
\begin{align}
   \mat{c}{
        x^*  \hl z^*
    } = \mat{c|c}{ \mathcal{A} & \mathcal{B} \hl
    \mathcal{C} &  \mathcal{D}}\mat{c}{
        x^* \hl w^*
    }, \qquad   w^* \in F(z^*). \label{eq:fixed_point}
\end{align}
\label{defn:fixed_point}
\end{definition}
\vspace{-2mm}
{This leads us to the following notion of algorithm convergence used in this paper.}

\begin{definition}
 The algorithmic interconnection  in \eqref{eq:algorithm} with $F \in \mathcal{O}_{m, L}$ is a \textbf{convergent optimization algorithm} for Problem \ref{prob:composite} with solution $\beta^*$ if there exists a fixed point $(x^*, w^*, z^*)$  of \eqref{eq:algorithm} {which satisfies the
 following three properties:}
\begin{subequations}
\begin{align}
\text{Consensus:}  \qquad &z^* =  \1_s \otimes \beta^*, \label{eq:reg_property} \\
\text{Optimality:}  \qquad & (\1_s \otimes I_c)^\top w^* = 0, \label{eq:reg_zero}\\
    \text{Attractivity:}  \qquad &\forall x_0 \in \R^n: \lim_{k \rightarrow \infty} (x_k, w_k, z_k) = (x^*, w^*, z^*).  \label{eq:least_convergence}
\end{align}
\label{eq:optim_requirements}
\end{subequations}
\label{defn:optimization_algorithm}
\end{definition}

\vspace{-6mm}
\begin{remark}

A sufficient condition for a pair $(\beta^*, w^*)$ solving Problem \ref{prob:composite} to be unique is if at most one index $i \in \{1, \ldots, s\}$ satisfies $L_i = \infty.$
If  $(\Acl, \Ccl)$ is detectable, then uniqueness of $x^*$ follows, because there is at most one $x^*$ such that
\begin{align}
\mat{c}{\Acl - I \\ \Ccl} x^* = \mat{c}{-\Bcl w^* \\ \1_s \otimes \beta^* - \Dcl w^*
}.
\end{align}

\end{remark}

\subsection{Algorithms over Networks}\label{Snet}

The  algorithm \eqref{eq:algorithm} involves a linear system $G$ directly interconnected to the operator oracle $F$. Design of optimization algorithms may also occur in settings where the oracle $F$ is not directly accessible. Example instances include time delays, packet drops, and network cross-talk \cite{doostmohammadian2025survey}.
{We model this setting by assuming that $G$ is given by the interconnection of a network 
\begin{align}
\label{eq:network_dynamics}
    P: & & \mat{c}{
    x^N_{k+1}  \hl z_k \\ y_k
    } &= \mat{c|cc}{ A & B_1 & B_2 \hl
    C_1 & D_1 & D_{12} \\
    C_2 & D_{21} & D_2 }\mat{c}{
        x^N_k \hl w_k \\ u_k
    },
\intertext{with a controller $\Kc$ with $n_y$ inputs and $n_u$ outputs as described by}
        \Kc: & & \mat{c}{
        \xi_{k+1}  \hl u_k}&=\mat{c|c}{\Ac & \Bc \hl
    \Cc & \Dc}\mat{c}{
        \xi_k \hl y_k
    }. \label{eq:controller}
\end{align}
If $E:=I - D_2 \Dc$ is invertible, we recall that this interconnection is well-posed
and
$P\star\Kc$
admits a state-space description with $(\Acl,\Bcl,\Ccl,\Dcl)$ as in \eqref{cll}.

For $G=P\star\Kc$, the algorithmic {inter}connection $F\star G=F\star (P\star\Kc)$
is visualized by block-diagrams in Figure \ref{fig:alPetwork}. These reveal how the (given and fixed)
network $P$ bridges the oracle $F$ and the (to be designed controller) $\Kc$ to generate the
algorithm $F\star (P\star\Kc)$. The specific choice
\begin{align}
            P^0: & & \mat{c}{
        z_k \\ y_k
    } &= \mat{cc}{0 & I_{cs} \\ I_{cs} & 0} \mat{c}{
        w_k \\ u_k
    } \label{eq:direct_interconnection}
\end{align}
recovers a direct interconnection, since then $P^0 \star \Kc=\Kc$ and $G$ is just equal to $\Kc$.}

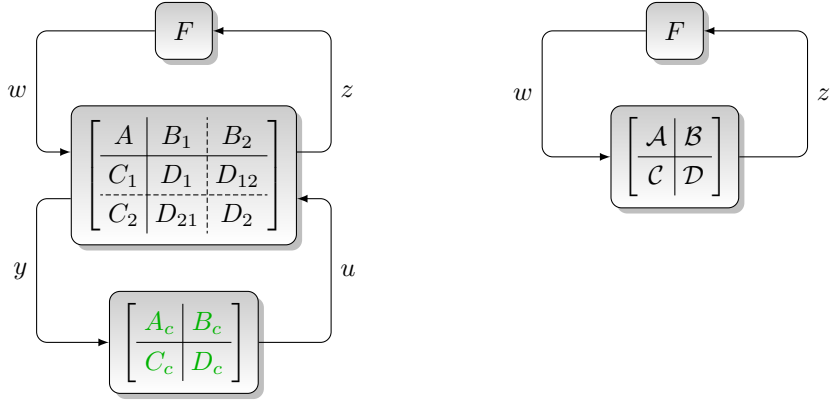
\begin{figure}[h]
    \centering
    \hfill
    \begin{subfigure}[T]{0.4\textwidth}
    \begin{tikzpicture}[xscale=1,yscale=1,baseline=(ko1)]
\def\dl{2*\dn}
\def\ds{3*\dn}
\node[sy3] (F) at (0,0)  {
$F$};
\tio{iF}{F}{east}{1/2};
\tio{oF}{F}{west}{1/2};

\node[sy3,below=\dl of F] (N) {
$
\mas{c|c:c}{ A & B_1 & B_2 \hl C_1 & D_1 & D_{12} \hdl C_2 & D_{21} & D_{2} } 
$
};

\node[sy3,below=\dl of N] (k) {
$
\mas{c|c}{\Ac & \Bc \hl \Cc & \Dc } 
$
};

\tio{ki1}{k}{east}{1/2};
\tio{ko1}{k}{west}{1/2};

\tio{Ni1}{N}{west}{1/3};
\tio{No1}{N}{east}{1/3};

\tio{Ni2}{N}{west}{2/3};
\tio{No2}{N}{east}{2/3};
\draw[<-] (Ni2)--   ($(Ni2) + (-1*\ds, 0)$)  |- node[pos=.25]{$w$} (oF) ;
\draw[->] (No2)--  ($(No2) + (1*\ds, 0)$) |- node[pos=.25,swap]{$z$} (iF);

\draw[->] (Ni1)--   ($(Ni1) + (-1*\ds, 0)$)  |- node[pos=.25, swap]{$y$} (ko1) ;
\draw[<-] (No1)--  ($(No1) + (1*\ds, 0)$) |- node[pos=.25]{$u$} (ki1);
\end{tikzpicture}
    \end{subfigure}
    \begin{subfigure}[T]{0.4\textwidth}
    \begin{tikzpicture}[xscale=1,yscale=1,baseline=(ko1)]
\def\dl{2*\dn}
\def\ds{3*\dn}
\node[sy3] (g) at (0,0)  {
$
F
$};
\tio{i2}{g}{east}{1/2};
\tio{o2}{g}{west}{1/2};

\node[sy3,below=\dl of g] (k) {
$
\mas{c|c}{\Acl & \Bcl \hl \Ccl & \Dcl } 
$
};

\tio{ki1}{k}{east}{1/2};
\tio{ko1}{k}{west}{1/2};
\draw[<-] (ko1)--   ($(ko1) + (-2*\ds, 0)$)  |- node[pos=.25]{$w$} (o2) ;
\draw[->] (ki1)--  ($(ki1) + (2*\ds, 0)$) |- node[pos=.25,swap]{$z$} (i2);
\end{tikzpicture}
    \end{subfigure}
     \hfill
    \caption{Block diagrams of an algorithm over a network {given by} $F \star (P \star \Kc)$ (left) and its closed-loop representation  $F \star G$ (right) with $G = P \star K$ {represented by $(\Acl, \Bcl, \Ccl, \Dcl)$.}}
    \label{fig:alPetwork}
\end{figure}

{Algorithm synthesis involves the design of a controller $\Kc$ such that the interconnection of $F$ and $P \star \Kc$ is a convergent optimization algorithm whenever $F\in \mathcal{O}_{m, L}$.}
The goal of this paper {is to} derive structural conditions on any controller $\Kc$ {which solves this synthesis problem. Deriving these conditions  restricts a } search over all controllers $\Kc$ to only  those the structure-obedient controllers.

\section{Well-Posedness and Algorithm Execution}

\label{sec:execution}

{

The algorithm \eqref{eq:algorithm} involves the implicit constraint $z_k=\c{C}x_k+\c{D}w_k$, $w_k\in F(z_k)$. This translates into the equivalent inclusions $\c{C}x_k+\c{D}w_k\in F^{-1}(w_k)$ and $w_k\in (F^{-1}-\c{D})^{-1}(\c{C}x_k)${, which motivates the following definition of well-posedness.}
\begin{definition}
    The algorithm   \eqref{eq:algorithm} is \textbf{well-posed} if {the generalized Yosida operator $H:= (F^{-1} - \mathcal{D})^{-1}$ is globally continuous.} \label{def:well_posed}
\end{definition}

Well-posed interconnections \eqref{eq:algorithm} can be rewritten as an explicit recursion
     \begin{align}
      x_{k+1} = \mathcal{A} x_k + \Bcl H (\Ccl x_k), \ \ w_k=H (\Ccl x_k),\ \
      z_k=\c{C}x_k+\c{D}H(\c{C}x_k) \ \ \forall k\in\N.   \label{eq:algorithm_causal}
     \end{align}
{Well-posedness implies that trajectories $(x_k, w_k, z_k)_{k \in \N}$ exist and are unique for every $x_0 \in \R^n$, and that convergence in $x$ implies convergence in $(w, z)$ (Proposition \ref{prop:well_posed_converge}).

Our developed  convergence and structural results require well-posedness of algorithms. In the sequel, we refer to the algorithm as $F \star G$ (represented by \eqref{eq:algorithm}), with the
tacit understanding that it is well-posed (and can also be expressed as \eqref{eq:algorithm_causal}).}

Unfortunately, {the evaluation of $H$} may be as challenging as solving the original optimization problem \eqref{eq:composite}. 
{If $\Dcl$ admits a lower-triangular block structure (possibly after
a permutation of $f_1,\ldots,f_s$), the evaluation of $H$ can be achieved with the generalized resolvents of
$\partial f_i$ for $i\in\{1,\ldots,s\}$. To perform this algorithm evaluation, we partition $\c{B}=(\c{B}_i)$, $\c{C}=(\c{C}_i)$ and $\c{D}=(\c{D}_{ij})$  from \eqref{eq:algorithm} into $c\times c$ blocks to form the representation }
\begin{align}
\mat{c|c}{\Acl & \Bcl \hl \Ccl & \Dcl} &= \mat{c|cccc}{
   \Acl & \Bcl_1 & \Bcl_2  & \hdots & \Bcl_s \hl
   \Ccl_1 & \Dcl_{11} & 0 &   \hdots & 0\\
   \Ccl_2 & \Dcl_{21 } & \Dcl_{22} &   \hdots & 0\\
   \vdots & \vdots & \vdots  & \ddots & \vdots\\
   \Ccl_s & \Dcl_{s1} & \Dcl_{s2} & \hdots & \Dcl_{ss}}. \label{eq:tri_partition}
\end{align}

The block-sparsity pattern of $\mathcal{D}$ defines an \textit{information structure} of the algorithm, specifying the coordinate values $w^{j}_k$ of the subdifferentials that may be used to compute the other subdifferential coordinates $w^i_k$ within the same iteration $k \in \N$. We encode this block-sparsity pattern into a 

Descriptions and labels for the eight possible sparsity patterns of the  block-lower-triangular  matrices $  \mathcal{D} $ for $s=2$ are listed in Table \ref{tab:info_structure}, in which $\bullet$ indicates a possibly nonzero $\Dcl_{ij}$ submatrix in \eqref{eq:tri_partition}.

\begin{table}[!h]
    \centering
     \caption{Block-lower-triangular sparsity patterns of $\mathcal{D}$ for  $s=2$.}
    \begin{tabular}{r|c | c}
          & Sequential & Parallel  \\
         Pure Yosida & $\mat{cc}{\bullet & 0 \\ \bullet & \bullet}$ & $\mat{cc}{\bullet & 0 \\ 0 & \bullet}$ \\
         Mixed Yosida and Gradient& $\mat{cc}{0 & 0 \\ \bullet & \bullet}, \quad  \mat{cc}{\bullet & 0 \\ \bullet & 0}$ &  $ \mat{cc}{0& 0 \\ 0 & \bullet}, \quad   \mat{cc}{\bullet & 0 \\ 0 & 0}$ \\
         Pure Gradient &  $\mat{cc}{0 & 0 \\ \bullet & 0}$  & $\mat{cc}{0 & 0 \\ 0 & 0}$
    \end{tabular}
    \label{tab:info_structure}
\end{table}

In Table \ref{tab:info_structure}, an algorithm is `parallel' if all $w^i$ coordinates for the subgradient can be evaluated independently (block-diagonal $\mathcal{D}$), or `sequential' if there are dependencies in  $w^i$ computation.

\begin{lemma}\label{Lwp} \label{lem:well_posed}
An algorithm \eqref{eq:algorithm} with $\Dcl$ having a block-lower-triangular structure as partitioned in \eqref{eq:tri_partition} is well-posed if the  diagonal blocks $\c{D}_{ii}$ satisfy
\eql{wpc}{%
\mathrm{Sym}(- \si_i + \mathcal{D}_{ii} (I - m_i \mathcal{D}_{ii})^{-1}) \prec 0
\te{with}\si_i:=(L_i-m_i)^{-1}\te{for}i\in\{1,\ldots,s\}.
}
We define the generalized Yosida operators $H_i := (\partial f_i^{-1} - \Dcl_{ii})^{-1}$ for each $i \in \{1, \ldots, s\}$.
The recursion in \eqref{eq:algorithm_causal} admits the explicit description starting from an initial condition $x_0 \in \R^n$ as {follows: Determine
}
\begin{equation}
\label{eq:algorithm_causal_eval}
w_{k}^i = H_i\left(\Ccl_{i} x_{k} + \textstyle\sum_{j=1}^{i-1} \Dcl_{ij} w_{k}^j\right)
\te{for $i=1,2,\ldots,s$ in ascending order}
\end{equation}
and set $x_{k+1} = \Acl x_k + \textstyle \sum_{i=1}^s \Bcl_i w_k^i$.
\end{lemma}
\begin{proof}
The relation $w\in (F^{-1}-\c{D})^{-1}(\mathcal{C} x)$ is equivalent to $z=\mathcal{C} x+\c{D}w$, $w\in F(z)$. In a partition of the vectors $(w, z)$ compatible to the index  $i \in \{1, \ldots, s\}$ {of the blocks of $\c{D}$} in \eqref{eq:tri_partition}, the vector $z$ can be expressed as
\eql{h20}{
\textstyle z^i=\Ccl_i x^i+\sum_{j=1}^{i-1} \c{D}_{ij}w^j+\c{D}_{ii}w^i,\ w^i\in F_i(z^i)\te{for}i\in\{1,\ldots,s\}.\textstyle
}
Since $F\in \c{O}_{m,L}$ implies $F_i=\partial f_i\in\partial\c{S}_{m_i,L_i}$, {Lemma \ref{lem:well_posed_indiv}
assures that  $H_i= (F_i^{-1}-\Dcl_{ii})^{-1}$} is globally continuous under 
\r{wpc}. Therefore, \r{h20} is
equivalent to
\eql{h1}{
\textstyle w^i=H_i\left(\Ccl_i x +\sum_{j=1}^{i-1} \c{D}_{ij}w^j\right)\te{for}i\in\{1,\ldots,s\}.
}
This representation of $w\in (F^{-1}-\Dcl)^{-1}(z)$ proves that $H = (F^{-1}-\Dcl)^{-1}$ is
globally continuous, and  hence the algorithm is well-posed. Moreover, \r{h1} {leads to \eqref{eq:algorithm_causal_eval}
and hence to an explicit algorithm description.}
\end{proof}

\label{sec:information}

\section{Convergence of Optimization Algorithms}

\label{sec:convergence}

Convergence conditions of optimiziation algorithms $F \star (P \star \Kc)$ in the class $F \in \mathcal{O}_{m, L}$ will be derived by forming a link to regulation theory for linear systems \cite{francis1976internal, francis1977linear, saberi2012control}.

\subsection{Error Formulation}\label{Serr}

To capture the optimality condition \eqref{eq:reg_zero}, we introduce a {so-called} consensus matrix.

\begin{definition}
A matrix $N\in \R^{sc \times (s-1) c}$ is a \textbf{consensus matrix} if
    \label{defn:consensus}
    \begin{align}
    \mathrm{rank}(N) = (s-1)c\quad \mathrm{and} \qquad
    {\mathrm{ran}(N)=\mathrm{null}((\1_s\otimes I_c)^\top).}
    \label{eq:consensus_matrix}
\end{align}
\end{definition}
A particular choice is $N = \begin{pmatrix}
    I_{s-1}  \\ -\1_{s-1}^\top
\end{pmatrix} \otimes I_c$, such as $N = \small{\begin{pmatrix}
    1 & 0 \\ 0 & 1 \\ -1 & -1
\end{pmatrix}} \normalsize \otimes I_c$ {for} $s=3$.
When $s=1$, the consensus matrix is omitted by setting $N = [\cdot]$.
{
The unique $(\beta^*, w^*)$ of Problem~\ref{prob:composite} with $w^* \in F(\1_s \otimes \beta^*)$  can then be characterized by
\eql{defwstar}{0\in F(\1_s\otimes\beta^*)-N\hat w^*\te{for some}\hat w^*\in\R^{(s-1)c}.}

Let us partition $w^*:=N\hat w^*\in\R^{sc}$ as $w^*=\col(w^*_1,\ldots,w^*_s)$, with $w^*_i\in\partial f_i(\beta^*)$ and
associate to $f_i$ the zero-centered function 
$\tilde{f}_i(z):=f_i(z+\beta^*)-f_i(\beta^*)-(w^*_i)^\top z$
with $\t{f}_i\in\c{S}_{m,L}^0$ for $i\in\{1,\ldots,s\}$.
Then the function $\t f(z):=\sum_{i=1}^s \t{f}_i(z^i)$ for $z=(z^1,\ldots,z^s)\in\R^{sc}$ leads to the error operator $\tilde{F}:=\partial\t{f} \in\mathcal{O}_{m, L}^0$, which satisfies
\begin{align}
    \tilde{F}(\tilde{z}) = F(\tilde{z} + \1_s \otimes \beta^*) - N \hat{w}^*\te{for }\t z\in\R^{sc}. \label{tF}
\end{align}

Figure \ref{fig:disturbance_sources} visualizxes the original interconnection (left) and the equivalent $(\beta^*, \hat{w}^*)$-disturbed error system.
}

\begin{figure}[!h]
    \centering
\begin{subfigure}{0.45\linewidth}
    \centering
    \begin{tikzpicture}[xscale=1,yscale=1,baseline=(ko1)]
\def\dl{2*\dn}
\def\ds{3*\dn}
\node[sy3] (F) at (0,0)  {
$F$};
\tio{iF}{F}{east}{1/2};
\tio{oF}{F}{west}{1/2};

\node[sy3,below=\dl of F] (N) {
$
\mas{c|c:c}{ A & B_1 & B_2 \hl C_1 & D_1 & D_{12} \hdl C_2 & D_{21} & D_{2} } 
$
};

\node[sy3,below=\dl of N] (k) {
$
\mas{c|c}{\Ac & \Bc \hl \Cc & \Dc } 
$
};

\tio{ki1}{k}{east}{1/2};
\tio{ko1}{k}{west}{1/2};

\tio{Ni1}{N}{west}{1/3};
\tio{No1}{N}{east}{1/3};

\tio{Ni2}{N}{west}{2/3};
\tio{No2}{N}{east}{2/3};
\draw[<-] (Ni2)--   ($(Ni2) + (-1*\ds, 0)$)  |- node[pos=.25]{$w$} (oF) ;
\draw[->] (No2)--  ($(No2) + (1*\ds, 0)$) |- node[pos=.25,swap]{$z$} (iF);

\draw[->] (Ni1)--   ($(Ni1) + (-1*\ds, 0)$)  |- node[pos=.25, swap]{$y$} (ko1) ;
\draw[<-] (No1)--  ($(No1) + (1*\ds, 0)$) |- node[pos=.25]{$u$} (ki1);
\end{tikzpicture}
    \caption{Interconnection with $F \in \mathcal{O}_{m, L}$ }
    \label{fig:disturbance_sources_orig}
    \end{subfigure}
\hfill
\begin{subfigure}{0.45\linewidth}
\begin{tikzpicture}[xscale=1,yscale=1, baseline = (F)]
\def\dl{2*\dn}
\def\ds{3*\dn}
\node[sy3] (F) at (0,0)  {
$\tilde{F}$};
\tio{iF}{F}{east}{1/2};
\tio{oF}{F}{west}{1/2};

\node[sy3,below=\dl of F] (N) {
$
\mas{c|c:c}{ A & B_1 & B_2 \hl C_1 & D_1 & D_{12} \hdl C_2 & D_{21} & D_{2} } 
$
};

\node[sy3,below=\dl of N] (k) {
$
\mas{c|c}{\Ac & \Bc \hl \Cc & \Dc } 
$
};

\tio{ki1}{k}{east}{1/2};
\tio{ko1}{k}{west}{1/2};

\tio{Ni1}{N}{west}{1/3};
\tio{No1}{N}{east}{1/3};

\tio{Nih}{N}{west}{1/2};
\tio{Noh}{N}{east}{1/2};

\tio{Ni2}{N}{west}{2/3};
\tio{No2}{N}{east}{2/3};

\draw[->] (Ni1)--   ($(Ni1) + (-1*\ds, 0)$)  |- node[pos=.25, swap]{$y$} (ko1) ;
\draw[<-] (No1)--  ($(No1) + (1*\ds, 0)$) |- node[pos=.25]{$u$} (ki1);

\node[sum,right= 4.6*\dl of F] (sbeta) {
$+$
};

\node[sum,left= 4.6*\dl of F] (sw) {
$+$
};

\draw[->] (sw.south)   |-  node[pos=0.25, swap]{$w$} (Ni2) ;
\draw[<-] (sbeta.south) |-  node[pos=0.25]{$z$} (No2);

\draw[<-] (sw)--   node[pos=.5, swap]{$\tilde{w}$} (oF);
\draw[->] (sbeta)--   node[pos=.5]{$\tilde{z}$} (iF);

\draw[->] ($(sw) + (0, 2*\ds)$)--   node[pos=.5,swap]{$N \hat{w}^*$} (sw);
\draw[->] ($(sbeta) + (0, 2*\ds)$)--   node[pos=.5]{$-\beta^*$} (sbeta);

\end{tikzpicture}
\caption{Interconnection with $\tilde{F} \in \mathcal{O}_{m, L}^0$ }
\label{fig:disturbance_sources_tilde}
\end{subfigure}
\hfill
\vspace{-3mm}
\caption{{Original interconnection (left) and error system (right) with exogenous disturbance $(-\beta^*, N\hat{w}^*)$}}
\label{fig:disturbance_sources}
\end{figure}
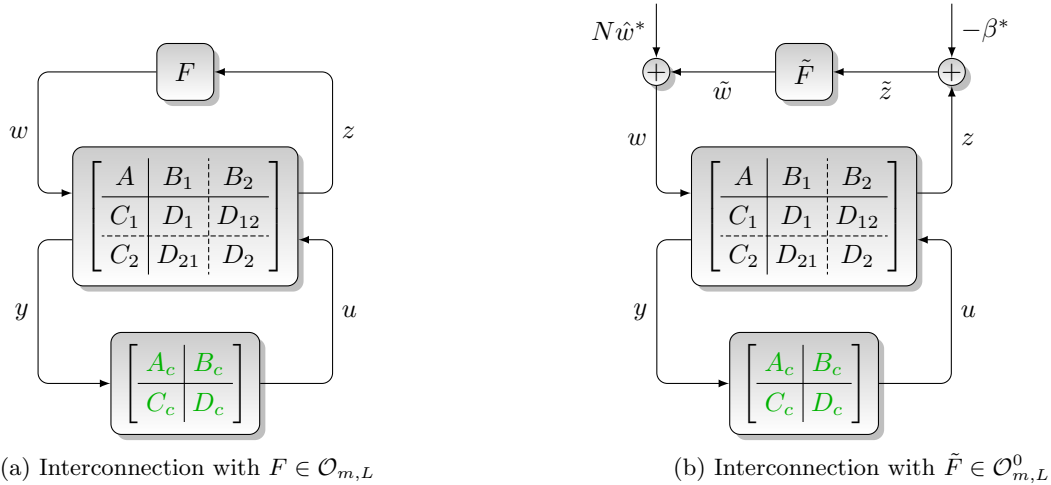

The error signals
\eql{errsig}{\t z_k := z_k - \1_s \otimes \beta^*\te{and}\t w_k:=w_k-N\hat w^*}
may be computed from any trajectory of $F\star (P\star\Kc)$, and 
are related as $\t w_k\in F(z_k)-N\hat w^*=F(\t z_k+\1_s\otimes \beta^*)-N\hat w^*=\t F(\t z_k)$ for all  $k\in\N$.

 Any trajectory of $F\star (P\star\Kc)$ (described by \eqref{eq:network_dynamics}-\eqref{eq:controller} and $w_k\in F(z_k)$)
therefore generates a trajectory  of the following error system
 \begin{subequations}
\label{eq:regnonlinear}
\begin{align}
   \t F: & & \tilde{w}_k &\in \tilde{F}(\tilde{z}_k), \label{eq:regnonlinear_delta}\\
   P_e: & &  \mat{c}{
    x^N_{k+1}  \hl \tilde{z}_k \hdl e_k \hdl y_k
    } &=
    \mat{c|c:cc:c}{ A & B_1 & 0 & B_1 N & B_2 \hl
    C_1 & D_1 & \1_s \otimes I_c  & D_{1}N  &  D_{12} \hdl
    C_1 & D_1 & \1_s \otimes I_c  & D_{1}N & D_{12} \hdl
    C_2 & D_{21} & 0 & D_{21}N  & D_2 }\mat{c}{
        x^N_k \hl \tilde{w}_k \hdl -\beta^*\\\hat w^* \hdl  u_k
    },
\label{eq:regnonlinear_network} \\
    \Kc: & &      \mat{c}{
        \xi_{k+1}  \hl u_k
    } &= \mat{c|cc}{ \Ac & \Bc \hl
    \Cc & \Dc}\mat{c}{
        \xi_k \hl y_k
    }.
    \end{align}
\end{subequations}
Here $e_k$ is a copy of $\t z_k$ and interpreted as an error output,
while $(-\beta^*,\hat w^*)$ is interpreted as a constant disturbance input.
We refer to this error interconnection as $\t{F}\star (P_e\star\Kc)$.
As a consequence,
$P_e\star\Kc$ is well-posed iff $P\star \Kc$ is well-posed.
Similarly, in view of Proposition \ref{cor:well_posed_exp} and \r{tF}, $\t F\star (P_e\star\Kc)$ is well-posed iff
$F\star (P\star\Kc)$ is well-posed.

If the well-posed algorithm $F\star (P\star\Kc)$ converges to the fixed point $(x^*,N\hat w^*,1\otimes\beta^*)$,
we infer all trajectories of the error system satisfy $\lim_{k\to\infty} e_k=0$, despite being affected by the (unknown) constant disturbance input $(-\beta^*,\hat w^*)$.

\subsection{Test Quadratics and Assumptions}\label{Stqu}

{The interconnections $F\star (P\star \Kc)$ and \eqref{eq:regnonlinear} involve the nonlinearities $F \in \mathcal{O}_{m, L}$ and $\tilde{F} \in \mathcal{O}_{m, L}^0$, and convergence or output regulation must occur for all nonlinearities in the respective classes. This leads to a problem of nonlinear robust output regulation. Since we are confronted with structured uncertainties,
neither nonlinear nor linear output regulation theory \cite{isidori2003robust,francis1977linear} can be applied to designing robustly converging algorithms.

To develop conditions for $F \star (P \star \Kc)$ to converge for all $F \in \mathcal{O}_{m, L}$, we first {investigate convergence of} $F^t \star (P \star \Kc)$ for affine maps $F^t$ that arise from considering the composite optimization problem \eqref{eq:composite} with $f^t(z^1,\ldots,z^s) := \sum_{i=1}^s f^t_i(z^i)$ involving the quadratic functions
\begin{align}
    f^t_i(z_i) & := \frac{m_i}{2} \norm{\beta}_2^2 + b_i^\top z_i \te{where}b_i \in \R^c\te{for} i\in \{1, \ldots, s\}. \label{eq:test_quadratics}
\end{align}
With $\bm := \diag(m) \otimes I_c$ and  $b := \col(b_1,\ldots,b_s)$,
we infer $F^t(z) = \nabla f^t(z) = \bm z + b$ and $\tilde{F}^t(\tilde{z}) = \bm \tilde{z}$.
We refer to these choices as test quadratics.

We note that the composite optimality principle in  \eqref{eq:composite} reads as $\beta^*\sum_{i=1}^s m_i+\sum_{i=1}^s b_i=0$ for all test quadratics. Strong convexity implies uniqueness of $(\beta^*, w^*)$ as  $\textstyle  \beta^*  \textstyle  = \left(\sum_{i=1}^s m_i\right)^{-1}\sum_{i=1}^s b_i$, and $w^*_i = m_i \beta^* + b_i$ for all $i$. The test quadratics then always satisfy
$F^t \in  \mathcal{O}_{m, L}$, and hence $\tilde{F}^t \in \mathcal{O}_{m, L}^0$.

If $E(\bm) :=  I -  D_1 \bm$ is invertible, the algorithm $\t F^t\star P_e$ described by \eqref{eq:regnonlinear_delta} and
\eqref{eq:regnonlinear_network} has the state-space representation with exogenous disturbance 
\begin{align}
P^\bm:&&
\mat{c}{x^N_{k+1}  \hl e_k \hdl y_k} &=
    \mat{c|c:c}{A^\bm & B_1^\bm & B_2^\bm \hl
    C_1^\bm & D_1^\bm & D_{12}^\bm \hdl
    C_2^\bm & D_{21}^\bm & D_{2}^\bm}
    \mat{c}{x^N_k \hl d \hdl  u_k}\te{where}d=\mat{c}{-\beta^*\\\hat w^*}\label{qua}
\end{align} 
with the matrices
\begin{align}
\setlength{\arraycolsep}{0.5pt}
    \mat{c|cc}{A^\bm & B_1^\bm & B_2^\bm \hl
    C_1^\bm & D_1^\bm & D_{12}^\bm \\
    C_2^\bm & D_{21}^\bm & D_{2}^\bm} \!\!:=\!\! \mat{c|cc:c}{ A & 0 & B_1 N & B_2 \hl
    C_1 & \1_s \!\otimes\! I_c  & D_{1}N  &  D_{12} \hdl
    C_2 &  0 & D_{21}N  & D_2 } \!\!+\!\! \mat{c}{B_1 \hl D_1 \\ D_{21}} \! \bm E(\bm)^{-1} \! \mat{c|cc:c}{C_1 & (\1_s \!\otimes\! I_c)  & D_{1}N & D_{12}}\!\!. \label{eq:loop_closure_m}
\end{align}

Solving the algorithm design problem for $F^t \star (P \star \Kc)$ then boils down to solving a nominal regulation problem for the linear system $P^\bm\star\Kc$. To guarantee that $P^\bm$ can be
internally stabilized with some $\Kc$, we require that
$\left(A^\bm,   B^\bm_2\right)$ is stabilizable and $\left(A^\bm,   C^\bm_2\right)$ is detectable.

Let us collect all assumptions as follows.}

\begin{assum}
    \label{assum:well_posed}
    The matrix $E(\bm) = I - D_1 \bm$ is invertible.
\end{assum}

\begin{assum}
 $\left(A^\bm,   B^\bm_2\right)$ is stabilizable and $\left(A^\bm,   C^\bm_2\right)$ is detectable.
   \label{assum:stab_dect}
\end{assum}

We recall that the class $\mathcal{O}_{m, L}$ is defined such that $\sum_{i=1}^s m_i > 0$.

\subsection{Convergence of Algorithms}

We are now ready to formulate the first main contribution of this paper.

{
\begin{theorem}%
\label{thm:convergence}
Suppose we are given a network $P$, a consensus matrix $N$, and a class of subgradient operators $\mathcal{O}_{m, L}$ such that Assumptions \ref{assum:well_posed}-\ref{assum:stab_dect} hold.
If the interconnection of $F$ and $(P \star \Kc)$ is well-posed for all $F \in \mathcal{O}_{m, L}$, then the algorithm $F \star (P \star \Kc)$ is convergent if and only if the following conditions are satisfied:
\begin{enumerate}
    \item \textbf{Robust Stability:} $I-D_2\Dc$ is invertible and $\tilde{F}\star (P\star \Kc)$ described by
    \begin{subequations}
    \label{eq:PK}
    \begin{align}
        \label{eq:reg_zero_nonlinear}
    \mat{c}{
    x_{k+1}^N  \hl z_k \hdl y_k
    } &= \mat{c|c:c}{ A & B_1 & B_2 \hl
    C_1 & D_1   & D_{12} \hdl
    C_2 & D_{21}  & D_2 }\mat{c}{
        x_k^N \hl w_k  \hdl  u_k
    },\\
         \mat{c}{
        \xi_{k+1}  \hl u_k
    } &= \mat{c|c}{ \Ac  & \Bc \hl
    \Cc & \Dc}\mat{c}{
        \xi_k \hl y_k
    }
    \label{eq:regnonlinear_F_zero}
    \end{align}
    amd $w_k\in \t F(z_k)$ is well-posed and satisfies $\lim_{k\rightarrow \infty} (x_k, \xi_k) = 0$ for all
     initial states $(x_0,\xi_0)$ and all $\tilde{F}\in \mathcal{O}^0_{m, L}$.
    \end{subequations}

    \item Solvability of \textbf{Regulator Equations:}  A unique tuple $(\Pi, \Theta, \Gamma, \Phi)$ exists with 
    \begin{subequations}
    \label{eq:nominal_regulation_control_sys}
    \begin{align}
        \mat{c|cc:c}{A & 0 & B_1 N & B_2 \hl
        C_1  & \1_s \otimes I_c & D_{1}N &  D_{12} \hdl
        C_2 & 0 & D_{21}N & D_2} \mat{c}{\Pi\hl I\hdl \Gamma}&=\mat{c}{\Pi \hl 0 \hdl \Phi}, \label{eq:nominal_regulation_control_sys_plant}  \\
\mat{cc}{\Ac&\Bc\\ \Cc&\Dc}
\mat{c}{\Theta\\\Phi}  &= \mat{c}{\Theta \\ \Gamma}. \label{eq:nominal_regulation_control_sys_control}
    \end{align}
    \end{subequations}

\end{enumerate}

Given a map $F \in \mathcal{O}_{m, L}$ with optimal solution $(\beta^*, w^*)$ with $w^* = N \hat{w}^*$, convergence of the well-posed $F \star (P \star \Kc)$ by conditions 1 and 2 imply the following asymptotics for all initial conditions $(x^N_0, \ \xi_0)$:

\begin{align}
    \lim_{k \rightarrow \infty}
    \mat{c}{x_k^N \\[-0.4ex] \xi_k \hl  y_k \\[-0.4ex] u_k} & = \mat{c}{\Pi \\[-0.4ex] \Theta \hl \Phi \\[-0.4ex] \Gamma} \mat{c}{-\beta^* \\ \hat{w}^*}, &   \lim_{k \rightarrow \infty}
    \mat{c}{z_k \\ w_k } & =  \mat{c}{z^* \\ w^*} = \mat{c}{\1_s \otimes \beta^* \\ N \hat{w}^*} .
     \label{eq:asymptotic_regulation_closed_sig}
\end{align}
\end{theorem}

\begin{proof}
    See Appendix \ref{app:convergence_proof}. 
\end{proof}

If the Robust Stability condition fails, there exists an operator $\t{F}\in\c{O}_{m,L}^0$ such that
$\t{F} \star (P\star \Kc)$ is not convergent. If the Regulator Equation has no solution, there exists a test quadratic $F^t$ such that that $F^t \star (P\star \Kc)$ is not convergent.

The regulator equation conditions \eqref{eq:nominal_regulation_control_sys} are independent of the specific $F \in \mathcal{O}_{m, L}.$
The regulator equation \eqref{eq:nominal_regulation_control_sys_plant}
only involves the description of the network $P$ and the consensus matrix $N$ (but not  the controller $\Kc$), and can
be verified using linear algebra. There may exist multiple solutions $(\Pi, \Gamma, \Phi)$ to \eqref{eq:nominal_regulation_control_sys_plant}, but the tuple $(\Pi, \Gamma, \Phi, \Theta)$ for a given convergent $F \star (P \star \Kc)$ is unique.

The Robust Stability condition is significantly more involved to verify than the Regulator Equation condition, because \eqref{eq:reg_zero_nonlinear} is a statement about global attractivity {of the fixed point $0$ for all operators in} $\mathcal{O}_{m, L}^0$.
In this work, we numerically certify Robust Stability  through sufficient LMI-based conditions.

\subsection{Convergence under Direct Interconnection}

In the case of direct interconnection dynamics in \eqref{eq:direct_interconnection}, the {error system $P^0_e$
in \eqref{eq:regnonlinear_network} is static with
\begin{align}
       P_e^0: \mat{c}{
    \t z_k \hdl e_k \hdl y_k
    } &= \mat{c: cc:c}{0& \1_s \otimes I_c & 0  &  I_{sc} \hdl
     0 & \1_s \otimes I_c & 0   & I_{sc} \hdl
    I_{sc} & 0 &  N & 0 }\mat{c}{
        \t w_k \hdl -\beta^* \\ \hat{w}^* \hdl  u_k
    }.      \label{eq:reg_aff_network_static}
\end{align}}

The regulator {equations \eqref{eq:nominal_regulation_control_sys_plant} for $P_e^0$ in
\eqref{eq:reg_aff_network_static} are satisfied by
\begin{align}
    \Pi &= [\cdot], & \Gamma &= -\mat{cc}{\1_s \otimes I_c  &0 }, &  \Phi &= \mat{cc}{0 &N }. \label{eq:gam_phi}
\end{align}
}

Eq. \eqref{eq:nominal_regulation_control_sys_control} may be expressed using the parameters in \eqref{eq:gam_phi} as
\begin{align}
\mat{c}{\Ac - I\\ \Cc } \Theta = \mat{cc}{0 &  -\Bc N\\  -(\1_s \otimes I_c) & - \Dc N}.
    \label{eq:direct_regulator}
\end{align}
The equation in \eqref{eq:direct_regulator} was previously derived in \cite[Proposition 1]{upadhyaya2025automated} Kronecker structure. We connected this equation to regulation theory, and generalized this construction to networked optimization algorithms.

{Let us} relate the  regulator equations in \eqref{eq:direct_regulator} to existing results in the optimization literature. The work in \cite{upadhyaya2025automated} considers the Kronecker-structured setting with $N^a \otimes I_c = N$.
An algorithm $F \star \Kc$ satisfies the  Fixed Point Encoding property \cite{ryu2020uniqueness} for some operator $F$ if every fixed point $(x^*, w^*, z^*)$ of the $F \star \Kc$ satisfies the consensus and optimality properties ($\1_s^\top  \otimes w^* = 0, N^\top z^* = 0)$. 
A  system $\Kc$ with Kronecker-structured representation $(\Ac^a, \Bc^a, \Cc^a, \Dc^a)$  satisfies the Fixed Point Encoding property of \cite[Definition 1]{upadhyaya2025automated} (based on prior work in \cite{ryu2020uniqueness}) if the following containments hold \cite[Proposition 1]{upadhyaya2025automated}:
\begin{subequations}
\label{eq:enconding_both}
\begin{align}
        \mathrm{ran}\mat{cc}{
            \Bc^a N^a & \0 \\ \Dc^a N^a & -\1_s
        } &\subseteq \mathrm{ran} \mat{c}{
            I-\Ac^a \\ -\Cc^a},
        \label{eq:encoding_range} \\
        \mathrm{null}\mat{cc}{
            \Ac^a - I & \Bc^a}
        &\subseteq \mathrm{null}\mat{cc}{
            (N^a)^\top \Cc^a & (N^a)^\top \Dc^a \\
            0 & \1^\top_s}.
            \label{eq:encoding_null}             \end{align}
\end{subequations}
Algorithm convergence in
\cite{upadhyaya2025automated} is based on {the verification that $(\Ac^a, \Bc^a, \Cc^a, \Dc^a)$} satisfies \eqref{eq:encoding_range} and \eqref{eq:encoding_null},  and $F \star \Kc$ obeys a   Lyapunov stability property based on interpolation conditions \cite{drori2014performance}. The Fixed Point Encoding is a valid property for composite optimization problems nonunique pairs $(\beta^*, w^*)$ such as convex problems, while our notion of $\mathcal{O}_{m, L}$ requires unique pairs $(\beta^*, w^*)$.

\begin{proposition}
If a Kronecker-structured controller $\Kc$ with satisfies the regulator equation in \eqref{eq:direct_regulator} and {$(\Ac-I\ \Bc)$ has  full row rank (e.g. if $(\Ac,\Bc,\Cc,\Dc)$ is a minimal realization)}, then both Fixed Point Encoding conditions in \eqref{eq:enconding_both} are satisfied.
\label{prop:fixed_point_encoding}
\end{proposition}
\begin{proof}
  See Appendix \ref{app:fixed_point_encoding}

\end{proof}

\section{Structure of Optimization Algorithms}
\label{sec:structure}

Controllers $\Kc$ forming convergent and well-posed 
optimization algorithms $F \star (P \star \Kc)$ satisfy structural constraints arising from the regulator equation \eqref{eq:nominal_regulation_control_sys_control}. These controllers $\Kc$  admit a factorization into a network-dependent internal model and a \cn{core subcontroller} carrying the algorithm parameters. The internal model decomposition is visualized in  Fig.~\ref{fig:internal_model_factor}.

\begin{figure}[!h]
    \centering
    \begin{tikzpicture}[xscale=1,yscale=1,baseline=(ko1)]
\def\dl{2*\dn}
\def\ds{3*\dn}

\node[sy3] (N) at (0, 0) {
$P$
};

\node[sy3, below right = \dl and 0*\dl of N] (M)  {
$\Sigma_{\min}$
};

\tio{iM1}{M}{east}{1/2};
\tio{oM1}{M}{west}{1/2};

\node[sy3, above =1*\dl of N] (LM)  {
$
F
$
};

\tio{iLM2}{LM}{east}{1/2};
\tio{oLM2}{LM}{west}{1/2};

\tio{ki1}{N}{east}{2/3};
\tio{ko1}{N}{west}{2/3};
\tio{ki2}{N}{east}{1/3};
\tio{ko2}{N}{west}{1/3};

\node[sy3, left=2*\dl of M] (kc) {
$\Sco$
};

\tio{kci1}{kc}{west}{1/2};
\tio{kco1}{kc}{east}{1/2};

\draw[<-] (ko1)--   ($(ko1) + (-1*\ds, 0)$)  |- node[pos=.25]{$w$} (oLM2) ;
\draw[->] (ki1)--  ($(ki1) + (1*\ds, 0)$) |- node[pos=.25,swap]{$z$} (iLM2);

\draw[->] (ko2)--   ($(ko2) + (-4*\ds, 0)$)  |- node[pos=.25, swap]{$y$} (kci1) ;
\draw[<-] (ki2)--  ($(ki2) + (4*\ds, 0)$) |- node[pos=.25]{$u$} (iM1);

\draw[->]    (kco1) -- node[pos=.5]{} (oM1);

\node[
    draw=green!60!black,
    line width=.3ex,
    rounded corners,
    dashed,
    fit=(M)(kc),
    inner xsep=0.8em,
    inner ysep=0.4em
] (kbox) {};

\node[below =0.2*\dl of kbox]{$\Kc$};

\end{tikzpicture}
    \caption{Factorization of 
    $F \star (P \star \Kc)$ into $F \star (P \star (\Sigma_{\min} \star \Sco))$.}
    \label{fig:internal_model_factor}
    \vspace{-3mm}
\end{figure}

\subsection{Internal Model Factorizations}

The internal model factorization is developed under the  additional assumptions.
\begin{assum}
    \label{assum:regulator}
    The matrix $\mat{cc}{A-I & B_2 \\ C_1 & D_{12}}$ has full column rank.
\end{assum}

    \begin{assum}
        \label{assum:disturbance_detec}
        The pair $\left(\mat{cc}{A^\bm & B^\bm_1 \\0 & I_{cs}}, \mat{cc}{ C^\bm_2 & D_{21}^\bm}\right)$ is detectable.
    \end{assum}

Assumption \ref{assum:regulator} ensures that the regulator equation in  \eqref{eq:nominal_regulation_control_sys_plant} has at most one solution $(\Pi, \Gamma, \Phi)$.
Assumption \ref{assum:disturbance_detec} is standard {in} linear regulation theory, and serves to ensure the existence of an internal-model-based controller $\Kc$ such that $F^t \star (P \star \Kc)$ is a convergent optimization algorithm for the test quadratics $F^t$.
 In the case of a direct interconnection \eqref{eq:direct_interconnection}, the {condition} $\sum_{i=1}^s m_i \neq 0$
 (Assumption 1) is necessary and sufficient for Assumption \ref{assum:disturbance_detec} to be satisfied.

\begin{theorem}[Algorithm Structure]
Under Assumptions \ref{assum:well_posed}-\ref{assum:disturbance_detec}, if $\Kc$ has a minimal {realization} $(\Ac, \Bc, \Cc, \Dc)$ and  $F \star (P \star \Kc)$ is convergent for all $F \in \mathcal{O}_{m, L}$ as certified by Theorem \ref{thm:convergence}, then there {exist}
\begin{enumerate}
    \item an internal model system $\Sigma_{{\text{min}}}$ that depends only on $(\Pi, \Gamma, \Phi)$,
    \item and a \cn{core subcontroller} $\cn{\Sigma_{\text{core}}}$ with its state-space matrices lying in a $\Theta$-parameterized affine space $\cn{\mathcal{L}_{\text{core}}}(\Theta)$
\end{enumerate}
{such that the controller $\Kc$ may be} factored into a cascade $\Kc = \Sigma_{\text{min}} \cn{\Sigma_{\text{core}}}$.
    \label{thm:main_structure}
\end{theorem}
\begin{proof}

By Assumption \ref{assum:regulator}, the solution $(\Pi, \Gamma, \Phi)$ {of} \eqref{eq:nominal_regulation_control_sys_plant} is unique.
{We introduce} an invertible coordinate-change matrix $R \in \R^{sc \times sc}$ with the partitioning $R = (R_1 \ R_2)$ such that $\text{ran}(R_1) = \text{null}(\Phi)$. We partition accordingly
\begin{align}
    \Pi R &= \mat{ccc}{\Pi_1 & \Pi_2},  & 
    \Gamma R &= \mat{ccc}{\Gamma_1 & \Gamma_2 }, & \Phi R &= \mat{ ccc}{0& \Phi_2}, &  \Theta R &= \mat{cc}{\Theta_1 & \Theta_2}.
    \label{eq:basis_transform}
\end{align}

Right-multiplying both sides of the regulator equation in \eqref{eq:nominal_regulation_control_sys_control} by  $R$ yields
    \begin{align}
\mat{ccc}{\Ac&\Bc\\ \Cc&\Dc}
\mat{ccc}{\Theta_1 & \Theta_2  \\0 & \Phi_2 }  &= \mat{ccc}{\Theta_1 & \Theta_2\\ \Gamma_1 & \Gamma_2} \label{eq:nominal_regulation_control_sys_trans_control}. 
    \end{align}
\eqref{eq:nominal_regulation_control_sys_control} under Assumptions \ref{assum:regulator} obeys $\text{rank}(\Theta_1) = \text{dim}(\text{null}(\Phi))$ by Lemmas \ref{lem:disjoint_null} and  \ref{lem:theta_injective}.
It is then possible to choose an  invertible matrix $Q$ to transform the controller state as $\xi \rightarrow Q \xi$  such that
\begin{align}
    Q^{-1} \Theta R &= \mat{cc}{-I_r & \Theta_{12} \\ 0 & \Theta_{22}}. \label{eq:theta_12_22}
\end{align}
Applying $Q$ as a similarity transformation to the representation of $\Kc$ in \eqref{eq:nominal_regulation_control_sys_trans_control} leads to
\begin{align}
\mat{cc}{Q^{-1}\Ac Q& Q^{-1}\Bc \\ \Cc Q&\Dc}
\mat{cc}{-I_r & \Theta_{12}    \\ 0 & \Theta_{22}  \hdl 0 & \Phi_2 }  &=  \mat{cc}{-I_r & \Theta_{12} \\ 0 & \Theta_{22}   \hdl \Gamma_1 & \Gamma_2  }. \label{eq:theta_equate_orig}
\end{align}
{Now} we partition the matrices of the {transformed controller representation} according
to {\eqref{eq:nominal_regulation_control_sys_trans_control}} to get
\begin{align}
    \mat{cc|c}{\Ac_{11}'&\Ac_{12}'&\Bc_1'\\\Ac_{21}'&\Ac'_{22}&\Bc'_2\hl \Cc'_1&\Cc'_2&\Dc'} &:= \mat{c|c} {Q^{-1}\Ac Q& Q^{-1} \Bc \hl \Cc Q&\Dc}.\label{eq:matrix_Q_partition}
\end{align}

Substitution of \eqref{eq:matrix_Q_partition} into \eqref{eq:theta_equate_orig} leads to
the {following structural constraints:}
\begin{align}
   \mat{cc|c}{\Ac_{11}'\\\Ac_{21}'\hl \Cc_1'} &=  \mat{cc|c}{I_r  \\
     0 \hl -\Gamma_1 },  &   \mat{cc}{\Ac_{12}'&\Bc_1'\\ \Ac_{22}'&\Bc_2'\hl \Cc_2'&\Dc'} \mat{cc}{\Theta_{22}   \\ \Phi_2 }  &=  \mat{cc }{ 0 \\  \Theta_{22}    \\ \Gamma_1 \Theta_{12} + \Gamma_2  }. \label{eq:structural_constraint}
\end{align}
The controller structure in \eqref{eq:structural_constraint} can be recognized as a cascade between an internal model subsystem $\Sigma_{\min}$ with internal dynamics $I_r$ and a subcontroller $\cn{\cn{\Sigma_{\mathrm{core}}}}$
\begin{align}
  \Sigma_{\min}: \qquad & \mas{c|cc}{
I_r    &I&0 \hl
-\Gamma_1  &0&I
}, & \text{and} & &  \cn{\Sigma_{\mathrm{core}}}: \qquad  & \mas{c|cc}{
\Ac_{22}' &\Bc_2'  \hl
\Ac_{12}' &\Bc_1' \\
\Cc_2'    &\Dc'
 } = \mas{c|cc}{\Aco & \Bco \hl
\Cco_1 & \Dco_1  \\ \Cco_2 & \Dco_2 }.\label{eq:sys_factorize}
\end{align}

The affine constraint $\cn{\mathcal{L}_{\text{core}}}(\Theta)$ imposed on the matrices of the core subcontroller is
\begin{align}
  \label{eq:structural_constraints}
\cn{\mathcal{L}_{\text{core}}}(\Theta): & & \mat{cc}{
   \Aco&\Bco \\
    \Cco_1&\Dco_1 \\
    \Cco_{2}&\Dco_2} \mat{c}{\Theta_{22}   \\ \Phi_2 }  &=  \mat{c }{ 0 \\  \Theta_{22}   \\ \Gamma_1 \Theta_{12} + \Gamma_2  },  
\end{align}
in which the notation $\cn{\mathcal{L}_\text{core}}(\Theta)$ refers to the dependence on $(\Theta_{12}, \Theta_{22})$ from \eqref{eq:theta_12_22}.
The   cascade {interconnection} $\Kc = \Sigma_{\min}  \cn{\Sigma_{\text{core}}}$
{hence admits the representation}
\begin{align}
      \mas{c|c}{\Ac&\Bc\hl \Cc&\Dc} = \mas{c|c}{Q^{-1}\Ac Q &Q^{-1}\Bc \hl  \Cc Q & \Dc} &= \mas{c|c:cc}{
I_r     &I_r&0 \hl
-\Gamma_1  &0&I_{n_u}
}  \mas{c|cc}{\Aco & \Bco \hl
\Cco_1 & \Dco_1  \\ \Cco_2 & \Dco_2 }. \label{eq:controller_factorization_minimal}
\end{align}

\end{proof}

The controller factorization in \eqref{eq:controller_factorization_minimal} is visualized in Figure \ref{fig:internal_model_general}.

\begin{figure}[h]
\centering
\begin{tikzpicture}[xscale=1,yscale=1,baseline=(ko1)]
\def\dl{2*\dn}
\def\ds{3*\dn}

\node[sy3] at (0, 0) (g) {
$
\mas{c|cc}{
I_r     &I&0 \hl
-\Gamma_1  &0&I}
$};

\node[sy3,left=4*\dl of g] (k) {
$
\mas{c|cc}{\Aco & \Bco \hl
\Cco_1 & \Dco_1  \\ \Cco_2 & \Dco_2 }
$
};

\node[sy3, left=8*\dl of k] (g0) {
$
\mas{c|c}{
\Ac &  \Bc \hl 
\Cc &  \Dc}
$
};

\tio{g0i1}{g0}{west}{1/2};
\tio{g0o1}{g0}{east}{1/2};

\draw[->] (g0o1)--  node[pos=.5, swap]{$u$}   ($(g0o1) + (2*\ds, 0)$)  ;
\draw[<-] (g0i1)--   node[pos=.5]{$y$}  ($(g0i1) + (-2*\ds, 0)$);

\tio{ki1}{k}{west}{1/2};
\tio{ko1}{k}{east}{1/2};

\tio{i2}{g}{west}{1/2};
\tio{o2}{g}{east}{1/2};
\draw[->] (o2)--  node[pos=.5, swap]{$u$}   ($(o2) + (2*\ds, 0)$)  ;
\draw[<-] (ki1)--   node[pos=.5]{$y$}  ($(ki1) + (-2*\ds, 0)$);
\draw[->] (ko1)--   node[pos=.5, swap]{$\tilde{u}$} (i2);

\end{tikzpicture}
    \caption{Factorization of {$\Kc$ (left) into an internal model and a core controller} (right)} %

    \label{fig:internal_model_general}
    \vspace{-3mm}
\end{figure}
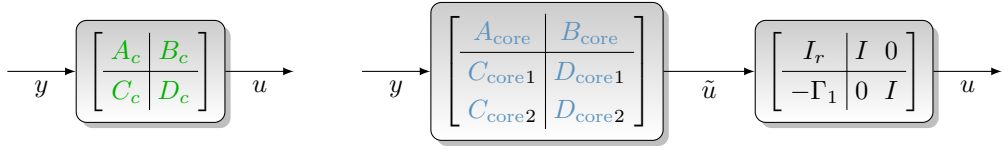

    The structural results in Theorem \ref{thm:main_structure} apply to algorithms that are convergent for the class $\mathcal{O}_{m, L}$ of $(m, L)$-bounded strongly convex composite optimization problems with unique $(\beta^*, w^*)$ pairs. Algorithms that are convergent for broader classes of problems than $\mathcal{O}_{m, L}$, such as convex optimization problems with multiple optimal solutions,  must  also satisfy these properties.

\subsection{Internal Model Structures with Direct Interconnections}

We now explore  structures of optimization algorithms under  a direct interconnection in \eqref{eq:direct_interconnection} ({such that the algorithm reads as $F\star G=F\star\Kc$ without network dynamics}), and demonstrate factorizations of existing optimization algorithms for composite optimization. The matrices $\Gamma = -\mat{cc}{\1_s \otimes I_c & 0}$ and $\Phi = \mat{cc}{0 & N}$ from \eqref{eq:gam_phi} with  $r = \text{dim}(\text{null}(\Phi))   = c$ and  form the internal model and the affine constraint
\vspace{-2mm}
\begin{align}
       \label{eq:core_integrator}
       \Sigma_{\text{min}}: \; & 
       \setlength{\arraycolsep}{2pt}
       \mas{c|cc}{
1     &1 & 0 \hl
\1_s  &0 & I_s}  \!\otimes\! I_c, & 
\cn{\mathcal{L}_{\text{core}}}(\Theta):
\setlength{\arraycolsep}{2pt}
\mat{cc}{
   \Aco&\Bco \\
    \Cco_1&\Dco_1 \\
    \Cco_{2}&\Dco_2} 
    \! \!
    \setlength{\arraycolsep}{1pt}
    \mat{c}{\Theta_{22} \\  N} 
    \!=\! \mat{c}{
    \Theta_{22} \\
    0 \\
    -(\1_s \!\otimes\! I_c) \Theta_{12} }.
\end{align}

The integrator cascade in \eqref{eq:core_integrator} is visualized in Figure \ref{fig:integrator_at_the_output}.
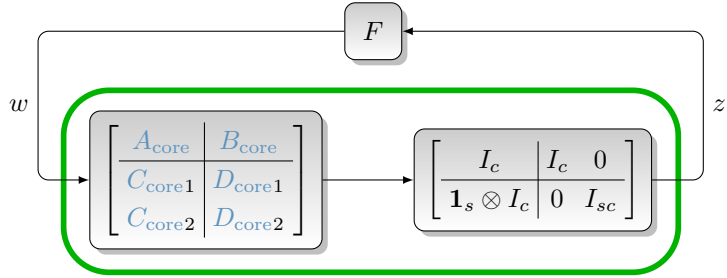
\begin{figure}[h]
    \centering
    \begin{tikzpicture}[xscale=1,yscale=1,baseline=(ko1)]
\def\dl{2*\dn}
\def\ds{3*\dn}

\node[sy3] (F3) at (0, 0) {
$
F
$};

\node[sy3, below right =2*\dl and 0.5*\dl of F3] (g) {
$
\mas{c|cc}{
I_c     &I_c & 0 \hl
\1_s \otimes I_c  &0 & I_{sc}}  
$};

\node[sy3,left=4*\dl of g] (k) {
$
\mas{c|c}{
\Aco &  \Bco \hl 
\Cco_1&  \Dco_1 \\ 
\Cco_2 &  \Dco_2}
$
};

\tio{ki1}{k}{west}{1/2};
\tio{ko1}{k}{east}{1/2};

\tio{iF}{F3}{east}{1/2};
\tio{oF}{F3}{west}{1/2};
\draw[<-] (ki1)--   ($(ki1) + (-1.5*\ds, 0)$)   |- node[pos=.25]{$w$} (oF);

\tio{i2}{g}{west}{1/2};
\tio{o2}{g}{east}{1/2};
\draw[->] (o2)--   ($(o2) + (1.5*\ds, 0)$)   |- node[pos=.25, swap]{$z$} (iF);

\draw[->] (ko1)--   node[pos=.5]{} (i2);

\node[
    draw=green!60!black,
    line width=.3ex,
    rounded corners,
    dashed,
    fit=(k)(g),
    inner xsep=0.8em,
    inner ysep=0.5em
] (kbox) {};

\end{tikzpicture}
    \caption{Decomposition of an algorithm with a direct interconnection}
    \label{fig:integrator_at_the_output}
\end{figure}

 Existing algorithms can be decomposed using our structural results.

\begin{example}
\label{exmp:factorize_existing}
    We perform factorizations of the Chambolle-Pock \cite{chambolle2011first}, Fast Iterative Shrinkage-Thresholding Algorithm (FISTA) \cite{beck2009fast},
     and Forward-Reflected-Backward splitting (FRBS) \cite{malitsky2020forward} algorithms.
We list the Kronecker-structured original system $G^a$ and the core subcontroller $\Sco^a$.

    The  Chambolle-Pock algorithm \cite{chambolle2011first} with parameters $(\tau_1, \tau_2, \theta)$ as
    \vspace{-5mm}
    \begin{align*}
     \\ \nonumber
    G^a&\!=\! 
    \setlength{\arraycolsep}{1.5pt}\mas{cc|cc}{
        1 & -\tau_1 & -\tau_1 & 0 \\
        0 & 0 & 0 & 1 \hl
        1 & -\tau_1 & -\tau_1 & 0 \\
        1 & \frac{1}{\tau_2} \!-\! \tau_1(1\!+\!\theta) & -\tau_1(1\!+\!\theta) & -\frac{1}{\tau_2}
    }, &
    \cn{\Sigma_{core}^a}&\!=\!
    \setlength{\arraycolsep}{1.5pt}
    \mas{c|cc}{
        0 & 0 & 1 \hl
        -\tau_1 & -\tau_1 & 0 \hdl
        -\tau_1 & -\tau_1 & 0 \\
        \frac{1}{\tau_2}\!-\!\tau_1(1\!+\!\theta) & -\tau_1(1\!+\!\theta) & -\frac{1}{\tau_2}
    },
    \intertext{Fast Iterative Shrinkage-Thresholding Algorithm (FISTA) \cite{beck2009fast} with parameters $(q, \lambda)$ decomposes as }
   G^a&=\mas{cc|cc}{
       \frac{2}{1+\sqrt{q}} & -\frac{1-\sqrt{q}}{1+\sqrt{q}}& -\lambda & -\lambda\\
       1 & 0 & 0 & 0 \hl
       \frac{2}{1+\sqrt{q}} & -\frac{1-\sqrt{q}}{1+\sqrt{q}}& 0 & 0 \\
       \frac{2}{1+\sqrt{q}} & -\frac{1-\sqrt{q}}{1+\sqrt{q}}& \lambda & \lambda
    }, \nonumber & \cn{\Sigma_{core}^a}&= \mas{c|cc}{
    -\frac{\sqrt{q}-1}{\sqrt{q}+1} & \lambda & \lambda \hl
    \frac{\sqrt{q}-1}{\sqrt{q}+1} & -\lambda & -\lambda \hdl
    \frac{\sqrt{q}-1}{\sqrt{q}+1} & 0 & 0 \\
    \frac{\sqrt{q}-1}{\sqrt{q}+1} & -\lambda & -\lambda
    }. \nonumber \\
        \intertext{Forward-Reflected-Backward splitting (FRBS) \cite{malitsky2020forward} with parameter $\lambda$ decomposes as }
        G^a&= \mas{ccc|cc}{
        1 & 0 & \lambda & -2\lambda & -\lambda \\
        1 & 0 & 0 & 0 & 0 \\
        0 & 0 & 0 & 1 & 0 \hl
        1 & 0 & 0 & 0 & 0 \\
        1 & 0 & \lambda & -2\lambda & -\lambda \\
    }, &
    \cn{\Sigma_{core}^a}&= \mas{cc|cc}{
        0 & -\lambda & -2\lambda & \lambda \\
        0 & 0 & 1 & 0 \hl
        0 & \lambda & 2\lambda & -\lambda\hdl
        0 & 0 & 0 & 0 \\
        0 & \lambda & -2\lambda & -\lambda
    }.
\end{align*}

In each case, a {state-coordinate change $Q$} exists such that the first column of $\mathcal{C}^a$ is $\1_s$. The similarity transformations and $\Theta$ matrices for each example are
\begin{align}
   Q_{\textrm{Chambolle-Pock}}^a &=  \mat{cc}{1 & 0 \\ 0 & 1}, & Q_{\textrm{FISTA}}^a &=  \mat{cc}{1 & 0 \\ 1 & 1}, & Q_{\textrm{FRBS}}^a &= \mat{ccc}{1 & 0 & 0 \\ 1 & 1 & 0 \\ 0 & 0 & 1}, \\
     \Theta^a_{\textrm{Chambolle-Pock}} &=  \mat{cc}{1 & 0 \\ 0 & 1}, & \Theta^a_{\textrm{FISTA}} &= \mat{cc}{1  & 0 \\1 &  0},  &  \Theta^a_{\textrm{FRBS}}  &= \mat{cc}{1 & 0  \\ 1 & 0 \\ 0 & -1}.
\end{align}
The controller $\cn{\Sigma_{\text{core}}}$ can then be read from the system matrices after performing the similarity transformation.

\end{example}

\begin{example}

The set of Kronecker-structured optimization algorithms with $s$ states and block-lower-triangular sparsity patterns is completely characterized by a three-parameter family $(b_0, b_{11}, b_{21}) \in \R^3$. These three parameters define a family of subcontrollers $\Sco$, with regulator equation certificate  $ \Theta_{12} = \mat{c:cc}{-1 & -b_{21} & 0_{1 \times s-1}} \otimes I_c.$ The internal model factorization of this class of algorithms is  visualized in Figure \ref{fig:static_cascade}.

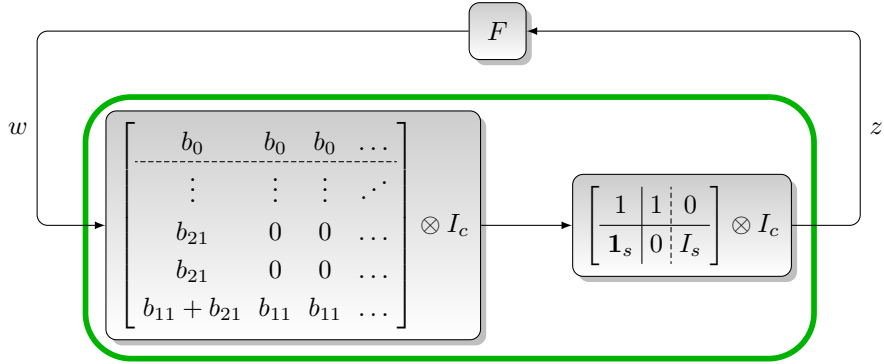
\begin{figure}[h]
    \centering
    \begin{tikzpicture}[xscale=1,yscale=1,baseline=(ko1)]
\def\dl{2*\dn}
\def\ds{3*\dn}

\node[sy3] (F3) at (0, 0) {
$
F
$};

\node[sy3, below right =4*\dl and 2*\dl of F3] (g) {
$
\mas{c|c:c}{
1     &1 & 0 \hl
\1_s &0 & I_{s} }  \otimes I_c  
$};

\node[sy3,left=4*\dl of g] (k) {
$
\mas{cccc}{
     b_0  & b_0 & b_0   & \ldots\hdl
     \vdots & \vdots & \vdots & \ddots  \\
      b_{21}  & 0 & 0  & \hdots \\
     b_{21}  & 0 & 0  & \hdots \\
      b_{11} + b_{21} & b_{11}  & b_{11} & \hdots
    } \otimes I_c  
$
};

\node[
    draw=green!60!black,
    line width=.3ex,
    rounded corners,
    dashed,
    fit=(k)(g),
    inner xsep=0.8em,
    inner ysep=0.5em
] (kbox) {};

\tio{ki1}{k}{west}{1/2};
\tio{ko1}{k}{east}{1/2};

\tio{iF}{F3}{east}{1/2};
\tio{oF}{F3}{west}{1/2};
\draw[<-] (ki1)--   ($(ki1) + (-2*\ds, 0)$)   |- node[pos=.25]{$w$} (oF);

\tio{i2}{g}{west}{1/2};
\tio{o2}{g}{east}{1/2};
\draw[->] (o2)--   ($(o2) + (2*\ds, 0)$)   |- node[pos=.25, swap]{$z$} (iF);

\draw[->] (ko1)--   node[pos=.5]{} (i2);

\end{tikzpicture}
    \caption{Lower-triangular structured algorithm factorization}
    \label{fig:static_cascade}
    \vspace{-5mm}
\end{figure}

Table \ref{tab:basis_parameter} summarizes these {parameters for} existing optimization algorithms. 

\begin{table}[h]
    \centering
    \begin{tabular}{l c c  c c c}
    \toprule
    Algorithm & Parameters &  $s$ & $b_0$ &  $b_{11}$ & $b_{21}$ \\ \midrule      Gradient Descent & $\alpha$ & 1 &  $-\alpha$ & 0 & 0  \\
         Proximal Point Method & $\lambda$ & 1 &  $-\lambda$ & $-\lambda$ & 0  \\
         Douglas Rachford \cite{douglas1956numerical}& $\gamma, \lambda$ & 2 & $-\gamma \lambda$ &  $-\gamma$& $-\gamma$\\
         Forward-Backward \cite{passty1979ergodic} &$\gamma$ &  2 & $-\gamma$ &  $-\gamma$ & 0\\
         Davis-Yin \cite{davis2017three}&  $\gamma, \lambda$ & 3 & $-\gamma \lambda$ &$-\gamma $& $-\gamma$ \\
         \bottomrule
    \end{tabular}
    \caption{Parameters describing static-core algorithms }
    \label{tab:basis_parameter}
    \vspace{-5mm}
\end{table}

\end{example}

\subsection{Information Order Bounds}

A second structural property of optimization algorithms is induced by an interplay between well-posedness and information constraints.
An example includes the requirement that the matrix $\Dcl$ from $G = P \star \Kc$ must be block-lower-triangular, in order to admit  evaluation by the procedure in \eqref{eq:algorithm_causal_eval}. {More generally, we choose} a set $\linfo \subseteq \R^{n_u \times n_y}$ {such that} $\Dc \in \mathcal{L}_{\mathrm{info}}$ implies that the algorithm $F \star (P \star \Kc)$ obeys {a desired information} constraint on $\Dcl$. {
sets $\linfo$, because the expression $\Dcl=D_1 + D_{12} (I - \Dc D_2 )^{-1} \Dc D_{21}$  from  \eqref{cll} is nonlinear in $\Dc$ if $D_2 \neq 0$.}

\begin{theorem}
\label{thm:info_structure}
Under Assumptions \ref{assum:well_posed}-\ref{assum:stab_dect}, consider a network $P$, a function class $\mathcal{O}_{m, L}$, a set $\linfo \subseteq \R^{n_u \times n_y}$ arising from the information constraint, and a given solution $(\Pi, \Gamma, \Phi)$   to \eqref{eq:nominal_regulation_control_sys_plant}.
Define $r_\Gamma := \text{dim}(\text{null}(\Gamma))$ and $r_{\Gamma \Phi} := \text{dim}(\text{null}(\Phi) \cap \text{null}(\Gamma)), $ and let $r_{\mathrm{info}}$ be a lower-bound on the solution {of the optimization problem}
\begin{subequations}
\label{eq:wpc}
\begin{align}
    r_{\mathrm{info}} \leq  \text{min} & \ \mathrm{rank}(\Dc) \\
      \text{s.t.}  & \ \Dc \in \linfo \ \text{such that $F\star(P\star\Kc)$ is well-posed for all
     $F \in \mathcal{O}_{m, L}$}. \label{eq:well_posed_condition}
\end{align}
\label{eq:rank_opt_cont}
\end{subequations}
Let  $\Kc$ be a controller with representation $(\Ac, \Bc, \Cc, \Dc)$ such that $\Dc \in \linfo,$  $F \star (P \star \Kc)$ is a convergent optimization algorithm for all $F\in \mathcal{O}_{m, L}$, and there exists a $\Theta$ such that $(P, \Kc)$ obeys the regulator equation  \eqref{eq:nominal_regulation_control_sys} with solution $(\Pi, \Gamma, \Phi, \Theta)$. Then $\Kc$   has at least 
$r_{\mathrm{info}} + r_\Gamma- r_{\Gamma \Phi} - n_y$ states.
\label{thm:minimal_lifting}
\end{theorem}
\begin{proof}
{If $\Cc$ has $\tau$ columns, the number of states of $\Kc$ with the representation $(\Ac, \Bc, \Cc, \Dc)$ is $\tau$.
Since $\tau\geq\rank{\Cc}\geq \rank{\Cc\Theta}$, it suffices to show 
$\rank{\Cc\Theta}\geq r_{\mathrm{info}} + r_\Gamma- r_{\Gamma \Phi} - n_y$.}
{We exploit the} bottom equation in    \eqref{eq:nominal_regulation_control_sys_control} {which reads} $\Cc \Theta = \Gamma - \Dc \Phi$.
  Let $\Xi = (\Xi_1\ \Xi_2)$ be an invertible matrix such that $\text{ran}(\Xi_2) = \text{null}(\Gamma)$ and {where}   $\Xi_2$ has full column rank. We {infer}
  \vspace{-3mm}
  \begin{subequations}
  \begin{multline}
      \text{rank}(\Cc \Theta) = \text{rank}(\Gamma - \Dc \Phi) \\ =  \text{rank}((\Gamma - \Dc \Phi) \Xi) =  \text{rank}\mat{cc}{\Gamma \Xi_1 - \Dc \Phi \Xi_1, & -\Dc \Phi \Xi_2}
  \end{multline}

\vspace{-3mm}
  
  \noindent leading to
  \vspace{-2mm}
  \begin{multline}
      \text{rank}(\Cc \Theta) \geq \text{rank}(\Dc \Phi \Xi_2)   = \text{rank}(\Xi_2) - \text{dim}(\text{null}(\Phi) \cap \text{ran}(\Xi_2)) \\ - \text{dim}(\text{null}(\Dc) \cap \text{ran}(\Phi \Xi_2)).
  \end{multline}

  \vspace{-2mm}
  
  \noindent Dropping $\text{ran}(\Phi \Xi_2)$ {in the right-most intersection and using} $\text{ran}(\Xi_2) = \text{null}(\Gamma)$ yields
  \begin{align}
      \text{rank}(\Cc \Theta) &\geq  \text{dim}(\text{null}(\Gamma)) - \text{dim}(\text{null}(\Phi) \cap \text{null}(\Gamma)) - \text{dim}(\text{null}(\Dc))\\
       &={r_\Gamma}-r_{{\Gamma\Phi}}- \text{dim}(\text{null}(\Dc)). \notag
  \end{align}
  \vspace{-2mm}
  Noting $\dim(\text{null}(\Dc))=n_y-\rank{\Dc}$, we obtain
  \begin{align}
      \text{rank}(\Cc \Theta) &\geq r_\Gamma- r_{\Gamma \Phi} - n_y + \text{rank}(\Dc) 
      \geq r_{\mathrm{info}} + r_\Gamma - r_{\Gamma \Phi} - n_y.
  \end{align}
  \end{subequations}
\end{proof}

Finding an exact solution to the optimization problem in  \eqref{eq:wpc} may be intractable.
A special case where this is possible involves the direct interconnection \eqref{eq:direct_interconnection} where $\linfo$ enforces block-lower-triangularity of $\Dc$.
By Proposition \ref{prop:invertible}, well-posedness of $F \star \Kc$  for all $F \in \mathcal{O}_{m, L}$ with a block-lower-triangular of $\Dc$ requires that  $[\Dc]_{ii}$  must be invertible for every $i \in \{1, \ldots, s\}$ satisfying $L_i = \infty$. Letting $t := \#(L_i = \infty)$ be the number of possibly nonsmooth oracles in $\mathcal{O}_{m, L}$, any well-posed algorithm $F \star \Kc$ must {satisfy} $\text{rank}(\Dc) \geq tc$.
{Since $(r_\Gamma, r_{\Gamma \Phi}, n_y) = ((s-1)c, 0, sc)$, 
Theorem \ref{thm:minimal_lifting} leads to the lower-bound $(t-1)c$ on the number of states
of $\Kc$.}
This $(t-1)c$ bound is found in  \cite[Theorem 8.1]{morin2024frugal}. The work in \cite{morin2024frugal} builds upon prior {results} in \cite{malitsky2023resolvent}, in which \cite[Theorem 3.3]{malitsky2023resolvent} reports an order bound of $(s-1)c$ for the special case {that} $t=s$.

The case of $t=0$ leads to a valid (but not very useful) lower bound of $-c$ states. Internal Model structures yield controller order lower bounds that remain insightful in the  $t=0$ setting.

\section{Analysis of Optimization Algorithms}

\label{sec:analysis}

In this section, we use LMIs to verify the Robust Stability condition of Theorem \ref{thm:convergence}. Our objective 
is to ensure that {$F \star (P \star \Kc)$} is exponentially convergent for all $F \in\mathcal{O}_{m, L}$.

The alternating  convex method {is based on a full-order internal model $\Sigma_{\textrm{full}}$ such that the  interconnection $\Kc=\Sigma_{\textrm{full}} \star \cf{\Sigma_f}$ for any system $\cf{\Sigma_f}$  will lead to satisfaction
of the regulator equations.} Alternation is performed between the convex problems of searching for a controller $\cf{\Sigma_f}$ to certify convergence, and searching over filter coefficients describing valid dissipation relations satisfied by uncertainties in the class $\mathcal{O}_{m, L}$. The full-order internal model methodology generally yields higher order controllers as compared to structured control, but offers  
convex computational formulation.

\subsection{Exponential Convergence Rates}

The speed of an algorithm may be judged based on exponential convergence rates. {This can} be used 
to analyze the performance of algorithms, and as an objective when synthesizing fast algorithms.
\begin{definition}
\label{defn:rho_convergent}
{Convergence in Definition~\ref{defn:optimization_algorithm} is exponential with rate $\rho \in (0, 1)$ if there exists a $\gamma>0$ such that for all $x_0 \in \R^n,\ F \in \mathcal{O}_{m, L}, k \in \N$, we have}
\begin{align}    
     \max\{\norm{x_k - x^*}_2,\norm{w_k - w^*}_2,\norm{z_k - z^*}_2\} \leq \gamma \rho^{k}\norm{x_0 - x^*}_2.
     \label{eq:rho_convergent_z}
\end{align}
\end{definition}

{Following \cite{scherer2023optimization}, we now sketch how the exponential weighting of signals allows us to develop
guarantees for exponential stability of the algorithmic interconnection in  Theorem~\ref{thm:convergence}.
Given a fixed rate $\rho \in (0, 1)$, the $\rho$-weighting $\bar{x}$ of a sequence $x=(x_k)_{k\in \N}$ is defined as $\bar{x}_k := x_k \rho^{-k}$ for $k\in\N$.
For $\t{F}\in\c{O}_{m,L}^0$, the $\rho$-weighting of the signals $(x^N, \xi, z, w)$ in the interconnection $\tilde{F}\star(P\star\Kc)$ in Theorem~\ref{thm:convergence} leads to
\begin{equation}
\label{eq:algorithm_exp}
\begin{aligned}
    \bar F: & &
    \bar{w}_k &\in \rho^{-k} \tilde{F}(\rho^{k} \bar{z}_k),\\
    \bar P: & &
    \mat{c}{
    \bar{x}^N_{k+1}  \hl \bar{z}_k \hdl \bar{y}_k }
    &=
    \mat{c|c:c}{ \rho^{-1} A &  \rho^{-1} B_1 & \rho^{-1} B_2 \hl
    C_1 & D_1   & D_{12} \hdl
    C_2 & D_{21}  & D_2 }\mat{c}{
    \bar{x}^N_k \hl \bar{w}_k  \hdl  \bar{u}_k},\\
    \bar\Kc: & &
    \mat{c}{\bar \xi_{k+1}  \hl \bar u_k}
    &=
    \mat{c|c}{\rho^{-1}\Ac & \rho^{-1}\Bc \hl\Cc & \Dc}
    \mat{c}{\xi_k \hl y_k}.
\end{aligned}
\end{equation}
If $G=P\star \Kc$ is described by $(\Acl,\Bcl,\Ccl,\Dcl)$, then the system 
$\bar G:=\bar P\star \bar\Kc$ admits a description with $(\rho^{-1}\Acl,\rho^{-1}\Bcl,\Ccl,\Dcl)$.
Hence, $\bar F\star(\bar P\star\bar\Kc)=\bar F\star \bar G$ is 
described by
 \begin{equation}
\label{eq:G_exp}
   \mat{c}{
        \bar x_{k+1}  \hl \bar z_k
    } = \mat{c|c}{ \rho^{-1}\mathcal{A} & \rho^{-1}\mathcal{B} \hl
    \mathcal{C} &  \mathcal{D}}\mat{c}{
        \bar x_k \hl \bar w_k
    },  \qquad   \bar{w}_k \in \rho^{-k} \tilde{F}(\rho^{k} \bar{z}_k),
\end{equation}
where $\bar x$ is defined as $x_k:= \col(\bar{x}_k^N,\bar{\xi}_k)$ for $k\in\N$.
Note that \r{eq:G_exp} emerges from $\t{F}\star G$ 
by $\rho$-weighting the signals $(x,w,z)$,
and both have the fixed point $(0,0,0)$.

We say that \r{eq:G_exp} is  Lyapunov stable if is well-posed and satisfies
\begin{align}
     \exists \gamma > 0, \quad  \forall \bar x_0 \in \R^n,\ k \in \N: & &
     \max\{\norm{\bar x_k}_2,\norm{\bar w_k}_2,\norm{\bar z_k}_2\} \leq \gamma \norm{\bar x_0}_2,
\end{align}
implying that $\t{F}\star G$ is well-posed and satisfies the exponential convergence property:
\begin{align}
     \exists \gamma > 0, \quad  \forall x_0 \in \R^n,\ k \in \N: & &
     \max\{\norm{x_k}_2,\norm{w_k}_2,\norm{z_k}_2\} \leq \gamma \rho^k\norm{x_0}_2.
\end{align}

The following variant of Theorem \ref{thm:convergence} guarantees such property. %

  \begin{corollary}
  \label{cor:exp_stable} Suppose that $I-D_2\Dc$ is invertible, that \r{eq:G_exp} resulting from the $\rho$-weighted interconnection \eqref{eq:algorithm_exp} is Lyapunov stable for all
  $\tilde{F} \in \mathcal{O}_{m, L}^0$, and that
  the regulator equation \eqref{eq:nominal_regulation_control_sys} has a solution.
  Then $F \star (P \star \Kc)$ described by $w_k\in F(z_k)$ and \eqref{eq:PK} is well-posed and
  $\rho$-exponentially convergent for all $F\in\mathcal{O}_{m, L}$.
  \end{corollary}

\subsection{Analysis Program}\label{sec:analysis_spec}

In extension of \cite{scherer2025tutorial} for $s=1$, the goal is to construct an LMI-based method to guarantee that \r{eq:G_exp} is Lyapunov stable. We use the matrices
\begin{align}
    \mathbf{m} &:= \text{diag}(m) \otimes I_c, &  \boldsymbol{\sigma} & := \text{diag}(L-m)^{-1} \otimes I_c, &
    \Sigma_{m, L} & := \mat{c:c}{-\boldsymbol{\sigma} & I_{cs} \hdl I_{cs} & \mathbf{m}} \label{eq:diagonal_transform}
\end{align}
and recall that any pair of  input-output sequences of any $\t F\in\c{O}_{m,L}^0$ satisfies the dissipation relations in Lemma~\ref{lem:operator_passive}.
With the corresponding coefficients collected as $\la_i=(\la^i_0,\ldots,\la^i_{\nu_{\max}})$
for $i \in\{1\ldots, s\}$ and $\blam = (\la_1,\ldots,\la_s)$, we can define the filters
\begin{align}\label{fil}
    \Psi^i(\lambda^i)\!=\!\! \mas{cc|c}{0_{1 \times \nu_{\max}-1} & 0 & 1\\
    I_{\nu_{\max}-1} & 0_{\nu_{\max}-1 \times 1} & 0 \hl
    \lambda_{1:\nu_{\max}-1}^i & \lambda_{\nu_{\max}}^i & \lambda_0^i }\!\otimes\! I_c, \ \
    \Psi(\blam)= \text{blkdiag}(\Psi^1(\lambda^1), \ldots, \Psi^s(\lambda^s)),
\end{align}
which are linear systems with initial condition zero.  If $\blam$ satisfies the constraints 
$\sum_{\nu=0}^{\nu_{\max}} \rho^{-\nu} \lambda_\nu > 0,   \quad   \lambda_{\nu}  \leq 0, \ \quad \forall \nu \geq 1, \quad \forall i \in \{1\ldots, s\}$, the conclusion of Lemma~\ref{lem:operator_passive} reads as follows:
$\sum_{k=0}^{T-1} \bar{q}_k^{\top} \bar{r}_k \geq 0$ {holds for any $T>0$, any $\t F\in\c{O}_{m,L}^0$, and any signal} satisfying
$$
\bar r=\Psi(\blam)\bar p,\ \ \mat{c}{\bar p\\\bar w}=\Sigma_{m, L}\mat{c}{\bar q\\\bar z}\te{and}
\bar{w}_k \in \rho^{-k} \tilde{F}(\rho^{k} \bar{z}_k)\te{for}k\in\N.
$$

Following \cite[Section 3.2]{scherer2025tutorial} and if $I-\bm\Dcl$ is invertible,
we construct the filtered interconnection
$G^\Psi(\blam) := \Psi(\blam)(\Sigma_{m, L} \star \bar G)$
as depicted in Figure~\ref{fig:iqc_basic}.  This filtered interconnection  leads to
a description $(\hat{\Acl}, \hat{\Bcl}, \hat{\Ccl}(\blam), \hat{\Dcl}(\blam))$
in which $\hat{\Ccl}(\blam)$ and $\hat{\Dcl}(\blam)$ are affine in $\blam$. We then arrive at the
following extension of \cite[Corollary 18]{scherer2025tutorial}.

\begin{figure}[h]
    \centering
    \begin{subfigure}{0.35\linewidth}
    \centering
    \begin{tikzpicture}[xscale=1,yscale=1,baseline=(ko1)]
\def\dl{2*\dn}
\def\ds{3*\dn}
\node[sy3] (g) at (0,0)  {
$
F
$};
\tio{i2}{g}{east}{1/2};
\tio{o2}{g}{west}{1/2};

\node[sy3,below=2*\dl of g] (k) {
$
\mas{c|c}{\Acl & \Bcl \hl \Ccl & \Dcl } 
$
};

\tio{ki1}{k}{east}{1/2};
\tio{ko1}{k}{west}{1/2};
\draw[<-] (ko1)--   ($(ko1) + (-2*\ds, 0)$)  |- node[pos=.25]{$w$} (o2) ;
\draw[->] (ki1)--  ($(ki1) + (2*\ds, 0)$) |- node[pos=.25,swap]{$z$} (i2);
\end{tikzpicture}
    \caption{Algorithm with $\t{F}$ }
    \label{alg:tF}
    \end{subfigure}
\hfill
\begin{subfigure}{0.6\linewidth}
\begin{tikzpicture}[xscale=1,yscale=1,baseline=(ko1)]
\def\dl{2*\dn}
\def\ds{3*\dn}
\node[sy3] (g) at (0,0)  {
$
F^k
$};
\tio{i2}{g}{east}{1/2};
\tio{o2}{g}{west}{1/2};

\node[sy3] (k) at (0, 0) {
$
\mas{c|c}{\rho^{-1} \Acl & \rho^{-1} \Bcl \hl \Ccl & \Dcl } 
$
};

\node[sy3, above =\dl of k] (LM)  {
$
\mat{cc}{-\bsig & I \\ I & \bm } 
$
};

\tio{iLM1}{LM}{east}{2/3};
\tio{oLM1}{LM}{west}{2/3};
\tio{iLM2}{LM}{east}{1/3};
\tio{oLM2}{LM}{west}{1/3};

\node[sy3,  left= 4*\dl of oLM1] (Psi)  {
$\Psi(\blam)$
};
\tio{iPsi}{Psi}{east}{1/2};
\tio{oPsi}{Psi}{west}{1/2};

\tio{ki1}{k}{east}{1/2};
\tio{ko1}{k}{west}{1/2};
\draw[<-] (ko1)--   ($(ko1) + (-1*\ds, 0)$)  |- node[pos=.25]{$\bar{w}$} (oLM2) ;
\draw[->] (ki1)--  ($(ki1) + (1*\ds, 0)$) |- node[pos=.25,swap]{$\bar{z}$} (iLM2);
\draw[<-]    (iLM1) -- node[pos=.5]{$\bar{q}$} ($(iLM1) + (3*\ds, 0)$);
\draw[->]    (oLM1) -- node[pos=.5, swap]{$\bar{p}$} (iPsi);
\draw[->]    (oPsi) -- node[pos=.5, swap]{$\bar{r}$} ($(oPsi) - (3*\ds, 0)$);
\end{tikzpicture}
\caption{Filtered interconnection $G^\Psi(\blam) := \Psi(\blam)(\Sigma_{m, L} \star \bar G)$}
\label{fig:iqc_basic}
\end{subfigure}
\hfill

    \vspace{-2mm}
    \caption{Framework for computational algorithm analysis}
    \label{fig:filter_replacement_analysis}
    \vspace{-5mm}
\end{figure}
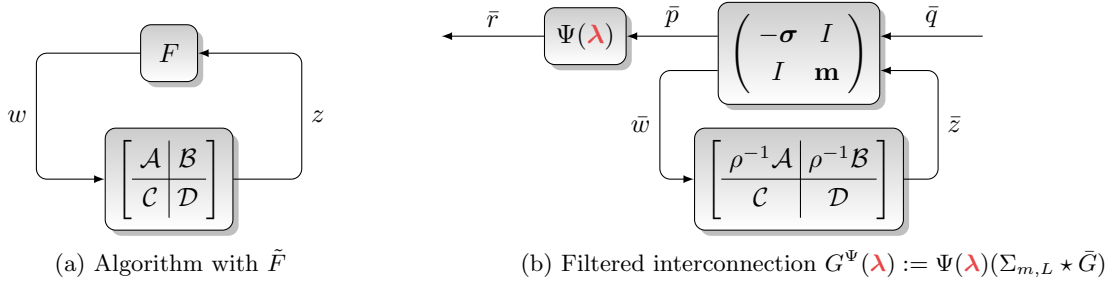

\begin{theorem}
\label{thm:analysis}
    Under Assumptions \ref{assum:well_posed}-\ref{assum:disturbance_detec}, suppose we are given a network $P$, a controller $\Kc$ such that
    $P \star \Kc$ is well posed, a rate $\rho \in (0, 1)$,  and a maximal filter length $\nu_{\max} \in \N$. Then sufficient conditions for the
    interconnection of  $F$ and $(P\star\Kc)$ to be well-posed and $\rho$-exponentially convergent for all $F \in \mathcal{O}_{m, L}$ are if we satisfy all of the following requirements:
    \begin{enumerate}
        \item \textbf{Information Constraint:} $\Dcl$ is block-lower-triangular as in \r{eq:tri_partition}.
        \item \textbf{Regulator Equations:}  The regulator equations in \eqref{eq:nominal_regulation_control_sys} admit a solution.
        \item \textbf{Robust Stability:} With $(\hat{\Acl}, \hat{\Bcl}, \hat{\Ccl}(\blam), \hat{\Dcl}(\blam))$ being a realization of $G^\Psi(\blam) := \Psi(\blam)(\Sigma_{m, L} \star \bar{P} \star \bar{\Kc})$, there exist $\mathcal{M} \succ 0 $ and $\blam$ for which
        \vspace{-3mm}
        \end{enumerate}
        \begin{subequations}
        \label{eq:antipassivity_analysis}
        \begin{align}
     \quad \mat{cc}{\hat{\Acl} & \hat{\Bcl} \\ I & 0}^\top \!\!\!
     \mat{cc}{\mathcal{M} & 0 \\ 0 &-\mathcal{M}}
     \!
     \mat{cc}{\hat{\Acl} & \hat{\Bcl} \\ I & 0}
        \!&+\! \frac{1}{2} \!
        \mat{cc}{\hat{\Ccl}(\blam) & \hat{\Dcl}(\blam) \\ 0 & I }^\top \!\!
        \mat{cc}{0 & I \\ I & 0} \!\!
     \mat{cc}{\hat{\Ccl}(\blam) & \hat{\Dcl}(\blam) \\ 0 & I }
    \!\! \prec \! 0,   \\
     \qquad \qquad \textstyle  \sum_{\nu=0}^{\nu_{\max}} \rho^{-\nu} \lambda^i_\nu > 0,&   \quad   \lambda^i_{\nu}  \leq 0, \ \quad \forall \nu \geq 1, \quad \forall i \in \{1\ldots, s\}. \label{eq:passive_conditions}
    \end{align}
    \end{subequations}
\end{theorem}

\begin{proof} 
See Appendix \ref{app:analysis_proof}.
\end{proof}

The \textit{identity filter} $\Psi(\blam)(\bz) = I$ obtained with $\lambda_0^i = 1$ and $\lambda_{\nu}^i = 0$ for
$\nu \in \{1, \ldots, \nu_{\max}\}$ and $i \in \{1, \ldots, s\}$ always satisfies the condition in \eqref{eq:passive_conditions}. {Due to the KYP Lemma, feasibility of the LMI in \eqref{eq:passive_conditions} can be interpreted as certifying strict negative-realness of the corresponding transfer matrix
\cite{scherer2021convex, scherer2025tutorial}.}

\section{Synthesis of Optimization Algorithms}
\label{sec:synthesis}
\label{sec:synthesis_program}
The formal description of the  algorithm synthesis task is as follows:
\begin{prob}[Algorithm Synthesis]
\label{prob:synthesis}
Given a network $P$ from \eqref{eq:network_dynamics}, an operator tuple $\mathcal{O}_{m, L}$, and a set $\linfo$, find a controller $\Kc$ with representation $(\Ac, \Bc, \Cc, \Dc)$ {where} $\Dc \in  \linfo$ and {such that} the interconnection $F \star (P \star \Kc)$ is a well-posed convergent optimization algorithm for any $F \in \mathcal{O}_{m, L}$.
\end{prob}

In this section we propose a full-order internal model, which 
opens the way for (alternating) convex algorithm design methods.

\begin{definition}
    {For a network $P$ and a solution $(\Pi, \Gamma, \Phi)$ of the plant regulator equation \eqref{eq:nominal_regulation_control_sys_plant},
     a \textbf{full-order} internal model $\Sigma_{\textrm{full}}$ and a \textbf{unstructured} subcontroller $\cf{\Sigma_f}$ are}  defined as
     \vspace{-3mm}
    \begin{align}
          \Sigma_{\textrm{full}}: \qquad & \mas{c|c:cc}{
I_{cs}     &0&I & 0 \hl
-\Gamma  &0&0 & I \hdl
\Phi     &I&0&0} & \text{and} &  &   \cf{\Sigma_{\text{f}}} := \mas{c|c}{\cf{A_{\text{f}}} & \cf{B_{\text{f}}} \hl
        \cf{C_{\text{f}}}^1 & \cf{D_{\text{f}}}^1 \hdl \cf{C_{\text{f}}}^2 & \cf{D_{\text{f}}}^2}. \label{eq:full_order_model}
    \end{align}

\end{definition}

\begin{proposition}
\label{prop:full_order_structure}
{For a network $P$ and a solution $(\Pi, \Gamma, \Phi)$ of the plant regulator equation \eqref{eq:nominal_regulation_control_sys_plant}, the full-order internal model $\Sigma_{\textrm{full}}$ and any subcontroller $\cf{\Sigma_{\text{f}}}$ define a} controller $\Kc = \Sigma_{\textrm{full}} \star \cf{\Sigma_{\text{f}}}$ {with a representation such that} the controller regulator equation \eqref{eq:nominal_regulation_control_sys_control} is satisfied with $\Theta = [-I_{cs}\  0]^\top$.

\end{proposition}
\begin{proof}
{The interconnection $\Kc = \Sigma_{\textrm{full}} \star \cf{\Sigma_f}$ is well-posed and has the realization
\vspace{-1mm}
\begin{align}
    \Kc = \mas{cc|c}{
I_{cs} - \cf{D_f}_1\Phi & \cf{C_f}_1 & \cf{D_f}_1\\ - \cf{B_f}\Phi & \cf{A_f} & \cf{B_f}\hl
-\Gamma  - \cf{D_f}_2\Phi & \cf{C_f}_2 & \cf{D_f}_2}.
\end{align}

\vspace{-2mm}

By inspection, we infer
\begin{align}
    \mat{cc|c}{
I_{cs} - \cf{D_f}_1\Phi & \cf{C_f}_1 & \cf{D_f}_1\\ - \cf{B_f}\Phi & \cf{A_f} & \cf{B_f}\hl
-\Gamma  - \cf{D_f}_2\Phi & \cf{C_f}_2 & \cf{D_f}_2
    }
\mat{c}{-I_{cs}\\0\hl \Phi}  &= \mat{c}{-I_{cs}\\0 \hl  \Gamma}.
\end{align}
}
\end{proof}

\label{sec:full_order_synthesis}
Given a network $P$ and a full-order model $\Sigma_{\textrm{full}}$ from \eqref{eq:full_order_model}, the original plant to be controlled by $\cf{{\Sigma}_{f}}$ is $P \star \Sigma_{\textrm{full}}$.
Figure \ref{fig:iqc_filtering_synth} visualizes the overall system setup for  $\rho$-exponentially weighted synthesis. The goal of full-order synthesis is to choose a {subcontroller} $\cf{{\Sigma}_{f}}$ such that the system formed by interconnection with input $\bar{q}$ and output $\bar{r}$  satisfies the LMI constraints in  \eqref{eq:antipassivity_analysis}.

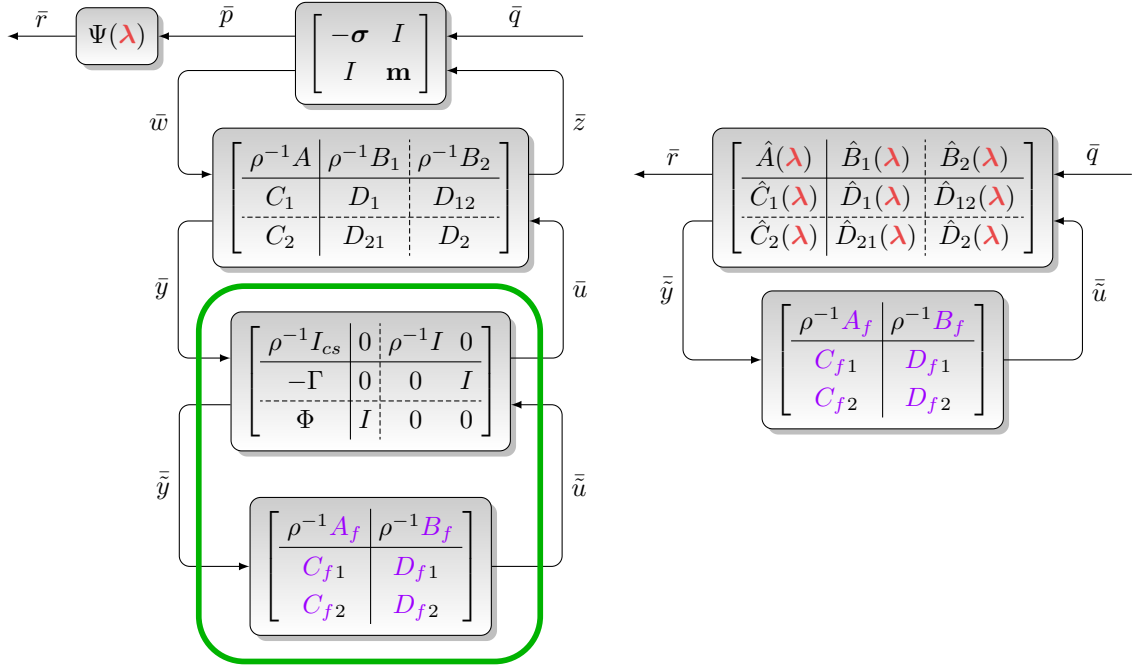
\begin{figure}[h]
\centering
\begin{tikzpicture}[xscale=1,yscale=1,baseline=(ko1)]
\def\dl{2*\dn}
\def\ds{3*\dn}
\node[sy3] (g) at (0,0)  {
$
F^k
$};
\tio{i2}{g}{east}{1/2};
\tio{o2}{g}{west}{1/2};

\node[sy3] (N) at (0, 0) {
$
\setlength{\arraycolsep}{1pt}
\mas{c|c:c}{ \rho^{-1}\!A & \rho^{-1}\! B_1 & \rho^{-1}\! B_2 \hl   
    C_1 & D_1   & D_{12} \hdl
    C_2 & D_{21}  & D_2}
$
};

\node[sy3, below=4*\dl of N] (M) at (0, 0) {
$
\setlength{\arraycolsep}{2pt}
\mas{c|c:cc}{ \rho^{-1}\! I_{cs} & 0  & \rho^{-1}\! I & 0\hl   
    -\Gamma & 0 & 0 &I\hdl
    \Phi & I & 0 & 0}
$
};

\tio{iM1}{M}{east}{2/3};
\tio{oM1}{M}{west}{2/3};
\tio{iM2}{M}{east}{1/3};
\tio{oM2}{M}{west}{1/3};

\node[sy3, above =\dl of N] (LM)  {
$
\mas{cc}{-\bsig & I \\ I & \bm } 
$
};

\tio{iLM1}{LM}{east}{2/3};
\tio{oLM1}{LM}{west}{2/3};
\tio{iLM2}{LM}{east}{1/3};
\tio{oLM2}{LM}{west}{1/3};

\node[sy3,  left= 5*\dl of oLM1] (Psi)  {
$\Psi(\blam)$
};
\tio{iPsi}{Psi}{east}{1/2};
\tio{oPsi}{Psi}{west}{1/2};

\tio{ki1}{N}{east}{2/3};
\tio{ko1}{N}{west}{2/3};
\tio{ki2}{N}{east}{1/3};
\tio{ko2}{N}{west}{1/3};
\draw[<-] (ko1)--   ($(ko1) + (-1*\ds, 0)$)  |- node[pos=.25]{$\bar{w}$} (oLM2) ;
\draw[->] (ki1)--  ($(ki1) + (1*\ds, 0)$) |- node[pos=.25,swap]{$\bar{z}$} (iLM2);

\draw[->] (ko2)--   ($(ko2) + (-1*\ds, 0)$)  |- node[pos=.25, swap]{$\bar{y}$} (oM1) ;
\draw[<-] (ki2)--  ($(ki2) + (1*\ds, 0)$) |- node[pos=.25]{$\bar{u}$} (iM1);

\draw[<-]    (iLM1) -- node[pos=.5]{$\bar{q}$} ($(iLM1) + (3*\ds, 0)$);
\draw[->]    (oLM1) -- node[pos=.5, swap]{$\bar{p}$} (iPsi);
\draw[->]    (oPsi) -- node[pos=.5, swap]{$\bar{r}$} ($(oPsi) - (1*\ds, 0)$);

\node[sy3,below=1*\dl of M] (kc) {
$
\mas{c|cc}{\rho^{-1} \Acf & \rho^{-1} \Bcf \hl
\Ccf_1 & \Dcf_1  \\ \Ccf_2 & \Dcf_2 }
$
};

\tio{kci1}{kc}{east}{1/2};
\tio{kco1}{kc}{west}{1/2};
\draw[->] (oM2)--   ($(oM2) + (-1.5*\ds, 0)$)  |- node[pos=.25, swap]{$\bar{\t y}$} (kco1) ;
\draw[<-] (iM2)--  ($(iM2) + (1.5*\ds, 0)$) |- node[pos=.25]{$\bar{\t u}$} (kci1);

\node[sy3, right=13*\dl of g] (ghat) at (0,0)  {
$
\setlength{\arraycolsep}{1pt}
\mas{c|c:cc}{
\hat{A}(\blam) &  \hat{B}_1(\blam) &  \hat{B}_2(\blam) \hl   
    \hat{C}_1(\blam) & \hat{D}_1(\blam)   & \hat{D}_{12}(\blam) \hdl
    \hat{C}_2(\blam) & \hat{D}_{21}(\blam)  & \hat{D}_2(\blam)}
$};
\tio{i1}{ghat}{east}{2/3};
\tio{o1}{ghat}{west}{2/3};
\tio{i2}{ghat}{east}{1/3};
\tio{o2}{ghat}{west}{1/3};
\draw[->] (o1)--  node[swap]{$\bar{r}$} ($(o1) + (-1.5*\ds, 0)$)  ;
\draw[<-] (i1)-- node[]{$\bar{q}$} ($(i1) + (1.5*\ds, 0)$) ;

\node[sy3,below=\dl of ghat] (khat) {
$
\mas{c|c}{\rho^{-1} \Acf & \rho^{-1}\Bcf \hl
\Ccf_1 & \Dcf_1  \\ \Ccf_2 & \Dcf_2 }
$
};

\tio{ki1}{khat}{east}{1/2};
\tio{ko1}{khat}{west}{1/2};
\draw[<-] (ko1)--   ($(ko1) + (-1.8*\ds, 0)$)  |- node[pos=.25]{$\bar{\t y}$} (o2) ;
\draw[->] (ki1)--  ($(ki1) + (1.8*\ds, 0)$) |- node[pos=.25,swap]{$\bar{\t u}$} (i2);

\node[
    draw=green!60!black,
    line width=.3ex,
    rounded corners,
    dashed,
    fit=(M)(kc),
    inner xsep=0.8em,
    inner ysep=0.5em
] (kbox) {};

\end{tikzpicture}

\caption{Filtering framework for synthesis with full-order models \label{fig:iqc_filtering_synth}}
\vspace{-6mm}
\end{figure}
{To enable a convex design of $\cf{\Sigma_f}$, we introduce the following assumption.
\begin{assum}
    \label{assum:quad_invariance}
    Given a network $P$ and an information constraint set $\linfo$ for $\Dc$, there exists a convex set $\mathcal{L}^D \subseteq \R^{n_u \times n_y}$ with an LMI representation such that
    \begin{align}
        X(I - D_2 X)^{-1} \in \linfo\te{holds for all}X\in \mathcal{L}^D.
    \end{align}
\end{assum}
}

Sufficient conditions for Assumption \ref{assum:quad_invariance} to be respected with $\mathcal{L}^D  := \linfo$ are if $\linfo$ only enforces block-lower-triangularity of $\Dc$ and {$D_2\in\linfo$,} or if  $D_2 = 0$ and $\linfo$ {is a set with an LMI representation.}

The full-order controller synthesis {problem reads as follows.}
{
\begin{prob}[Full-Order Synthesis]
Let $\mathcal{O}_{m, L}$ be a class of operators, $\mathcal{L}^D$ an information structure set derived from $\linfo$, and $P$ a network model such that Assumptions \ref{assum:well_posed}-\ref{assum:quad_invariance} hold.
For a rate $\rho \in (0, 1)$ and with the full-order internal model $\Sigma_{\textrm{full}}$ from \eqref{eq:full_order_model}, find
a controller $\cf{{\Sigma}_{\text{f}}}$  and filter coefficients $\blam$ such that the following hold:
    \begin{enumerate}
        \item The constraint $\cf{D_f}_2 \in \mathcal{L}^D$ is obeyed.
        \item The LMI in \eqref{eq:antipassivity_analysis} is feasible for the system $\hat{G}(\blam) =  \Psi(\blam)(\Sigma_{m, L} \star \bar{P} \star (\bar{\Sigma}_{\textrm{full}} \star \cf{\bar{\Sigma}_{\text{f}}})).$
    \end{enumerate}
    \label{prob:synth_step}
\end{prob}
}

{For fixed filter coefficients $\blam$, the design of $\cf{{\Sigma}_{\text{f}}}$ 
can be convexified,
cf. Appendix \ref{app:synthesis}.
For a fixed controller $\cf{{\Sigma}_{\text{f}}}$, the search over $\blam$ is convex. This permits
the minimization of $\rho$ by the alternating solution of convex programs for ($\blam$, $\cf{\Sigma_f}$) and bisection in $\rho$.

\begin{proposition}
\label{prop:converge_full}
For any solution of Problem \ref{prob:synth_step}, the controller
$\Kc:= \Sigma_{\textrm{full}} \star \cf{{\Sigma}_{\text{f}}}$
leads to a $\rho$-convergent well-posed algorithm $F\star (P\star\Kc)$ 
for
all $F \in \mathcal{O}_{m, L}$.
\end{proposition}

\begin{proof}
Due to Proposition \ref{prop:full_order_structure}, the choice of $\Kc$ guarantees that the regulator equations
\eqref{eq:nominal_regulation_control_sys} admit a solution. Then the claim follows
by applying Theorem~\ref{thm:analysis}.
\end{proof}
}

\section{Numerical Examples}

\label{sec:examples}

Code to generate the following examples is publicly available\footnote{\url{https://github.com/jarmill/composite_opt}}. All code was written in MATLAB (R2023b). For {structured synthesis}, the nonconvex optimization problems were solved using \texttt{hinfstruct} with fixed matrices $\Theta_{22}$ \cite{gahinet2011decentralized}. For full-order internal models, the linear matrix inequalities were solved using LMILab \cite{gahinet1993lmilab} from the MATLAB Robust Control Toolbox. Synthesis and controller recovery in the provided experiments always impose Kronecker structure.

\subsection{Two Operator Splitting over an Unstable Network}

We consider a constrained optimization problem
{(Problem~\ref{prob:composite} for $m=(m_1,0)$, $L=(L_1,\infty)$, $0<m_1<L_1<\infty$),
with $f_2 =  I_{\mathcal{Z}}$ for a closed} convex set $\mathcal{Z}\subset \mathbb{R}^c$. The resolvent {$(I + {\gamma \partial f_2})^{-1}$} is the projection operator onto $\mathcal{Z}$. We assume
the following network with an unstable pole $\bz=-1.1$ and impose the following sparsity pattern on $\Dco_2$:
\begin{align}
    P: \ \ &  \mat{cc:cc}{0 & 0 & \frac{1}{\bz+1.1} & 0 \\
    0 & 0 & \frac{1}{\bz-0.3} & 2 \hdl
    1 & \frac{1}{\bz} & \frac{2}{\bz^2} & \frac{1}{\bz - 0.9} \\
    1 & 1 - \frac{1}{\bz^3} & 0 & 1} \otimes I_c, & \linfo: \  \ \mat{cc}{\bullet & 0  \\ \bullet & \bullet }, \label{eq:unstable_network}
\end{align}
in which $\bullet$ denotes possibly nonzero entries. 
This block-sparsity pattern allows for   proximal evaluations of both $\partial f_1$ and $\partial f_2$.

{With $f_1 \in \mathcal{S}_{1,5}$, we} solve Problem \ref{prob:synth_step} to create a Kronecker-structured optimization algorithm. The synthesized controller stabilizes the unstable channel dynamics and achieves a rate $\rho<0.9476$.
Figure \ref{fig:lasso_delay_ugly} plots an execution of the algorithm
if minimizing a strongly convex quadratic function for $c=7$ with the unconstrained optimum $\beta_Q$ and an $L_1$ norm constraint as in
\begin{align}
    \beta^* &= \argmin_{\beta \in \R^{7}} \frac{1}{2}(\beta - \beta_Q)^\top Q(\beta - \beta_Q)  +  \mathcal{I}_{\norm{\cdot}_1 \leq 110}(\beta). \label{eq:lasso}
\end{align}
The matrix {$Q=Q^\top$ is} positive definite with eigenvalues between 1 and 5. 
The top plots of Figure \ref{fig:lasso_delay_ugly} show $z_k, w_k,$ and $x_k$, respectively, over the course of $k=100$ iterations. In the top-left plot, each curve is a trace of $z_k^i$ {for} $k \in \{1, \ldots, 40\}$ in each coordinate ${i} \in \{1, \ldots, 7\}.$ Similar traces {are shown} for the other plots {of} $w_k$ and $x_k$.
The bottom-left plot shows the elementwise {c}onsensus error $\abs{z^i_k - z_{avg, k}}$ for $i\in \{1, 2\}$ with respect to the average $z_{avg, k}:=\frac{1}{2}(z^1_k+z^2_k)$, {the} bottom-center plot displays the elementwise {o}ptimality condition $\abs{w^1 + w^2}$, {while the} bottom-right plot shows the gap {between} the function value {and} the optimal value $f^*$.
\begin{figure}[!h]
    \centering
    \includegraphics[width=0.9\linewidth]{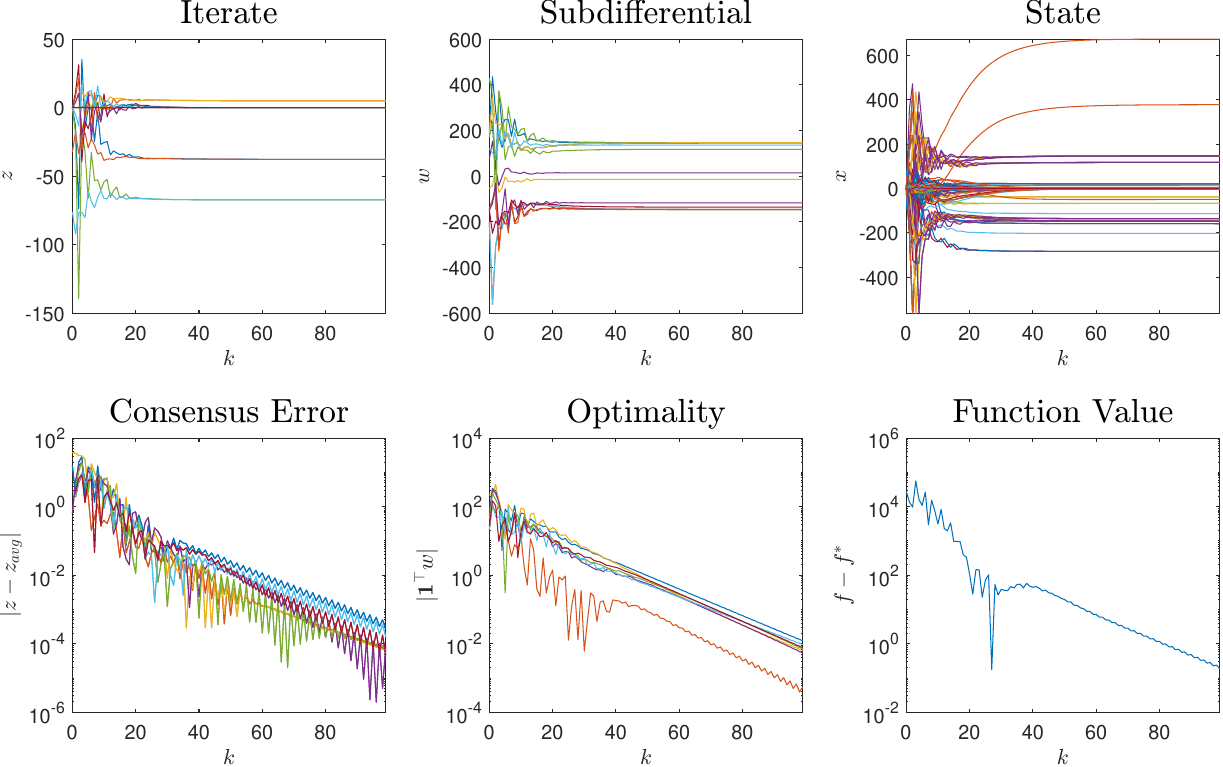}
    \caption{Trajectories of a designed algorithm for Problem \eqref{eq:lasso} with network \eqref{eq:unstable_network}.}
    \label{fig:lasso_delay_ugly}
\end{figure}

\subsection{Case Study: Image Denoising with Transmission Delay}

We provide an example {of} image denoising using total variation regularization \cite{chambolle2004}.
The spatial dimensions of the image are $n_x$ and $n_y$, and the image has $n_c$ color channels ($n_c=3$ for RGB). The values of each pixel are normalized to lie within {$[0,1]$}.
The goal is to recover a clean image $X^\star \in [0, 1]^{n_x \times n_y \times n_c}$ {from a noisy observation $Y = X^\star + \omega$ corrupted by {additive} noise $\omega$.}  
Here, $\mathrm{TV}(X)$  denotes the total variation of the image {defined as}
\vspace{-2mm}
\begin{equation}
    \mathrm{TV}(X) =
\sum_{c=1}^{n_c}
\sum_{i=1}^{n_x}
\sum_{j=1}^{n_y}
\sqrt{
\left(\Delta_x X_{i,j,c} \right)^2
+
\left(\Delta_y X_{i,j,c} \right)^2
}
\end{equation}
involving the differencing operators
\vspace{-2mm}
\begin{align}
(\Delta_x X)_{i,j,c}
\!=\!
\begin{cases}
X_{i+1,j,c}\!-\!X_{i,j,c}, & \!i\!<\!n_x,\\
0, & \!i\!=\!n_x,
\end{cases}
\quad
(\Delta_y X)_{i,j,c}\!=\!
\begin{cases}
X_{i,j+1,c} \!-\! X_{i,j,c}, & \!j\!<\!n_y,\\
0, & \!j\!=\!n_y.
\end{cases}
\end{align}

The total-variation regularizer involves a weight $\alpha_{\text{TV}} \geq 0$. An indicator function $\chi_{[0, 1]}$ is also imposed to ensure that all pixels have normalized intensities between $0$ and $1$.
We formulate the recovery  of the original image $X^\star$  as an optimization problem in the variable $\beta:=\text{vec}(X)$:
\vspace{-2mm}
\begin{equation}\label{eq:denoising_problem}
    \min_{\beta \in \mathbb{R}^{n_x \times  n_y \times n_c}} \, \underbrace{\frac{1}{2}\|\mathrm{vec}(Y) - \beta \|_2^2}_{f_1(\beta)} \, + \, \underbrace{\alpha_{\mathrm{TV}}  \mathrm{TV}(\beta) + \chi_{[0,1]}(\beta)}_{f_2(\beta)}.
\end{equation}

We demonstrate this approach on the denoising of a pepper-garlic test image,  as shown in Figure \ref{fig:denoising_images}. The original image $X^*$ is in Figure \ref{fig:denoising_images:original}. The noise-corrupted image $Y = X^* + \omega$ is in Figure \ref{fig:denoising_images:noisy}, in which each $\omega_{i,j,c}$ is i.i.d. randomly sampled from an normal distribution with mean 0.01 and standard deviation 0.2.  The denoised image $\beta^*$ found by solving \eqref{eq:denoising_problem} with Total Variation regularization parameter $\alpha_{\mathrm{TV}} = 0.15$ is in Figure \ref{fig:denoising_images:tv}.

\begin{figure}[!ht]
    \centering
    \begin{subfigure}{0.3\linewidth}
        \centering
        \includegraphics[width=0.7\textwidth]{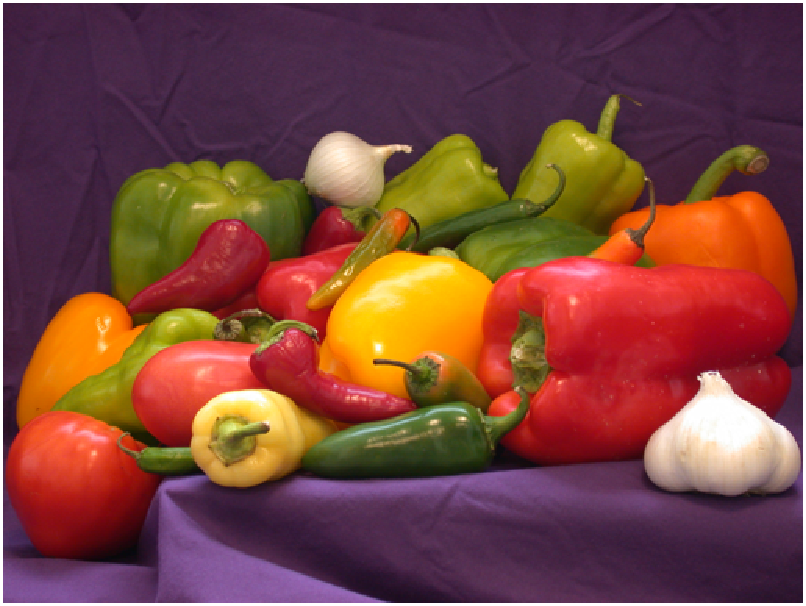}
        \caption{Original Image $X^*$}
        \label{fig:denoising_images:original}
    \end{subfigure}
    \hfill
    \begin{subfigure}{0.3\linewidth}
        \centering
        \includegraphics[width=0.7\textwidth]{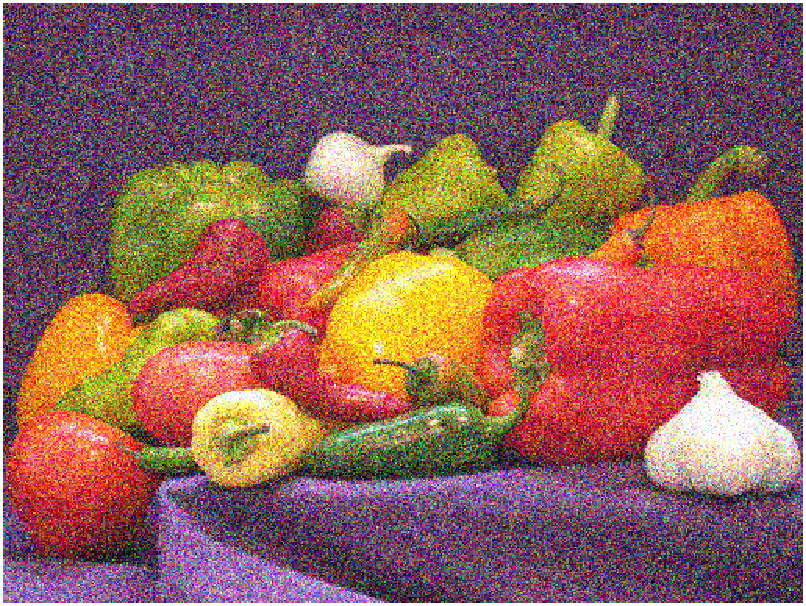}
        \caption{Noisy Image $Y$}
        \label{fig:denoising_images:noisy}
    \end{subfigure}
     \hfill
    \begin{subfigure}{0.3\linewidth}
        \centering
        \includegraphics[width=0.7\textwidth]{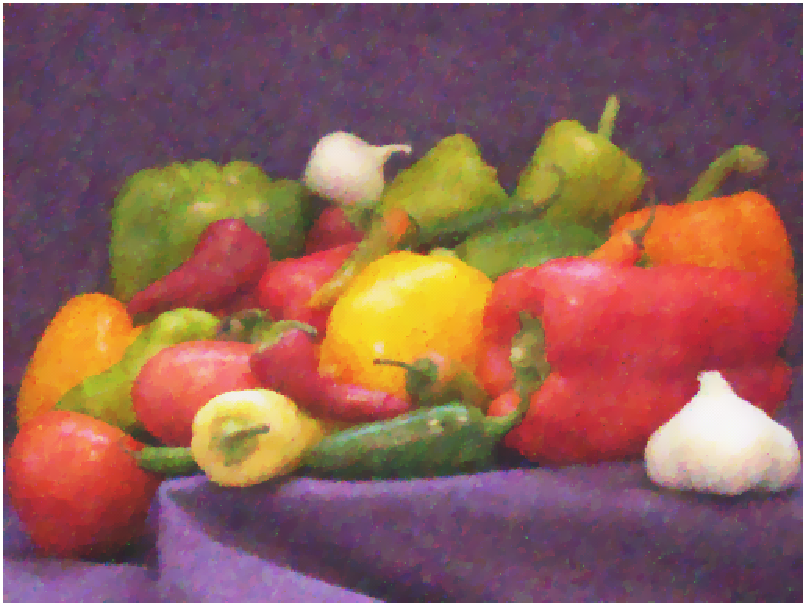}
        \caption{Denoised Image $\beta^*$}
        \label{fig:denoising_images:tv}
    \end{subfigure}
    \vspace{-3mm}
    \caption{Sample pepper-garlic images used for problem \eqref{eq:denoising_problem}.}
    \vspace{-5mm}
    \label{fig:denoising_images}
\end{figure}

We solve the 
denoising problem in \eqref{eq:denoising_problem} using a two-operator splitting algorithm with $f_1 \in \mathcal{S}_{1, 1},$ and $f_2 \in \mathcal{S}_{0,\infty}$. {The function $f_2$ admits}
the proximal operator
    $\mathrm{prox}_{\gamma f_2}(z) = \mathrm{Proj}_{[0,1]}\left[ \mathrm{prox}_{\gamma \mathrm{TV}}(z) \right],$
which follows from adjusting the first-order optimality condition of $\mathrm{prox}_{\gamma \mathrm{TV}}$ from \cite{chambolle2004}. The proximal operator of $\mathrm{TV}$ itself is computed with an inner  iteration given in \cite{chambolle2004}.

We consider the case where the {noisy} image $Y$ is only available remotely, involving a delay of $h\in \N$ steps before and after the evaluation of $\nabla f_1$. These delays are described by the networks 
\vspace{-2mm}
\begin{align}
    P^{h}:\ \ \mat{cc:cc}{0 & 0 & \bz^{-h} & 0 \\
    0 & 0 & 0 & 1 \hdl
    \bz^{-h} & 0 & 0 & 0 \\
    0 & 1 & 0 & 0} \otimes I_c. \label{eq:first_delay_network}
\end{align}
Perfect transmission means to choose $h=0$.
Our reference algorithm is Douglas-Rachford \cite{douglas1956numerical} with parameters $(\gamma, \lambda) = (1, 2)$ and a description of
\begin{align}
    \Kc_{DR} & = \mas{c|cc}{1 & -2 & -2 \hl
    1 & -1 &0 \\
    1 & -2 & -1}. \label{eq:image_denoise_reference_alg}
\end{align}

We first solve (\ref{eq:denoising_problem}) by applying the Douglas-Rachford algorithms
for different values of delays $h \in \{0, \ldots, 5\}$. We then synthesize algorithm for the given delay-$h$ transmission network model using minimal and full-order internal models, and we compare these results to the Douglas-Rachford solution.
The results are summarized in Table~\ref{tab:denoising_rates}.

\begin{table}[h]
    \centering
    \caption{Convergence rates $\rho$ for the denoising problem \eqref{eq:denoising_problem} with different transmission delays $h$ using the Davis-Yin splitting method and our synthesized algorithms (nonconvex structured optimization over parameters of $\Sco$, and convex design using  full-order internal models).}
    \begin{tabular}{cccc}
        \toprule
        \multirow{2}{1.5cm}{Delay $h$} & \multicolumn{3}{c}{Convergence rate $\rho$} \\
         \cmidrule{2-4}
         & Douglas-Rachford   & Syn. Structured synthesis &  Full-order synthesis \\
        \midrule
        $0$ & 0.0134 &  0.01 & 0.064 \\
        $1$ & 2.4149  & 0.5932 & 0.2092 \\
        $2$ &1.4159 &   0.7223 & 0.3483 \\
        $3$ & 1.6184 &   0.7915 & 0.4497 \\
        $4$ & 1.3481 &     0.8365 & 0.5249 \\
        $5$ & 1.4199 &   0.8568 & 0.5827 \\
        \bottomrule
    \end{tabular}
    \label{tab:denoising_rates}
\end{table}

While all methods perform comparably well in the absence of delays, convergence of the {Douglas-Rachford} splitting scheme cannot be certified once delays are introduced. In contrast, our synthesized algorithms (both {of} full and reduced order) retain convergence rates well below one across all tested delay values. We observe an advantage of using full over reduced order algorithms in terms of smaller convergence rates.

Figure \ref{fig:pepper_dr} plots the image iterates $z^2_k$ (input to $\partial f_2$) generated by the Douglas-Rachford scheme as $k$ increases (vertical) and as $h$ increases (horizontal). The denoised image is recovered only at $h=0$ (first column), {while} at all other delays $h$ the Douglas-Rachford algorithm is nonconvergent. Figure \ref{fig:pepper_syn} plots the same iterates $z^2_k$ generated by the synthesized algorithms with full-order internal models. The denoised image is found in the presence of delays.

\begin{figure}[!h]
    \centering
    \includegraphics[width=0.85\linewidth]{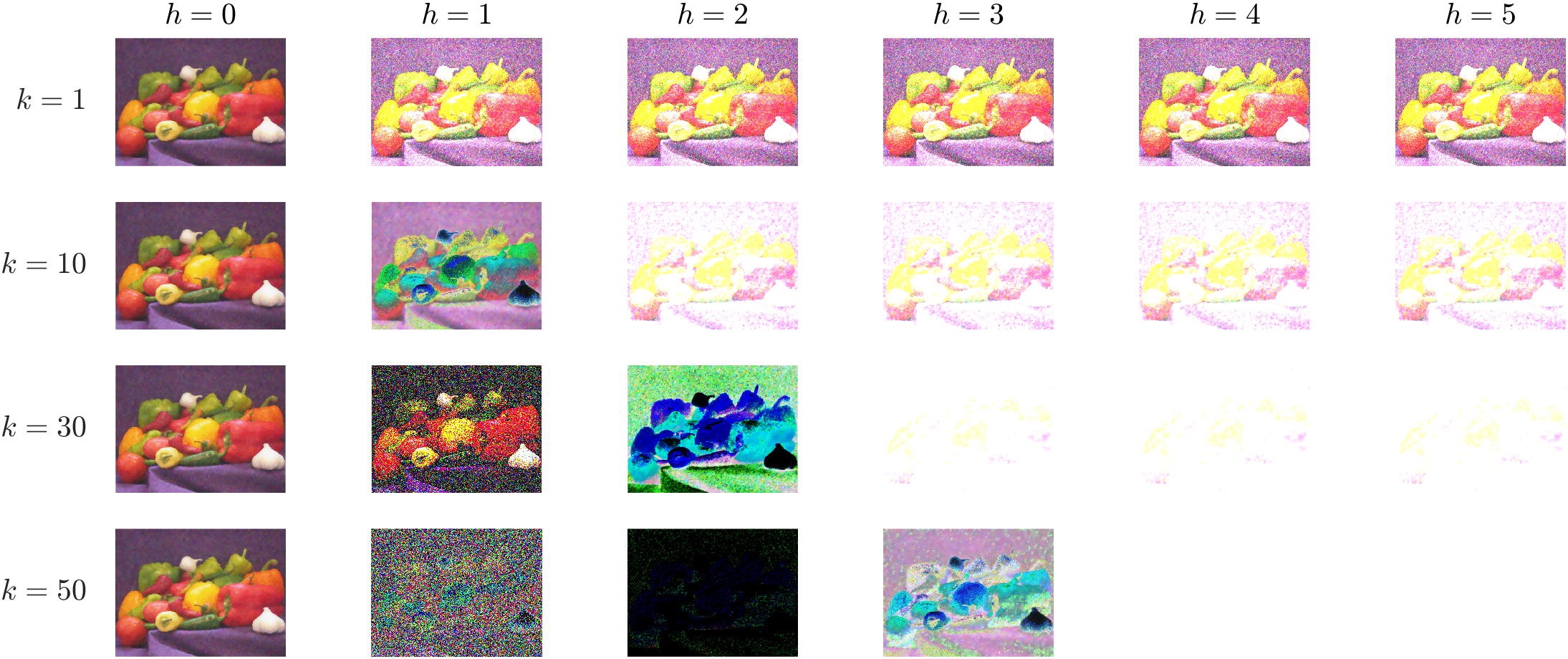}
    \caption{Douglas-Rachford algorithm vs. delay  $h$ and iteration $k$. Nonconvergent when $h>0$.}
    \label{fig:pepper_dr}
    \vspace{-5mm}
\end{figure}

\begin{figure}[!h]
    \centering
    \includegraphics[width=0.85\linewidth]{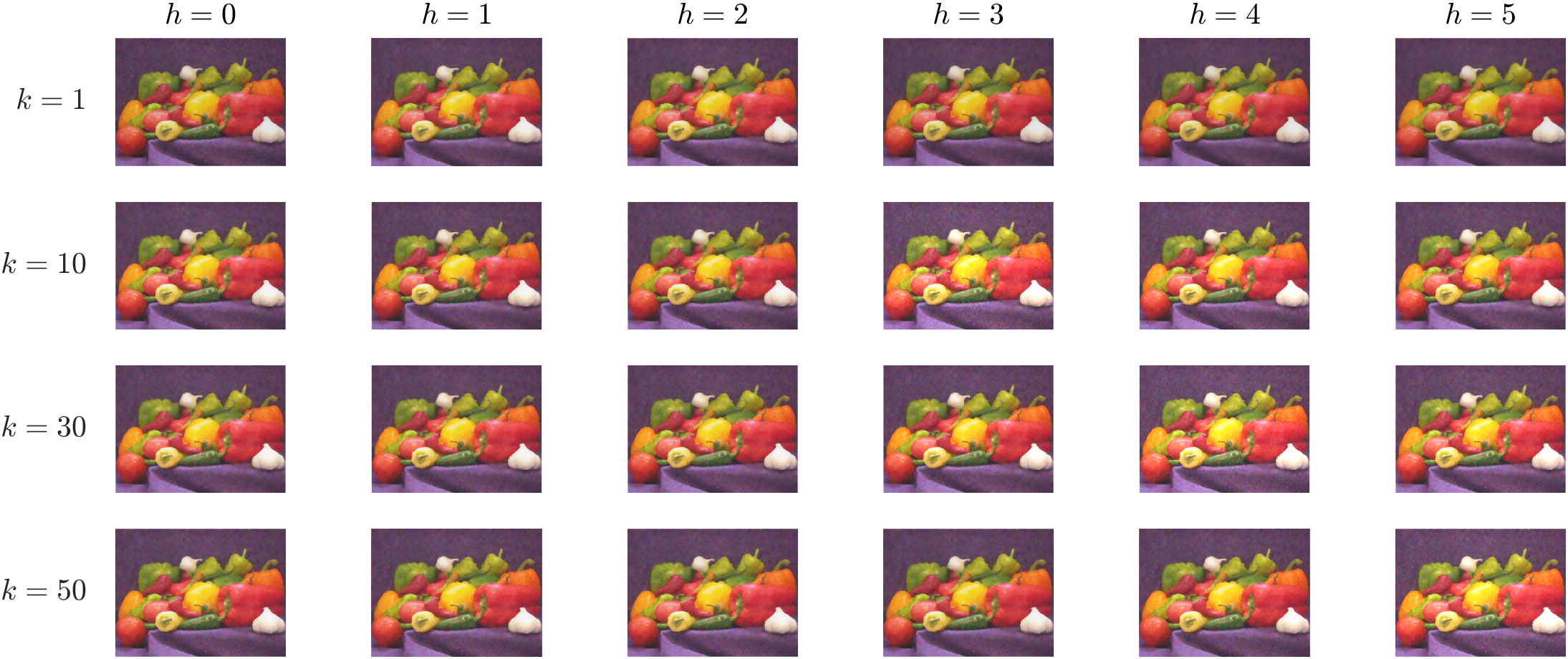}
    \caption{Full-order synthesized algorithm vs. delay  $h$ and iteration $k$. Convergent when $h>0$.}
    \label{fig:pepper_syn}
    \vspace{-5mm}
\end{figure}

\section{Conclusion}
\label{sec:conclusion}

A first-order networked composite optimization algorithms can be interpreted as an interconnection of a static map $F \in \mathcal{O}_{m, L},$ a given network $P$, and a to-be-designed controller $\Kc$.
In this paper, we unveil the modeling power of this interconnection-based paradigm.
We first show that the convergence of the
the corresponding well-posed algorithmic interconnection $F \star (P\star \Kc)$ for all $F \in \mathcal{O}_{m, L}$ requires satisfaction of Robust Stability and Regulator Equation conditions. These necessary conditions for algorithmic convergence induce structural constraints on any controller $\Kc$, {in that it} must satisfy a minimal order bound based on information constraints $\linfo$, and {that it} must factorize into the cascade of a core subcontroller $\Sco$ and an internal model $\Sigma_{\text{min}}$.
Moreover, we address how these structural properties permit
the design of novel composite optimization algorithms with certified exponential convergence rates, even  under the presence of network dynamics. The results are illustrated with several examples, include the tailored design of algorithms for image denoising under delayed gradient information transmission.

Possible extensions include generalizations of algorithm synthesis from strongly convex to merely convex problems with guarantees of sublinear convergence, the use of integral quadratic constraint synthesis to generate algorithms that can solve solve variational inclusion problems, and the reduction of the conservatism of the employed uncertainty characterizations for the class $\mathcal{O}_{m, L}$.
Finally, we aim at developing convex methods for the joint search of stability filter coefficients and subcontrollers.

\appendix

\section{Properties of Functions in $\c{S}_{m,L}$}
\label{app:extended_subdiff}

This appendix collects {several} results of {the class} $\c{S}_{m,L}$ {as defined in} Section \ref{Sconana}.

{

\subsection{Dissipation Relations}

Valid dissipation relations will be built from the following lemma.
\begin{lemma}
\label{lem:dissip_scalar}
    Given a function
$f \in \mathcal{S}_{m, L}$, define $\sigma:=(L-m)^{-1}\geq 0$ 
and the map
\begin{equation}
V(x,g):=f(x)-m\bq(x)-\sigma \bq(g-mx)\te{for}x,g\in\R^c.
\end{equation}

For any constants $\alpha_j> 0$,
    points $\bar{x}_j \in \R^c$, and vectors $g_j \in \alpha_i^{-1} \partial f (\alpha_i \bar{x}_i)$ for $j \in \{1, 2\}$, the following inequalities hold{:}
    \begin{subequations}
    \begin{align}
& V(\alpha_1 \bar{x}_1, \alpha_1 \bar{g}_1)-V(\alpha_2 \bar{x}_2, \alpha_2 \bar{g}_2)\label{eq:supply_1_exp}\\
& \qquad \qquad \qquad \leq  \alpha_1 (\bar{g}_1-m\bar{x}_1)^\top [\alpha_1 (\sigma L\bar{x}_1-\sigma g_1)-\alpha_2 (\sigma L\bar{x}_2-\sigma \bar{g}_2)],
    \nonumber \\
0& \leq [\alpha_1 (\bar{g}_1-m\bar{x}_1)-\alpha_2 (\bar{g}_2-m\bar{x}_2)]^\top [\alpha_1 (\sigma L\bar{x}_1-\sigma \bar{g}_1)-\alpha_2 (\sigma L\bar{x}_2-\sigma \bar{g}_2)].
\label{eq:supply_2_exp}
\end{align}
\end{subequations}
\label{cor:subdiff_extended_ineq_exp}
\end{lemma}
Lemma \ref{lem:dissip_scalar} is a slight  extension of  \cite[Theorem 5]{scherer2023robustozf} as in \cite[Lemma 12]{scherer2025tutorial} to the
the case of  $f\in\c{S}_{m,\infty}$, in which $(x,g)$ are scaled by  $\alpha$. Choosing $\alpha$ to be powers of the exponential  rate $\rho > 0$ allows for the definition of a  family of valid Zames-Falb filters.

\begin{lemma}
\label{lem:operator_passive}
    Let $\tilde{F} \in \mathcal{O}_{m, L}^0$,   $\rho >0$ be an exponential convergence rate. 
    Given any sequence $(\bar{z}, \bar w)$ satisfying  $\bar{w}_k \in \rho^{-k}\tilde{F} (\rho^k \bar{z}_k)$ for all $k \in \N$, we define the new signals 
    \begin{align}
    \label{eq:exp_signals}
        \bar{q}_k^i &:= \bar{w}_k^i - m_i \bar{z}_k^i, &
        \bar{p}_k^i &:= \sigma_i L_i \bar{z}_k^i - \sigma_i \bar{w}_k^i.
    \end{align}

    These signals satisfy the  dissipation relations for all  $T, \ell \in \N,$ with  $T-\ell > 0$,
    \begin{align}
         \textstyle \sum_{k=0}^{T-1} \bar{q}_k^\top \bar{p}_k & \geq 0, &  \textstyle \sum_{k=0}^{T-\ell} \bar{q}_k^\top (\bar{p}_k - \rho^\ell \bar{p}_{k-\ell}) & \geq 0. \label{eq:indiv_dissip}
    \end{align}

    Next, for any finite filter length $\nu_{\max} \in \N$  and coefficients  $(\mu^i_{\nu})_{\nu=0}^{\nu_{\max}}$ satisfying  $\mu^i_\nu \geq 0 $ and $\mu^i_0 > 0$ at all $(i, \nu)$, the signal 
    \begin{align}
        \bar{r}_k^i &:= \textstyle \bar{p}_k^i (\sum_{\nu=0}^{\nu_{\max}} \mu_\nu) - \sum_{\nu=1}^{\nu_{\max}} \mu_\nu \bar{p}_{k-\nu}^i.
    \end{align}
    satisfies {the} dissipation relation
    \begin{align}
        \forall T \in \N, \ T > 0: & &  \textstyle \sum_{k=0}^{T-1} \bar{q}_k^{\top} \bar{r}_k \geq 0. \label{eq:exp_passive}
    \end{align}
\end{lemma}
\begin{proof}
Relation \eqref{eq:exp_passive} follows by taking conic combinations of relations in \eqref{eq:indiv_dissip} with weights $\mu$. 
The lemma then follows as an extension of \cite[Lemma 5]{scherer2023optimization} at each $i$  to the case where the operator ${\tilde{F} \in \mathcal{O}_{m, L}^0}$ is not globally defined.
The arguments used here are identical to those employed by \cite[Lemma 5]{scherer2023optimization} and \cite[Section 3.2]{scherer2025tutorial}, and are thus omitted for space.

The association to \cite[Lemma 5]{scherer2023optimization} is by 
defining coefficients $\lambda$ from $\mu$ for all $i$ by $\lambda^i_0 = \sum_{\nu=0}^{\nu_{\max}} \mu_\nu$, and by $\lambda^i_\nu = -\mu_\nu$ if $\nu > 1$.
The nonnegativity and positivity conditions on $\mu$ then imply 
\begin{align}
\label{eq:passive_conditions_app}
\textstyle \sum_{\nu=0}^{\nu_{\max}} \rho^{-\nu} \lambda^i_\nu > 0,   \quad   \lambda^i_{\nu}  \leq 0, \ \quad \forall \nu \geq 1, \quad \forall i \in \{1\ldots, s\}.
\end{align}

\end{proof}

\subsection{Generalized Resolvents and Yosida Operators}

We first review properties of operators  $H: \R^c \rightrightarrows \R^c$ \cite{rockafellar2009variational, bauschke2020hilbert}.
The graph of the operator $H$ is $\mathrm{gra}(H) = \{(x, y)\in\R^c\times\R^c \mid y \in H(x)\}$.
$H$ is strongly monotone with constant $m>0$ (or just monotone for $m=0$) if $(y_1 - y_2)^\top (x_1 - x_2) \geq m\|x_1-x_2\|^2$ holds for all $(x_1, y_1), (x_2, y_2) \in \mathrm{gra}(H)$.
An operator $H$ is maximal (strongly) monotone if it is (strongly) monotone and there does  not exist a (strongly) monotone $H'$ such that $\mathrm{gra}(H) \subset \mathrm{gra}(H')$.
If $H$ is maximal strongly monotone, then $H^{-1}$ is
globally continuous \cite[Proposition 12.54]{rockafellar2009variational}. We need the following auxiliary results to provide sufficient well-posedness conditions.

\begin{lemma}\label{Linv}
Let $E,F,G:\R^c\tto\R^c$ be operators, and suppose that $F$ is at-most  single-valued. Then
$$
(EF^{-1}-G)^{-1}=F(E-GF)^{-1}.
$$
\end{lemma}

\begin{proof} For any $x\in\R^c$, $y\in (EF^{-1}-G)^{-1}(x)$ implies $x\in (EF^{-1}-G)(y)$, and hence $x\in E(z)-G(y)$ for some $z\in F^{-1}(y)$.
Because $y\in F(z)$, we have $x\in E(z)-GF(z)=(E-GF)(z)$ with $z\in (E-GF)^{-1}(x)$. This finally shows
$y\in F(E-GF)^{-1}(x)$.
Now let $y\in F(E-GF)^{-1}(x)$ for any $x\in\R^c$. Then $y\in F(z)$ for some $z\in (E-GF)^{-1}(x)$, and thus $x\in (E-GF)(z)$. It therefore holds that
$x\in E(z)-G(\t y)$ for some $\t y\in F(z)$.  Since $F$ is at most single-valued, we have $\t y=y$. Using the relations $x\in E(z)-G(y)$ and  $z\in F^{-1}(y)$, we infer $x\in EF^{-1}(y)-G(y)=(EF^{-1}-G)(y)$, which finally shows $y\in (EF^{-1}-G)^{-1}(x)$.
\end{proof}

\begin{corollary}\label{Cinv}
For an operator $F:\R^c\tto\R^c$ and $D\in\R^{c\times c}$, it holds:
\enu{
\item If $F$ is at most single-valued, then $(F^{-1}-D)^{-1}=F(I-DF)^{-1}$.\label{p1}
\item If $F^{-1}$ is at most single-valued, then $(DF-I)^{-1}=F^{-1}(D-F^{-1})^{-1}$.\label{p2}
\item If  $D$ is invertible, then $(F^{-1}-D)^{-1}=-D^{-1}[I-(I-DF)^{-1}]$.\label{p3}
}
\end{corollary}
\begin{proof}
To prove \ref{p1}., apply Lemma \ref{Linv} with $E=I$ and $G=D$.
For \ref{p2}. use Lemma \ref{Linv} for $E=D$, $F$ replaced by $F^{-1}$ and $G=I$.
To show \ref{p3}., choose $E=F^{-1}$, $F=D$ and $G=I$ in  Lemma \ref{Linv} to get
$(F^{-1}D^{-1}-I)^{-1}=D(F^{-1}-D)$. By the inverse-resolvent identity,
the left-hand side indeed equals $-(I-(I-DF)^{-1})$.
\end{proof}

The following
concerns the notion of well-posedness in this paper.

{
\begin{proposition} 
\label{cor:well_posed_exp} 
Given $F:\R^c\tto\R^c$ and $D\in\R^{c\times c}$, let $(F^{-1}-D)^{-1}$ be globally continuous. Then:
\enu{
\item If $\al>0$ and $\bar{F}(x) := \alpha^{-1} F(\alpha x)$ for $x\in\R^c$ then
$(\bar F^{-1}-D)^{-1}(x)=\alpha^{-1} (F^{-1}\! -\! \Dcl)^{-1}(\alpha x)$ for $x\!\in\!\R^c$. Hence
$(\bar F^{-1}\!-\!D)^{-1}$ is globally continuous.
\item If $a,b\!\in\!\R^c$ and $\bar{F}(x)\!:=\!F(x\!+\!a)\!-\!b$ for $x\!\in\!\R^c$, then $(\bar{F}^{-1}\!\!-\!D)^{-1}(x)=(F^{-1}\!\!-\!D)^{-1}(x+a-\!Db)\!-b$ for $x\!\in\!\R^c$. \!Hence $(\bar{F}^{-1}\!-\!D)^{-1}$ is globally continuous.

}
\end{proposition}
\begin{proof} To show 1., take any $x\in\R^c$ and let $y=\alpha^{-1} (F^{-1} - D)^{-1}(\alpha x)$.
This is equivalent to $\alpha y = (F^{-1} - D)^{-1}(\alpha x)$ and hence to $\alpha x\in F^{-1}(\alpha y) - D(\alpha y)$, which is nothing but $x\in \alpha^{-1} F^{-1}(\alpha y) - D y$; since $\alpha^{-1} F^{-1}(\alpha \ \bullet)=\bar{F}^{-1}$,
this holds iff $x\in\bar F^{-1}(y)-D y$, which is finally equivalent to $y\in (\bar{F}^{-1}-D)^{-1}(x)$.

To show 2., pick any $x\in\R^c$ and let $y\in (F^{-1}-D)^{-1}(x+a-Db)-b$; then $y+b\in (F^{-1}-D)^{-1}(x+a-Db)$ and hence
$x+a-Db\in F^{-1}(y+b)-D(y+b)$; we infer that  there exists some $z\in F^{-1}(y+b)$ with $x=z-a-Dy$;
hence $y+b\in F(z)$ and $x=z-a-Dy$; then $\bar z:=z-a$ satisfies $y\in F(\bar z+a)-b=\bar F(\bar z)$ and $x=\bar z-Dy$,
implying $\bar z\in \bar{F}^{-1}(y)$ and $x=\bar z-Dy$, which finally gives
$x\in (\bar{F}^{-1}-D)(y)$ and thus $y\in (\bar{F}^{-1}-D)^{-1}(x)$. Reversing the arguments proves
that $(F^{-1}-D)^{-1}(x+a-Db)$ is not only contained in, but actually equal to $(\bar{F}^{-1}-D)^{-1}(x)$.
\end{proof}
}

\begin{proposition}
\label{prop:invertible}

If a matrix {$D \in \R^{c \times c}$} ensures that $H = (\partial f^{-1} - D)^{-1}$ is globally continuous for all $f \in \mathcal{S}_{m, \infty}$, then $D$ is invertible.
\end{proposition}
\begin{proof}
We prove this proposition by contradiction. {For any $m\in\R$,} consider 
$f(\beta) = \norm{\beta}_1 + \frac{m}{2} \norm{\beta}^2_2$ with $f \in \mathcal{S}_{m, \infty}$.  The subdifferential of $f$ at $0$ is $\partial f(0) = \partial \norm{0}_1 + m 0 = \partial \norm{0}_1  =  [-1, 1]^c$. If {$D$} is rank-deficient, there exists a vector  $w \in \R^c$ {with} $\norm{w}_{1} = 1$ such that $D w = 0$. We have $0, w \in \partial f(0)$, which implies $0 \in \partial f^{-1}(0)$ and $0 \in \partial f^{-1}(w)$ {and hence}  $0 \in \partial f^{-1}(0) - {D0} = \partial f^{-1}(0)$ and $0 \in \partial f^{-1}(w) - {Dw} = \partial f^{-1}(w)$. {This shows} $0 \in (\partial f^{-1} - D)^{-1}(0)$ and $w  \in (\partial f^{-1} - D)^{-1}(0)$, which contradicts global single-valuedness  of $H$.
{Hence $H$ is not globally continuous.}
\end{proof}

\subsection{Guaranteeing Well-Posedness}

{The goal is to} characterize those matrices $D\in\R^{sc\times sc}$ for which $({F^{-1}}-D)^{-1}$ is globally continuous for all $F\in\mathcal{O}_{m, L}$. 
To simplify discussion, we assume that $D$ is block-lower-triangular.

\begin{lemma}
\label{lem:well_posed_indiv}
    Given a function $f_i \in \mathcal{S}_{m_i, L_i}$  with $\sigma_i = \frac{1}{L_i - m_i}$ and a symmetric matrix $D_{ii} \in \R^{c \times c}$, the inequality $\text{Sym}(-\sigma_i + D_{ii}(I- m_i D_{ii})^{-1}) \prec 0$ implies that $(I - D_{ii} \partial f_i)^{-1}$ and $(\partial f_i^{-1} - D_{ii})^{-1}$ are both globally continuous. 
    \end{lemma}
\begin{proof}

Appendix A of \cite{scherer2025tutorial}
proves that $(I - D_{ii} \partial f_i)^{-1}$ is globally continuous if the inequality holds when $L_i < \infty$.  This furthermore implies that $(\partial f_i^{-1} - D_{ii})^{-1}$ is globally continuous by Corollary \ref{Cinv} part 1.

We therefore focus on the case $L_i = \infty$, in which $\sigma_i = 0$. Strictness of the inequality $\text{Sym}(D_{ii}(I- m_i D_{ii})^{-1}) \prec 0$ requires that  $D_{ii}$ is an invertible matrix. It therefore holds that $(\partial f_i^{-1} - D_{ii})^{-1} = D_{ii}^{-1}(I - (I - D_{ii} \partial f_i)^{-1})$ by Corollary \ref{Cinv} part 3.
If $w_i \in (I - D_{ii} \partial f_i)^{-1}(z_i)$, then $z_i  \in w_i  - D_{ii} \partial f_i(w_i) = (I - D_{ii} m_i) w_i  - D_{ii} \partial (f_i - m_i I)(w_i)$ and $-D_{ii}^{-1} z_i \in -D_{ii}^{-1}(I - D_{ii} m_i) w_i  +  \partial (f_i - m_i I)(w_i)$. Because $f_i \in \mathcal{S}_{m_i, \infty}$, the function $f_i - m_i I$  is c.c.p. and $\partial (f_i - m_i I)$ is maximal monotone. The constraint $\text{Sym}(D_{ii}(I- m_i D_{ii})^{-1}) \prec 0$ ensures that $(I - D_{ii} f_i)$ is maximal  strongly monotone, which further implies that $(I - D_{ii} f_i)^{-1}$ is globally continuous \cite[Proposition 12.54]{rockafellar2009variational}.
\end{proof}

 \section{Regulation Theory}\label{app:regulation}

Consider {the following linear system driven by a constant disturbance $d_0\in\R^{n_d}$:}
\begin{align}
\label{eq:reg_system_exo}
    \t G: & &\mat{c}{x_{k+1} \hl e_k} &= \mat{c|c}{\t\Acl & \t\Bcl \hl \t\Ccl &  \t\Dcl } \mat{c}{x_k \hl d_k},\ \ d_{k+1}=d_k.
\end{align}

\begin{definition}
\label{defn:output_regulation}
    The system in \eqref{eq:reg_system_exo} attains \textbf{output regulation} if $\t\Acl$ is Schur, and for all $x_0 \in \R^n, d_0 \in \R^{n_d}$, it holds that $\lim_{k \rightarrow \infty} e_k= 0$.
\end{definition}

A {core} condition for output regulation is the feasibility of the  Regulator Equation.
\begin{lemma}[Lemma 1 \cite{francis1976internal}]
\label{lem:nominal_regulation}

The system $\t G$ with representation $(\t\Acl, \t\Bcl, \t\Ccl, \t\Dcl)$ achieves output regulation iff the following two conditions hold:
\begin{enumerate}
    \item Stability: $\lim_{k\rightarrow 0} x_k = 0$ for all $x_0 \in \R^n$ if $d_0=0$.
    \item {Solvability of} Regulator Equation: There exists {a matrix $\Upsilon$ satisfying}
        \begin{align}
         \mat{cc}{\t\Acl - I& \t\Bcl \\ \t\Ccl & \t\Dcl} \mat{c}{\Upsilon \\ I} & = 0.%
         \label{eq:regulator_closed}
    \end{align}
\end{enumerate}
\end{lemma}

\subsection{Open-Loop Regulation}

Based on \cite{francis1976internal, francis1977linear, saberi2012control}, let us review nominal regulation theory for
the interconnection of a plant  $\tilde{P}$ and a controller $\Kc$ described as
\begin{align}
     \mat{c}{
    x_{k+1}  \hl e_k \\ y_k
    } &= \mat{c|cc}{ \tlA & \tlB_1 & \tlB_2 \hl
    \tlC_1 & \tlD_1 & \tlD_{12} \\
    \tlC_2 & \tlD_{21} & \tlD_2 }\mat{c}{
        x_k \hl d \\ u_k
    }, & & \mat{c}{
        \xi_{k+1}  \hl u_k}=\mat{c|c}{\Ac & \Bc \hl
    \Cc & \Dc}\mat{c}{
        \xi_k \hl y_k}
    \label{eq:formed_by_interconnection}
\end{align}
and affected by the constant disturbance $d$. {It is assumed that $\Kc$ internally stabilizes $\t{P}$ 
and that $\t P\star\Kc$ is represented by $(\t\Acl,\t\Bcl,\t\Ccl,\t\Dcl$) (Section~\ref{Ssys}).

Now suppose that $\Kc$ achieves output regulation for $\t P\star\Kc$. Then \eqref{eq:regulator_closed} admits a unique solution $\Upsilon$ and we can partition
$\Upsilon=  (\Pi^\top, \ \Theta^\top)^\top$ according to the partition of $\t\Acl$ (see \r{cll}).
Then the  matrices
}
\begin{align*}
    \Gamma := (I - \Dc \tilde{D}_{2})^{-1}(\Cc \Theta + \Dc (\tilde{C}_2 \Pi + \tilde{D}_{21}))\te{and}
    \Phi := \tilde{C}_2 \Pi + \tilde{D}_{21} + \tilde{D}_{2}\Gamma
\end{align*}
generate a solution $(\Pi, \Gamma, \Phi, \Theta)$ of the \textit{open-loop regulator equations}

\begin{align}
\label{eq:regulator_equation}
\mat{c|ccc}{
\tlA  &\tlB_1   &\tlB_2\hl
\tlC_1&\tlD_{11}&\tlD_{12}\\
\tlC_2&\tlD_{21}&\tlD_{2}}
\mat{c}{\Pi\hl I\\\Gamma}&=\mat{c}{\Pi\hl 0\\\Phi}, &
\mat{cc}{\Ac&\Bc\\ \Cc&\Dc}
\mat{c}{\Theta\\\Phi} &= \mat{c}{\Theta \\ \Gamma}.
\end{align}

Conversely, for any quadruple $(\Pi, \Gamma, \Phi, \Theta)$, the signal transformations
\begin{align}
    \hat{x}_k &= x_k - \Pi d, & \hat{\xi}_k &= \xi_k - \Theta d, & \hat{y}_k &= y_k - \Phi d, & \hat{u}_k &= u_k - \Gamma d \label{eq:coord_transform_s}
\end{align}
in the interconnection \eqref{eq:formed_by_interconnection} lead to
\begin{align*}\arraycolsep.4ex
     \mat{c}{
    \hat{x}_{k+1}  \hl e_k \\ \hat{y}_k
    } &= \mat{c|c:c}{ \tlA & \tlA\Pi - \Pi  + \tlB_2 \Gamma + \tlB_1 & \tlB_2 \hl
    \tlC_1 & \tlC_1 \Pi + \tlD_{12} \Gamma + \tlD_1 & \tlD_{12} \\
    \tlC_2 & \tlC_2 \Pi + \tlD_2 - \Phi + \tlD_{21} & \tlD_2 }\mat{c}{
        \hat{x}_k \hl d \\ u_k
    },\\ 
    \mat{c}{
        \hat{\xi}_{k+1}  \hl \hat{u}_k}& =\mat{c|c : c}{\Ac & \Ac \Theta - \Theta + \Bc \Phi &  \Bc \hl
    \Cc & \Cc \Theta + \Dc \Phi - \Gamma & \Dc}\mat{c}{
        \hat{\xi}_k \hl d \\ \hat{y}_k    }. \label{eq:formed_by_interconnection_coord}
\end{align*}
If $(\Pi, \Gamma, \Phi, \Theta)$ is taken to satisfy the regulator equations \eqref{eq:regulator_equation}, we clearly obtain
\begin{equation}
     \mat{c}{
    \hat{x}_{k+1}  \hl e_k \\ \hat{y}_k
    } = \mat{c|c:c}{ \tlA & 0 & \tlB_2 \hl
    \tlC_1 & 0 & \tlD_{12} \\
    \tlC_2 & 0 & \tlD_2 }\mat{c}{
        x_k \hl d \\ u_k
    }, \ \  \mat{c}{
        \hat{\xi}_{k+1}  \hl \hat{u}_k}=\mat{c|c : c}{\Ac & 0 &  \Bc \hl
   \Cc & 0& \Dc}\mat{c}{
        \hat{\xi}_k \hl d \\ \hat{y}_k    }.\label{eq:formed_by_interconnection_coord_decouple}
\end{equation}
As a consequence of the regulator equations, the signals $(\hat{x}, \hat{\xi}, \hat{y}, \hat{u}, e)$ in \r{eq:formed_by_interconnection_coord_decouple} are decoupled from the exogenous input $d$, which implies that
$\Kc$ achieves output regulation for $\t P\star\Kc$. Indeed, since $\Acl$ is Schur, all trajectories
of \eqref{eq:formed_by_interconnection_coord_decouple} satisfy $\lim_{k \rightarrow \infty} (\hat{x}_k, \hat{\xi}_k) = 0$.
Invertibility of $I - \tilde{D}_2 \Dc$ ensures that the internal signal $(\hat{u}, \hat{y})$ also satisfies  $\lim_{k \rightarrow \infty} (\hat{u}_k, \hat{y}_k) = 0$. This implies $\lim_{k \rightarrow \infty} e_k = 0$ for any constant disturbance $d$. In view of \eqref{eq:coord_transform_s}, we can conclude for all trajectories of \eqref{eq:formed_by_interconnection} that
\begin{align}
    \lim_{k \rightarrow \infty} (x_k, \xi_k, y_k, u_k) =
     (\Pi d, \Theta d,  \Phi d,  \Gamma d).\label{eq:asymptotic_regulation_closed}
\end{align}

\section{Proof of Theorem \ref{thm:convergence}}
\label{app:convergence_proof}

{
To show the necessity of Conditions 1 and 2, we assume that
$F\star(P\star \Kc)$ is a well-posed convergent algorithm for all  $F \in \mathcal{O}_{m, L}$.
Hence $P\star \Kc$ is well-posed and, therefore, $I-D_2\Dc$ is invertible.
Moreover, for any $\t F \in \mathcal{O}_{m, L}^0$, the algorithm
$\t F\star(P\star \Kc)$ is well-posed and convergent since $\t{F}\in \mathcal{O}_{m, L}$.
Because $\t{F}$ satisfies $0\in \t{F}(0)$, this algorithm has the fixed point $(0,0,0)$. We then prove that closed-loop convergence in $x$ implies convergence in $(w, z)$.
{
\begin{proposition} If there exists a unique vector $x^*$ such that all trajectories $x$ of the well-posed  $F \star G$   algorithm satisfy  with
$\lim_{k\to\infty}x_k=x^*$, then $F \star G$ admits a unique fixed point
$(x^*, z^*, w^*)$. Furthermore,  and all its trajectories satisfy $\lim_{k\to\infty} (x_k,w_k,z_k)=(x^*, w^*, z^*)$. \label{prop:well_posed_converge}
\end{proposition}
}
\begin{proof}
{Since $H = (F^{-1} - \Dcl)^{-1}$ in \eqref{eq:algorithm_causal} is continuous, by taking the limit $k\to\infty$ we infer
$x^*=\Acl x^*+\Bcl H(\Ccl x^*)$ and $\lim_{k\to \infty}w_k = H(\Ccl x^*)=:w^*$ as well as
$\lim_{k\to\infty}z_k=\Ccl x^*+\Dcl H (\Ccl x^*)=:z^*$.
Hence $w^*=\Acl x^*+\Bcl w^*$ and $\Ccl x^*\in H^{-1}(w^*)=F^{-1}(w^*)-\c{D}w^*$, i.e., $z^*=\Ccl x^*+\Dcl w^*\in F^{-1}(w_*)$ and thus $w^*\in F(z^*)$. This shows that $(x^*,w^*,z^*)$ is a fixed point of \eqref{eq:algorithm} and $\lim_{k\to\infty} (x_k,w_k,z_k)=(x^*, z^*, w^*)$.

Any other fixed point $(\bar x^*,\bar w^*,\bar z^*)$ of \eqref{eq:algorithm} satisfies
$\bar x^*=\Acl \bar x^*+\Bcl \bar w^*$, $\bar z^*=\Ccl \bar x^*+\Dcl \bar w^*$ and $\bar w^*\in F(\bar z^*)$; hence
$\Ccl \bar x^*\in F^{-1}(\bar w^*)-\Dcl \bar w^*$ and thus $\bar w^*=H(\Ccl \bar x^*)$, which gives $\bar x^*=\Acl\bar x^*+\Bcl H (\Ccl \bar x^*)$ as well as $\bar z^*=\Ccl\bar x^*+H(\Ccl \bar x^*)$.
If we initialize a trajectory of \eqref{eq:algorithm_causal} as $x_0=\bar x^*$, it hence
satisfies $x_k=\bar x^*$ for all $k\in\N$; by assumption, we have $\lim_{k\to }x_k=x_*$ which implies 
$\bar x^*=x^*$ and hence also $(\bar x^*,\bar w^*,\bar z^*)=(x^*,w^*,z^*)$.
}
\end{proof}

All  trajectories therefore converge to $(0,0,0)$, which  shows that Condition 1 (robust stability) holds true.
}

\subsection{Proof of Necessity of Condition 2}
\label{sec:reg_necessity}
{Let us use denote the matrices describing  $G=P\star \Kc$ by $(\Acl,\Bcl,\Ccl,\Dcl)$
and  represent the algorithm $F\star(P\star\Kc)$ by \eqref{eq:algorithm}. Moreover, we now choose test quadratics $F^t(z)=\bm z+b$, which satisfy $F^t\in\mathcal{O}_{m, L}$ (by Assumption 1);
hence $F^t\star G$ is well-posed and convergent.

First suppose $b=0$ such that $F^t(z)=\bm z$ and $F^t\in\c{O}_{m,L}^0$. Since $F^t\star G$ is well-posed, we conclude that $I-\Dcl\bm$ is invertible
and that $((F^t)^{-1}-\Dcl)^{-1}$ is defined by the matrix $H:=\bm(I-\Dcl\bm)^{-1}$.

Indeed, suppose there exists some $z\neq 0$ with $(I-\Dcl\bm)z=0$; this implies $0=z-\Dcl y$ for $y=\bm z\neq 0$; hence
$z\in (F^t)^{-1}(y)$ and thus $0\in ((F^t)^{-1}-\Dcl)(y)$, implying $y\in ((F^t)^{-1}-\Dcl)^{-1}(0)$;
the same argument for $z=0$ shows $0\in ((F^t)^{-1}-\Dcl)^{-1}(0)$, which contradicts the fact that $((F^t)^{-1}-\Dcl)^{-1}$
is single-valued. By 1.\ in Corollary~\ref{Cinv}, we infer $((F^t)^{-1}-\Dcl)^{-1}=\bm(I-\Dcl\bm)^{-1}$.

Hence, $\bm\star G$ is well-posed as a star-product of linear systems and \eqref{eq:algorithm_causal} is
\eql{alg2}{
x_{k+1}=(\Acl+\Bcl(I-\bm\Dcl)^{-1}\bm\Ccl)x_k,\ \ w_k=H\Ccl x_k,\ \ z_k=(I+\Dcl H)\Ccl x_k \; \forall k\in\N.
}
By convergence of $F^t\star G$ described by \eqref{eq:algorithm} with $F^t$ replacing $F$, we infer that 
$(x_k,w_k,z_k)$ converges to a fixed point $(x^*,w^*,z^*)$  for $k\to\infty$; clearly,
$(0,0,0)$ is a fixed point of the algorithm since $F^t(0)=0$; since fixed points are unique (Proposition~\ref{prop:well_posed_converge}), we infer $(x^*,w^*,z^*)=(0,0,0)$ and hence conclucde $\lim_{k\to\infty}x_k=0$.
Since \r{alg2} is another representation of the algorithm $F^t\star G$, we infer that all trajectories of \r{alg2}
satisfy $\lim_{k\to\infty}x_k=0$, which in turn shows that $\Acl+\Bcl(I-\bm\Dcl)^{-1}\bm\Ccl$ is Schur.

Moreover, $\bm\star P$ is well-posed by Assumption 2. Since $\bm\star G$ and
$P\star\Kc$ are well-posed, 
the matrices describing $\bm\star G=\bm\star(P\star\Kc)$ and $(\bm\star P)\star\Kc=P^\bm\star\Kc$ are identical; since $\Acl+\Bcl(I-\bm\Dcl)^{-1}\bm\Ccl$ is Schur, we conclude that
$\Kc$ internally stabilizes $P^\bm$.
Choose arbitrary $\beta^*\in\R^{c}$ and $\hat w^*\in\R^{(s-1)c}$ and define $F^t(z)=\bm z+b$ with
$b:=N\hat w^*-\bm(\1_s\otimes\beta^*)$. By construction, $\hat w^*$ is the unique vector satisfying
$F^t(\1_s\otimes\beta^*)-N\hat w^*=0$, which is in turn used in the definition \r{tF} of $\t F^t$.
Since $(\1_s\otimes I_c)^\top N=0$, we infer $(\1_s\otimes I_c)^\top\bm(\1_s\otimes\beta^*)+(\1_s\otimes I_c)^\top b=0$; hence $\beta^*$ is the unique solution to Problem~\ref{prob:composite}.
By convergence of $F^t\star (P\star \Kc)$ and hence of $(\tilde{F}^t\star P)\star\Kc=P^\bm\star\Kc$, we conclude that $\Kc$ achieves regulation for the plant $P^\bm$ described by \eqref{qua}.

As seen in Section \ref{app:regulation}, we conclude that there exits a solution $(\Pi, \Gamma, \Phi, \Theta)$ to
\begin{align}
\label{eq:regulator_equation_m}
(\Omega_0 + \Omega_\bm)
\mat{c}{\Pi\hl I\\\Gamma}&=\mat{c}{\Pi\hl 0\\\Phi}, &
\mat{c|c}{\Ac&\Bc\hl \Cc&\Dc}
\mat{c}{\Theta\\\Phi} &= \mat{c}{\Theta\\ \Gamma},
\end{align}
where we introduce the following abbreviations for the matrices describing $P^\bm$:
\begin{align}
    \label{eq:Omega_sys}
     \Omega_0 &:= \mat{c|cc:c}{ A &  0 & B_1 N  & B_2 \hl
    C_1  & \1_s \otimes I_c  & D_{1}N    &  D_{12} \hdl
    C_2 & 0 & D_{21}N  & D_2 }, \\
    \Omega_\bm &:= \mat{c}{B_1 \hl D_1 \hdl D_{21}} \bm E(\bm)^{-1} \mat{c|cc:c}{C_1 & \1_s \otimes I_c & D_{1} N  & D_{12}}.
    \end{align}
Since $I+D_1\bm E(\bm)^{-1}=E(\bm)^{-1}$, the middle block row of $\Omega_0+\Omega_\bm$ is
\begin{align}
    E(\bm)^{-1}\mat{c|cc:c}{C_1  &\1_s \otimes I_c & D_{1}N  &  D_{12}}.
\end{align}
Hence the middle block row of the first equation in \eqref{eq:regulator_equation_m} leads to
    \begin{align}
        E(\bm)^{-1}\mat{c|cc:c}{C_1  &\1_s \otimes I_c & D_{1}N  &  D_{12}} \mat{c}{\Pi \hl I \hdl \Gamma} = 0 & & & \text{and hence} &  \Omega_\bm  \mat{c}{\Pi \hl I \hdl \Gamma} = 0. \label{eq:cancel_implication}
    \end{align}
In view of the definition of $\Omega_0$, \eqref{eq:regulator_equation_m} is identical to  \eqref{eq:nominal_regulation_control_sys}, which concludes the proof.
}

\subsection{Proof of Sufficiency of Conditions 1 and 2 for Convergence}

{
Fix any $F \in \mathcal{O}_{m, L}$. By assumption, Problem~\ref{prob:composite} has a solution $\beta^*$ and
we can pick $\hat w^*$ with $0\in F(\1_s\otimes\beta^*)-N\hat w^*$. This permits to define the error map $\t F\in \mathcal{O}_{m, L}^0$ by \eqref{tF} and the corresponding error system \eqref{eq:regnonlinear}.

Since $I-D_2\Dc$ is invertible (Condition 1), $P_e\star\Kc$ and $P\star\Kc$ are both well-posed. Since $\t F\star (P\star\Kc)$ is well-posed (Condition 1), we infer that $\t F\star(P_e\star \Kc)$ and hence also
$F\star (P\star\Kc)$ is well-posed (Section~\ref{Serr}).

Let us now pick any trajectory of the algorithm $F\star (P\star\Kc)$ described by $w_k\in F(z_k)$ and \r{eq:PK}.
The signal transformation \eqref{errsig} generates a trajectory of the error system \eqref{eq:regnonlinear}.
Now pick a solution $(\Pi, \Gamma, \Phi, \Theta)$ of the regulator equations in Condition 2. The subsequent
signal transformation
\begin{align}\label{tra}
 \hat{x}_k^N &= x_k^N - \Pi d, & \hat{\xi}_k &= \xi_k - \Theta d, & \hat{y}_k &= y_k - \Phi d, & \hat{u}_k &= u_k - \Gamma d.
\end{align}
applied to \eqref{eq:regnonlinear} leads
with $d:=\col(-\beta^*,\hat w^*)$ to 
\begin{subequations}
\begin{align*}
    \mat{c}{
    \hat{x}^N_{k+1}  \hl \tilde{z}_k \hdl e_k \hdl \hat{y}_k
    } &= \mat{c|c:c:c}{ A & B_1 &(A-I) \Pi + B_2 \Gamma + \begin{pmatrix}
        0 & B_1 N
    \end{pmatrix} & B_2 \hl
    C_1 & D_1 & C_1 \Pi + D_{12} \Gamma + \begin{pmatrix}
    \1_s \otimes I_c & D_1 N
    \end{pmatrix} &  D_{12} \hdl
    C_1 & D_1 & C_1 \Pi + D_{12} \Gamma + \begin{pmatrix}
        \1_s \otimes I_c & D_1 N
    \end{pmatrix}  & D_{12} \\
    C_2 & D_{21} & C_2 \Pi + D_2\Gamma + \begin{pmatrix}
        0 & D_{21} N
    \end{pmatrix} - \Phi & D_2 }\mat{c}{
        \hat{x}^N_k \hl \tilde{w}_k \hdl d \hdl  \hat{u}_k
    },  \\
         \mat{c}{
        \hat{\xi}_{k+1}  \hl \hat{u}_k
    } &= \mat{c|c:c}{ \Ac & (\Ac-I) \Theta + \Bc \Phi & \Bc \hl
    \Cc & \Cc \Theta + \Dc \Phi - \Gamma & \Dc}\mat{c}{
        \hat{\xi}_k \hl d \hdl \hat{y}_k
    },\qquad \tilde{w}_k \in \tilde{F}(\tilde{z}_k).
    \end{align*}
\end{subequations}
Due to the regulator equations, this simplifies to
\label{eq:reg_nonlinear_zero}
\begin{align*}
    \mat{c}{
    \hat{x}^N_{k+1}  \hl \tilde{z}_k \hdl e_k \hdl \hat{y}_k
    } &= \mat{c|c:c:c}{ A & B_1 &0& B_2 \hl
    C_1 & D_1 & 0 &  D_{12} \hdl
    C_1 & D_1 &0  & D_{12} \\
    C_2 & D_{21} & 0 & D_2 }\mat{c}{
        \hat{x}^N_k \hl \tilde{w}_k \hdl d \hdl  \hat{u}_k
    }, &
         \mat{c}{
        \hat{\xi}_{k+1}  \hl \hat{u}_k
    } &= \mat{c|c:c}{ \Ac & 0 & \Bc \hl
    \Cc & 0 & \Dc}\mat{c}{
        \hat{\xi}_k \hl d \hdl \hat{y}_k
    }, 
    \end{align*}
with $\tilde{w}_k \in \tilde{F}(\tilde{z}_k)$, and hence, to
\begin{align}
\mat{c}{
    \hat{x}^N_{k+1}  \hl \tilde{z}_k \hdl \hat{y}_k
    } &= \mat{c|c:c}{ A & B_1 & B_2 \hl
    C_1 & D_1 &  D_{12} \hdl
    C_2 & D_{21} & D_2 }\mat{c}{
        \hat{x}^N_k \hl \tilde{w}_k \hdl  \hat{u}_k
    }, &
         \mat{c}{
        \hat{\xi}_{k+1}  \hl \hat{u}_k
    } &= \mat{c|c}{ \Ac & \Bc \hl
    \Cc &  \Dc}\mat{c}{
        \hat{\xi}_k \hl \hat{y}_k
    }. 
    \label{eq:reg_nonlinear_delta_zero}
\end{align}
By Condition 1, we conclude $\lim_{k\to\infty} \hat{x}_{k}^N=0$ and $\lim_{k\to\infty} \hat{\xi}_k=0$.

Note that \r{eq:reg_nonlinear_delta_zero} is a description of $\t F\star (P\star\Kc)$. Again with
$(\Acl,\Bcl,\Ccl,\Dcl)$ representing $G=P\star \Kc$, we obtain a trajectory of
\begin{align}
\label{eq:algorithm_zero}
   \hat{x}_k &= \mat{c}{\hat{x}^N_k \\ \hat{\xi}_k}, & \mat{c}{\hat{x}_{k+1}  \hl \tilde{z}_k}&=
        \mat{c|c}{ \mathcal{A} & \mathcal{B} \hl\mathcal{C} &  \mathcal{D}}\mat{c}{\hat{x}_k \hl \tilde{w}_k},  &   \tilde{w}_k \in \tilde{F}(\tilde{z}_k)
\end{align}
which satisfies $\lim_{k\to\infty}\hat{x}_k=0$. Since \eqref{eq:algorithm_zero} is well-posed
(because $\t{F}\star G=\t{F}\star (P\star\Kc)$ is) and has the fixed-point $(0,0,0)$ (since $0\in\t{F}(0)$),
Proposition \ref{prop:well_posed_converge} implies $\lim_{k\to\infty}(\hat{x}_k,\tilde{w}_k,\tilde{z}_k)=0$.
The definition of the error signals $(\tilde{w}, \tilde{z})$ in \r{errsig} {lets} us conclude that $\lim_{k\rightarrow \infty} z_k =   \1_s \otimes \beta^*$ and $\lim_{k\rightarrow \infty} w_k = N \hat{w}^*$.

On the one hand, we infer $\lim_{k\to\infty}(\hat{x}_k^N,\hat{\xi}_k)=0$,
and hence $\lim_{k\to\infty}(x_k^N,\xi_k)=(\Pi d,\Theta d)$ due to \eqref{tra}. On the other hand,
well-posedness of $P\star\Kc$ permits to infer $\lim_{k\to\infty}(\hat{u}_k,\hat{y}_k)=0$ (Section~\ref{Ssys})
which in turn gives $\lim_{k\to\infty}(u_k,y_k)=(\Gamma d,\Phi d)$ by \eqref{tra}. Since
$d=\col(-\beta^*,\hat w^*)$, this concludes the proof.
}

\section{Proof of Proposition \ref{prop:fixed_point_encoding}: Integrator Structures}
\label{app:fixed_point_encoding}

 \textbf{Range-space {condition}: } The range-space condition in \eqref{eq:encoding_range} can be recognized as a Kronecker-structured and sign-adjusted version of the {equation} \eqref{eq:direct_regulator}. If $\Theta^a$ solves the Kronecker structured instance of \eqref{eq:direct_regulator}, then \eqref{eq:encoding_range}  is certified  with
\begin{align}
      \mat{c}{I - \Ac^a\\ -\Cc^a } \left[ \Theta^a \mat{cc}{0 & 1 \\ -I_{s-1} & 0} \right]= \mat{cc}{0 &  \Bc^a N^a\\  \1_s & \Dc^a N^a}.
    \label{eq:direct_regulator_kron}
\end{align}

\textbf{Nullspace {condition}:} 
If $(\Ac-I, \Bc)$ has  full row rank,  {we have $\text{dim}( \nulls(\Ac-I \  \Bc)) = sc$. Due to the Kronecker structure, we infer} implies $\text{dim}( \nulls(\Ac^a-I \  \Bc^a)) = s$.
Letting $n_c$ be dimension of $\Ac^a$, we can partition and rearrange \eqref{eq:direct_regulator} in the Kronecker structured setting into
\begin{align}
    \mat{cc}{\Ac^a -I& \Bc^a \\ \Cc^a & \Dc^a} \mat{cc}{\Theta_1^a & \Theta_2^a \\ 0 & N^a} = \mat{cc}{0 & 0 \\ \1_s & 0}, \label{eq:kron_manipulate_reg}
\end{align}
in which $\Theta_1^a \in \R^{n_c \times 1}$ and $\Theta_2^a \in \R^{n_c \times (s-1)}$. Equation \eqref{eq:kron_manipulate_reg} {shows} $\Cc^a \Theta_1^a = \1_s$, which implies that $\text{rank}(\Theta_1^a) = 1$. The matrix $N^a \in \R^{s \times (s-1)}$ has full column rank because $N$ is a  Kronecker-structured consensus matrix (Definition \ref{defn:consensus}). We therefore have $text{dim}\left( \nulls {\mat{cc}{\Ac^a-I &  \Bc^a}}\right) = s,$
\begin{align}
    \text{dim} \left(\text{ran}  \mat{cc}{\Theta_1^a & \Theta_2^a \\ 0 & N^a} \right) &= s, & \text{and} & &   \text{ran}  \mat{cc}{\Theta_1^a & \Theta_2^a \\ 0 & N^a} \subseteq \nulls {\mat{cc}{\Ac^a-I &  \Bc^a}},
\end{align}
from which it follows that $\text{ran}  \mat{cc}{\Theta_1^a & \Theta_2^a \\ 0 & N^a} =\nulls 
{\mat{cc}{\Ac^a-I &  \Bc^a}}$.
We {continue} by noting that{
\begin{align}
   \mat{cc}{(N^a)^\top \Cc^a & (N^a)^\top \Dc^a \\
            0 & \1^\top_s} \mat{cc}{\Theta_1^a & \Theta_2^a \\ 0 & N^a} &=0
         \label{eq:null_property}
\end{align}
if using the regulator equation \eqref{eq:kron_manipulate_reg} and the relation $\1_s^\top N^a=0$.} 
Together, this implies {\eqref{eq:encoding_null} since}
\begin{align}
    \nulls \mat{cc}{\Ac^a-I &  \Bc^a} = \text{ran}\mat{cc}{\Theta_1^a & \Theta_2^a \\ 0 & N^a} \subseteq \mathrm{null}\mat{cc}{
            (N^a)^\top \Cc^a & (N^a)^\top \Dc^a \\
            0 & \1^\top_s}.
\end{align}

\section{Proof of Theorem \ref{thm:main_structure}: Algorithm Structure}

\label{app:structure_proof}

Assumption \ref{assum:disturbance_detec} separates the nullspaces of $\Phi$ and $\Gamma$ {as follows}.
\begin{lemma}
\label{lem:disjoint_null}
    Under Assumptions \ref{assum:well_posed}-\ref{assum:disturbance_detec}, any solution $(\Pi, \Gamma, \Phi)$ to the regulator  \\equation \eqref{eq:nominal_regulation_control_sys_plant} satisfies $\mathrm{null}(\Phi) \cap \mathrm{null}(\Gamma) = 0$.
\end{lemma}
\begin{proof}
    Applying {the transformation $\hat{x}_{k}^N := x_k^N - \Pi d$ to the system \r{qua} yields
    \begin{align}
    \mat{c}{
    \hat{x}^N_{k+1}  \\ d \hl  \tilde{y}_k
    } \!=\! \left[ \! \mat{c:c}{ A &  -B_2 \Gamma   \hdl
     0 & I_{sc}    \hl
    C_2 & \Phi - D_{2} \Gamma } \!\!+\!\!  \mat{c}{B_1 \hdl 0  \hdl D_{21}} \bm E(\bm)^{-1} \!\! \mat{c:c}{C_1 & -D_{12} \Gamma} \! \right] \!\! \mat{c}{
        \hat{x}^N_k \\ d
    }\!, \label{eq:formed_by_interconnection_exo_close}
\end{align}
which is detectable by by Assumption \ref{assum:disturbance_detec}. Hence,
     \begin{align}
 \mat{cc}{A   + B_1 \bm E(\bm)^{-1} C_1 - I & -(B_2 + B_1 \bm E(\bm)^{-1} D_{12}) \Gamma \\
     0 & 0_{sc} \hl
     C_2 + D_{21} \bm E(\bm)^{-1} C_1 & \Phi - (D_{22} + D_{21} \bm E(\bm)^{-1} D_{12})\Gamma} \label{eq:detectable_full_rank}
\end{align}
has full column rank, which} implies that
\begin{align}
    \mathrm{null}(\Phi) \cap \mathrm{null}\left[\mat{c}{B_2 + B_1 \bm E(\bm)^{-1} D_{12} \\ D_{22} + D_{21} \bm E(\bm)^{-1} D_{12}}\Gamma\right] = \{0\}. \label{eq:null_large}
\end{align}

Since $\text{null}(\Gamma) \subseteq \nulls(M \Gamma)$ for any matrix $M$, we {conclude}
 $\mathrm{null}(\Phi) \cap \mathrm{null}(\Gamma) = \{0\}$.
\end{proof}

\begin{lemma}
\label{lem:theta_injective}
    {
    Under Assumptions \ref{assum:well_posed}-\ref{assum:disturbance_detec}, for any solution $(\Pi, \Gamma, \Phi, \Theta)$ of \eqref{eq:nominal_regulation_control_sys} and any invertible matrix  $R = (R_1 \ R_2)$ with $\text{ran} (R_1) \!=\! \nulls(\Phi)$, the matrices $(\Theta_1 \ \Theta_2)\!:=\!\Theta R$ and
    $(\Gamma_1 \ \Gamma_2  )\!:=\!\Theta R$}
    satisfy $\text{rank}(\Theta_1) \!=\!  \text{rank}(\Gamma_1) \!=\! \text{dim}(\nulls(\Phi))$.
\end{lemma}

\begin{proof}
We prove this Lemma by contradiction.

\textbf{$\mathbf{\Gamma_1}$ condition:}
If $v \neq0$ is a vector such that $\Gamma_1 v = 0$, then $\Gamma R_1 v = 0$. Because $\text{ran}(R_1) = \nulls(\Phi)$, it holds that $\Phi R_1 v = 0$, which implies that $R_1 v \in \nulls(\Phi) \cap \nulls(\Gamma)$. By Lemma \ref{lem:disjoint_null}, we have $\nulls(\Phi) \cap \nulls(\Gamma) = \{0\}$, so $R_1 v = 0$. Since $R_1$ has full column rank, $R_1 v = 0$ {implies $v = 0$, which contradicts $v \neq0$.} Consequently, $\Gamma_1$ has full column rank, {which implies 
$\text{rank}(\Gamma_1) = \text{dim}(\nulls(\Phi))$.}

\textbf{$\mathbf{\Theta_1}$  condition:} We {post-multiply} both sides of the bottom equation in \eqref{eq:nominal_regulation_control_sys_control}:
\begin{align}
    \Gamma R_1 &= \Cc \Theta R_1 + \Dc \Phi R_1&\te{{which reads}} & \Gamma_1 = \Cc \Theta_1 + 0 = \Cc \Theta_1. \label{eq:reg_relation}
\end{align}
Hence ${\text{dim}(\nulls(\Phi))=\rank{\Gamma_1} = \rank{\Cc \Theta_1} = } \text{rank}(\Theta_1) - \text{dim}(\text{ran}(\Theta_1) \cap \nulls(\Cc)) \leq \text{rank}(\Theta_1)$.  Because $\Theta_1$ has $\text{dim}(\nulls(\Phi))$ columns, {we infer that} $\Theta_1$ has full rank.

\end{proof}

We now isolate structural properties of $\Kc$ {in} convergent optimization algorithms $F \star (P \star \Kc)$.
By Lemmas \ref{lem:disjoint_null} and  \ref{lem:theta_injective}, there exists an invertible matrix $Q$ such that
\begin{align}
    Q^{-1} \Theta = \mat{ccc}{-I_r & \Theta_{12}  \\ 0 & \Theta_{22}  }.
\end{align}

\section{Proof of Theorem \ref{thm:analysis}}
\label{app:analysis_proof}
Due to Corollary~\ref{cor:exp_stable}, it suffices to prove that \r{eq:G_exp} is Lyapunov stable.

Since the `strictly proper terms of $P$ and $\bar P$ coincide, we observe that
$\bar P$ satisfies Assumption 2 and $\bar G=\bar P\star\bar \Kc$ is well-posed. Hence
$I-\Dcl\bm$ is invertible and, therefore, $\Sigma_{m,L}\star \bar G$ is well-posed, having the ``D''-matrix $-\bsig I + \Dcl(I-\bm \Dcl)^{-1}$. Since the ``D''-matrix of the state-space representation of $\Psi(\blam)$ in \r{fil} is $\La:=\diag(\lambda^1_0I_c,\ldots\lambda^s_0I_c)$, we conclude that
$\hat{\Dcl}(\blam)=\Blam (-\bsig I + \Dcl(I-\bm \Dcl)^{-1})$. Now note that \eqref{eq:antipassivity_analysis} implies
\begin{align}
    \hat{\Dcl}(\blam) + \hat{\Dcl}(\blam)^\top \prec -\hat{\mathcal{B}}^\top \mathcal{M} \hat{\mathcal{B}} \preceq 0.\label{eq:neg_relation}
\end{align}
Hence
$\mathrm{Sym}[\Blam (-\bsig I + \Dcl(I-\bm \Dcl)^{-1})]\cl 0$.
In view of Condition 1, we conclude that \r{wpc} is satisfied, which implies that $(F^{-1}-\Dcl)^{-1}$ is globally continuous for all $F\in\c{O}_{m,L}$ by Lemma~\ref{Lwp}. For $\t{F}\in\c{O}_{m,L}^0$,
global continuity of $(\t{F}^{-1}-\Dcl)^{-1}$ implies global continuity of
$(\rho^{-k}(\t{F}(\rho^k .)^{-1}-\Dcl)^{-1}$ for every $k\in\N$ by Proposition~\ref{cor:well_posed_exp},
which in turn guarantees that \r{eq:G_exp} is well-posed. By following routine dissipation arguments \cite{scherer2023optimization,scherer2025tutorial}, the LMI \eqref{eq:antipassivity_analysis} guarantees
that \r{eq:G_exp} is Lyapunov stable.

\section{Full-Order Synthesis of Optimization Algorithms}

\label{app:synthesis}

The $\rho$-weighting of the full-order internal model $\Sigma_{\text{full}}$ is
\begin{align}
    \bar{\Sigma}_{\text{full}}: \qquad \mat{c|c:cc}{\rho^{-1} I_{sc} & 0 & \rho^{-1} I_{sc} & 0 \hl
    -\Gamma & 0 & 0 & I_{n_u} \hdl
    \Phi & I_{n_y} & 0 & 0 }.
\end{align}
{{Then} the interconnection $\bar{P} \star \bar{\Sigma}_{\text{full}}$ has the representation %
\begin{align}
\label{eq:partition_full}
  \bar{P} \star \bar{\Sigma}_{\text{full}}: \qquad \mat{c:c|c:cc}{
  \rho^{-1}A & -\rho^{-1} B_2 \Gamma & \rho^{-1} B_1 &  0 & \rho^{-1} B_2 \hdl
  0 & {\rho^{-1}}I_{cs} & 0 & {\rho^{-1}}I_{cs} & 0 \hl
  C_1 & -D_{12} \Gamma & D_1 & 0 & D_{12} \\
  C_2 & {-D_{22}\Phi} & D_{21} & 0 & D_{2}  }.
\end{align}
{We compute a state-space description of
$\hat{P}(\blam) = \Psi(\blam)(\Sigma_{m, L} \star \bar{P} \star \bar{\Sigma}_{\text{full}})$ denoted as
\begin{align}
    \hat{P}(\blam): \quad \mat{c}{\eta_{k+1} \hl \bar{r}_k \hdl  \bar{\tilde{y}}_k} &= \mat{c|c:c}{
    \hat{A}(\blam) &  \hat{B}_1(\blam) &  \hat{B}_2(\blam) \hl
    \hat{C}_1(\blam) & \hat{D}_1(\blam)   & \hat{D}_{12}(\blam) \hdl
    \hat{C}_2(\blam) & \hat{D}_{21}(\blam)  & \hat{D}_2(\blam)
    } \mat{c}{\eta_k \hl \bar{q}_k \hdl \bar{\tilde{u}}_k}. \label{eq:synth_system_hat_full}
\end{align}
}

For fixed filter coefficients $\blam$, the goal is to design  $\cf{{\Sigma}_{\text{f}}}$ such that
$\hat{P}(\blam)\star \cf{\bar{\Sigma}_{\text{f}}}$ satisfies the properties in  Problem~\ref{prob:synth_step}. {To convexify this problem, we follow the nonlinear convexification procedure for controller synthesis in \cite{scherer1997multiobjective}.
It only remains to clarify how to handle the fact that $\hat{D}_2(\blam) = D_2$ does not vanish and that the constraint
$\cf{D_f}_2 \in \mathcal{L}^D$ needs to be satisfied.
 
We first separate the direct feedthrough matrix $\hat{D}_{2}(\blam)$ from the plant $\hat{P}(\blam)$ using a linear fractional transformation, thus forming a control synthesis task with respect to a plant $\cf{\Sigma}^0_f$ (Figure \ref{fig:direct_feedthrough_syn}).
We  achieve condition 1 in
Problem~\ref{prob:synth_step} for the interconnection  on the right in Figure \ref{fig:direct_feedthrough}, by making sure that
the designed controller satisfies $\Dcf_2^0 \in \mathcal{L}^D$. Since $\mathcal{L}^D$ is assumed to be convex, this constraint can be taken into account in the transformation approach to controller design as proposed in \cite{scherer1997multiobjective}. In our presented examples with block-sparsity patterns, the set $\mathcal{L}^D$ is representable by a finite number of linear equality constraints in the entries of $\Dcf_2^0$.

{If non-zero, the matrix $\hat{D}_2(\blam)$ is} brought back in by choosing the controller on the left in
Figure \ref{fig:direct_feedthrough} such that both closed-loop systems coincide. This is achieved for
$$
\mas{c|c}{\Acf & \Bcf \hl \Ccf & \Dcf}= \mat{cc}{0 & I \\ 0 & -\hat{D}_{2}(\blam)}  \star \mas{c|c}{\Acf^0 & \Bcf^0 \hl \Ccf^0 & \Dcf^0}=
 \mas{c|c}{\Acf^0 - \Bcf^0 \es^{-1} \hat{D}_2 \Dcf^0 & \Bcf^0\es^{-1} \hl (I - \Dcf^0 \es^{-1} \hat{D}_{2})\Ccf^0 & \Dcf^0 \es^{-1}}
$$
where $\es = I + \hat{D}_2(\blam) \Dcf^0$.  A slight numerical perturbation of $\hat{D}_2(\blam)$
may be required to render $\es$ invertible,  and invertibility of $\es$  ensures that the star product {is} well-posed. Due to
Assumption \ref{assum:quad_invariance}, we infer that $\Dcf = \Dcf^0(I + \hat{D}_2 \Dcf^0)^{-1}\in\linfo$ is satisfied. Thus, the constructed controller will satisfy Conditions 1 and 2 in
Problem~\ref{prob:synth_step}.
}
\begin{figure}[h]
    \centering
    \begin{subfigure}{0.48\linewidth}
    \center 
    \begin{tikzpicture}[xscale=1,yscale=1,baseline=(ko1)]
\def\dl{2*\dn}
\def\ds{3*\dn}
\node[sy3] (ghat) at (0,0)  {
$
\setlength{\arraycolsep}{2pt}
\mas{c|c:cc}{
\hat{A}(\blam) &  \hat{B}_1(\blam) &  \hat{B}_2(\blam) \hl   
    \hat{C}_1(\blam) & \hat{D}_1(\blam)   & \hat{D}_{12}(\blam) \hdl
    \hat{C}_2(\blam) & \hat{D}_{21}(\blam)  & 0}
$};
\tio{i1}{ghat}{east}{2/3};
\tio{o1}{ghat}{west}{2/3};
\tio{i2}{ghat}{east}{1/3};
\tio{o2}{ghat}{west}{1/3};
\draw[->] (o1)--  node[swap]{$\bar{r}$} ($(o1) + (-2*\ds, 0)$)  ;
\draw[<-] (i1)-- node[]{$\bar{q}$} ($(i1) + (2*\ds, 0)$) ;

\node[sy3,below=0.8\dl of ghat] (lft) {
$
\mat{cc}{0 & I \\ I & \hat{D}_{2}(\blam)}
$
};

\tio{ilft1}{lft}{east}{2/3};
\tio{olft1}{lft}{west}{2/3};
\tio{ilft2}{lft}{east}{1/3};
\tio{olft2}{lft}{west}{1/3};

\node[sy3,below=0.8\dl of lft] (khat) {
$
\mas{c|c}{\rho^{-1} \Acf & \rho^{-1}\Bcf \hl
\Ccf_1 & \Dcf_1  \\ \Ccf_2 & \Dcf_2 }
$
};

\tio{ki1}{khat}{east}{1/2};
\tio{ko1}{khat}{west}{1/2};
\draw[<-] (olft1)--   ($(olft1) + (-2.8*\ds, 0)$)  |- node[pos=.25]{$\bar{\t y}_0$} (o2) ;
\draw[->] (ilft1)--  ($(ilft1) + (2.8*\ds, 0)$) |- node[pos=.25,swap]{$\bar{\t u}_0$} (i2);

\draw[<-] (ko1)--   ($(ko1) + (-1*\ds, 0)$)  |- node[pos=.25]{$\bar{\t y}$} (olft2) ;
\draw[->] (ki1)--  ($(ki1) + (1*\ds, 0)$) |- node[pos=.25,swap]{$\bar{\t u}$} (ilft2);
\end{tikzpicture}
    \caption{Separation of feedthrough}
    \label{fig:direct_feedthrough_sep}
    \end{subfigure}$\quad$
    \begin{subfigure}{0.48\linewidth}
    \center 
    \begin{tikzpicture}[xscale=1,yscale=1,baseline=(ko1)]
\def\dl{2*\dn}
\def\ds{3*\dn}
\node[sy3] (ghat) at (0,0)  {
$
\setlength{\arraycolsep}{2pt}
\mas{c|c:cc}{
\hat{A}(\blam) &  \hat{B}_1(\blam) &  \hat{B}_2(\blam) \hl   
    \hat{C}_1(\blam) & \hat{D}_1(\blam)   & \hat{D}_{12}(\blam) \hdl
    \hat{C}_2(\blam) & \hat{D}_{21}(\blam)  & 0}
$};
\tio{i1}{ghat}{east}{2/3};
\tio{o1}{ghat}{west}{2/3};
\tio{i2}{ghat}{east}{1/3};
\tio{o2}{ghat}{west}{1/3};
\draw[->] (o1)--  node[swap]{$\bar{r}$} ($(o1) + (-2*\ds, 0)$)  ;
\draw[<-] (i1)-- node[]{$\bar{q}$} ($(i1) + (2*\ds, 0)$) ;

\node[sy3,below=\dl of ghat] (khat) {
$
\mas{c|c}{\rho^{-1} \Acf^0 & \rho^{-1}\Bcf^0 \hl
\Ccf_1^0 & \Dcf_1^0  \\ \Ccf_2^0 & \Dcf_2^0 }
$
};

\tio{ki1}{khat}{east}{1/2};
\tio{ko1}{khat}{west}{1/2};
\draw[<-] (ko1)--   ($(ko1) + (-1.5*\ds, 0)$)  |- node[pos=.25]{$\bar{\t y}_0$} (o2) ;
\draw[->] (ki1)--  ($(ki1) + (1.5*\ds, 0)$) |- node[pos=.25,swap]{$\bar{\t u}_0$} (i2);
\end{tikzpicture}    
    \caption{{Modified configuration for } synthesis}
\label{fig:direct_feedthrough_syn}
    \vfill
    \end{subfigure}
    \vspace{-8mm}
    \caption{Removal of direct feedthrough by interconnection}
    \label{fig:direct_feedthrough}
    \vspace{-5mm}
\end{figure}
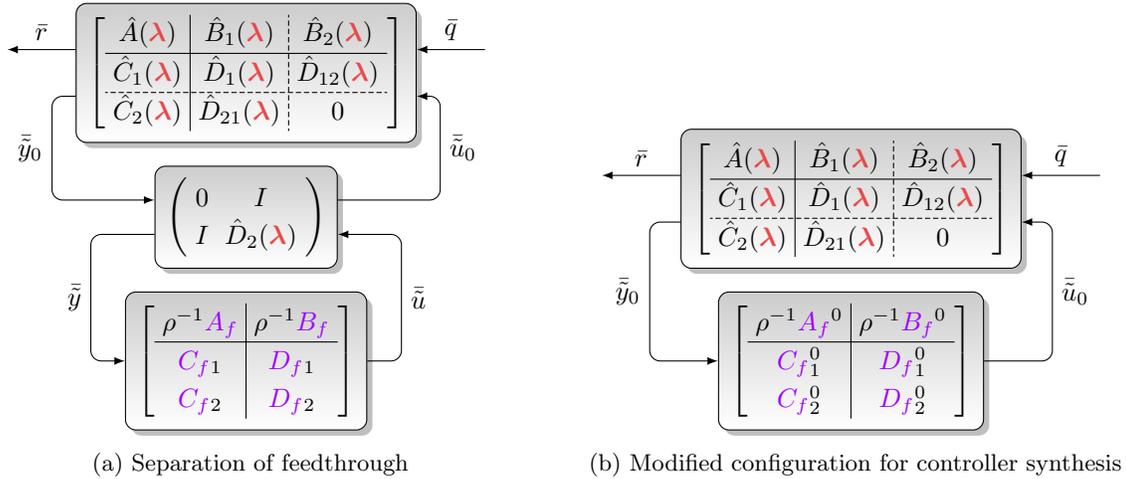

If the {network model and the} algorithm all have Kronecker structure,
then we perform synthesis for dimension $c=c'$ by performing algorithm design for $c=1$, and then applying a Kronecker product 
to all matrices in the controller representation $\Kc$.

\bibliographystyle{siamplain}
\bibliography{reference}
\end{document}